\pdfoutput=1
\documentclass[11pt]{article}
\usepackage{authblk}
\usepackage{extarrows}
\usepackage[top=2cm, bottom=2cm, left=2cm, right=2cm]{geometry}
\usepackage{centernot}

\usepackage{color}
\usepackage{helvet}         
\usepackage{courier}        
\usepackage{type1cm}        

\usepackage{framed}
\usepackage{tikz}
\usepackage{makeidx}         
\usepackage{graphicx}        
\usepackage{multicol}        
\usepackage[bottom]{footmisc}

\usepackage{amsmath}
\usepackage{amssymb}
\usepackage{bbold}
\usepackage{amsthm}
\usepackage{subcaption}
\usepackage{sidecap}
\usepackage{floatrow}
\usepackage{pdflscape}
\usepackage{pdflscape}
\usepackage{comment}
\usepackage[font=small]{caption}
\usepackage{enumitem}
\usepackage{esint}

\usepackage{cancel}

\usepackage{scalerel}

\newtheorem{theorem}{Theorem}

\newtheorem{lemma}[theorem]{Lemma}

\theoremstyle{remark}

\newtheorem{remark}[theorem]{\bf Remark}

\numberwithin{theorem}{section}
\numberwithin{question}{section}
\numberwithin{figure}{section}
\numberwithin{equation}{section}

\allowdisplaybreaks[4]

\begin{document}

\title{The percolation energy field and its logarithmic partner}
\bigskip{}
\author[1]{Federico Camia\thanks{federico.camia@nyu.edu}}
\author[2]{Yu Feng\thanks{yufeng\_proba@163.com. This work was conducted while Y.F. was at Tsinghua University, China.}}
\affil[1]{NYU Abu Dhabi, UAE \& Courant Institute, USA}
\affil[2]{University of Michigan, USA}

\date{}

%

\global\long\def\CR{\mathrm{CR}}
\global\long\def\ST{\mathrm{ST}}
\global\long\def\SF{\mathrm{SF}}
\global\long\def\cov{\mathrm{cov}}
\global\long\def\dist{\mathrm{dist}}
\global\long\def\SLE{\mathrm{SLE}}
\global\long\def\unitD{\mathbb{D}}
\global\long\def\hSLE{\mathrm{hSLE}}
\global\long\def\CLE{\mathrm{CLE}}
\global\long\def\GFF{\mathrm{GFF}}
\global\long\def\inte{\mathrm{int}}
\global\long\def\ext{\mathrm{ext}}
\global\long\def\opi{\overline{\pi}}
\global\long\def\orho{\overline{\rho}}
\global\long\def\inrad{\mathrm{inrad}}
\global\long\def\outrad{\mathrm{outrad}}
\global\long\def\dimH{\mathrm{dim}}
\global\long\def\capa{\mathrm{cap}}
\global\long\def\diam{\mathrm{diam}}
\global\long\def\free{\mathrm{free}}
\global\long\def\Dist{\mathrm{Dist}}
\global\long\def\hF{{}_2\mathrm{F}_1}
\global\long\def\simple{\mathrm{simple}}
\global\long\def\even{\mathrm{even}}
\global\long\def\odd{\mathrm{odd}}
\global\long\def\st{\mathrm{ST}}
\global\long\def\usf{\mathrm{USF}}
\global\long\def\Leb{\mathrm{Leb}}
\global\long\def\LP{\mathrm{LP}}
\global\long\def\coulomb{\LH}
\global\long\def\coulombnew{\LG}
\global\long\def\kfunc{p}
\global\long\def\OO{\mathcal{O}}
\global\long\def\parti{\mathbf{Q}}
\global\long\def\rad{\mathrm{rad}}

\global\long\def\eps{\epsilon}
\global\long\def\ov{\overline}
\global\long\def\U{\mathbb{U}}
\global\long\def\T{\mathbb{T}}
\global\long\def\HH{\mathbb{H}}
\global\long\def\LA{\mathcal{A}}
\global\long\def\LB{\mathcal{B}}
\global\long\def\LC{\mathcal{C}}
\global\long\def\LD{\mathcal{D}}
\global\long\def\LF{\mathcal{F}}
\global\long\def\LK{\mathcal{K}}
\global\long\def\LE{\mathcal{E}}
\global\long\def\LG{\mathcal{G}}
\global\long\def\LI{\mathcal{I}}
\global\long\def\LJ{\mathcal{J}}
\global\long\def\LL{\mathcal{L}}
\global\long\def\LM{\mathcal{M}}
\global\long\def\LN{\mathcal{N}}
\global\long\def\LQ{\mathcal{Q}}
\global\long\def\LR{\mathcal{R}}
\global\long\def\LT{\mathcal{T}}
\global\long\def\LS{\mathcal{S}}
\global\long\def\LU{\mathcal{U}}
\global\long\def\LV{\mathcal{V}}
\global\long\def\LW{\mathcal{W}}
\global\long\def\LX{\mathcal{X}}
\global\long\def\LY{\mathcal{Y}}
\global\long\def\PartF{\mathcal{Z}}
\global\long\def\LH{\mathcal{H}}
\global\long\def\LJ{\mathcal{J}}
\global\long\def\R{\mathbb{R}}
\global\long\def\C{\mathbb{C}}
\global\long\def\N{\mathbb{N}}
\global\long\def\Z{\mathbb{Z}}
\global\long\def\E{\mathbb{E}}
\global\long\def\PP{\mathbb{P}}
\global\long\def\QQ{\mathbb{Q}}
\global\long\def\A{\mathbb{A}}
\global\long\def\one{\mathbb{1}}
\global\long\def\bn{\mathbf{n}}
\global\long\def\MR{MR}
\global\long\def\cond{\,|\,}
\global\long\def\la{\langle}
\global\long\def\ra{\rangle}
\global\long\def\tree{\Upsilon}
\global\long\def\prob{\mathbb{P}}
\global\long\def\hm{\mathrm{Hm}}
\global\long\def\cross{\mathrm{Cross}}

\global\long\def\sf{\mathrm{SF}}
\global\long\def\wr{\varrho}

\global\long\def\Im{\operatorname{Im}}
\global\long\def\Re{\operatorname{Re}}

\global\long\def\ud{\mathrm{d}}
\global\long\def\pder#1{\frac{\partial}{\partial#1}}
\global\long\def\pdder#1{\frac{\partial^{2}}{\partial#1^{2}}}
\global\long\def\der#1{\frac{\ud}{\ud#1}}

\global\long\def\bZnn{\mathbb{Z}_{\geq 0}}

\global\long\def\Vfunc{\LG}
\global\long\def\gfunc{g^{(\rr)}}
\global\long\def\hfunc{h^{(\rr)}}

\global\long\def\SimplexInt{\rho}
\global\long\def\CubeInt{\widetilde{\rho}}

\global\long\def\ii{\mathfrak{i}}
\global\long\def\rr{\mathfrak{r}}
\global\long\def\chamber{\mathfrak{X}}
\global\long\def\Wchamber{\mathfrak{W}}

\global\long\def\SimplexIntKappa8{\SimplexInt}

\global\long\def\nested{\boldsymbol{\underline{\Cap}}}
\global\long\def\unnested{\boldsymbol{\underline{\cap\cap}}}
\global\long\def\unnested{\boldsymbol{\underline{\cap\cap}}}

\global\long\def\acycle{\vartheta}
\global\long\def\bcycle{\tilde{\acycle}}
\global\long\def\Gloop{\Theta}

\global\long\def\metric{\mathrm{dist}}

\global\long\def\adj#1{\mathrm{adj}(#1)}

\global\long\def\bs{\boldsymbol}

\global\long\def\edge#1#2{\langle #1,#2 \rangle}
\global\long\def\graph{G}

\newcommand{\conn}{\vartheta_{\scaleobj{0.7}{\mathrm{RCM}}}}
\newcommand{\hatconn}{\widehat{\vartheta}_{\mathrm{RCM}}}
\newcommand{\realpt}{\smash{\mathring{x}}}
\newcommand{\corrind}{\LC}
\newcommand{\bssymb}{\pi}
\newcommand{\PRCM}{\mu}
\newcommand{\coeff}{p}
\newcommand{\MainConst}{C}

\global\long\def\removeLink{/}
\maketitle

\begin{abstract}
 For site percolation on the triangular lattice, we define two lattice fields that form a logarithmic pair in the sense of conformal field theory. We show that, at the critical point, their two- and three-point correlation functions have well-defined scaling limits, whose structure agrees with that predicted for logarithmic conformal field theories. One of the two fields can be identified with the percolation analog of the Ising energy field, whereas the other is related to the percolation four-arm event.
\end{abstract}





\tableofcontents

\section{Introduction}

\subsection{Background and motivation}
 The large-scale properties of critical two-dimensional percolation are believed to be described by a logarithmic conformal field theory \cite{Gur93}.
Such properties are encoded in the (continuum) scaling limit of the model,
which corresponds to rescaling space and sending the scaling factor to zero.
In the case of site percolation on the triangular lattice, the scaling limit, obtained by sending the lattice spacing to zero, can be analyzed rigorously \cite{Smirnov:Critical_percolation_in_the_plane,CamiaNewmanPercolationFull,CamiaNewmanPercolation}.

For this model, one of us recently showed \cite{Cam23} that connection probabilities (i.e., the probabilities that selected vertices belong to the same cluster), when appropriately rescaled, have a conformally covariant scaling limit.
The same conclusion applies to various combinations of connection probabilities, which can be interpreted as correlation functions of a percolation spin (density) field (see~\cite{Cam23}).
These results allow one to cast scaling limit statements in the language of conformal field theory (CFT) and represent a significant step in the direction of establishing a rigorous connection between the scaling limit of critical percolation and CFT.

In~\cite{CamiaFeng2024logarithmic}, we addressed the question of the logarithmic nature of the percolation CFT.
More precisely, we showed the presence of logarithmic singularities in the scaling limits of the probabilities of various connectivity events, which can be interpreted as CFT correlation functions, including the four-point function of the density field mentioned above.
In~\cite{CamiaFeng2024logarithmic}, the focus is on the mechanism that produces the logarithmic singularities and on the consequences of such singularities for the putative percolation CFT.

In~\cite{CF24}, building on the results and ideas of~\cite{Cam23}, we provided complete and detailed proofs of some of the results of~\cite{CamiaFeng2024logarithmic} and further explored the CFT structure of critical percolation, establishing a connection with fundamental CFT concepts such as those of \emph{operator product expansion} (OPE) and \emph{fusion rules}.

Making this connection rigorous for percolation is particularly interesting because the field of logarithmic CFTs is significantly less developed than that of ordinary CFTs, despite the fact that logarithmic CFTs have attracted considerable attention in recent years due to their role in the study of important physical models and phenomena such as the Wess-Zumino-Witten (WZW) model, the quantum Hall effect, disordered critical systems, self-avoiding polymers, and the Fortuin-Kasteleyn (FK) random-cluster model (see~\cite{creutzig2013logarithmic} for a review).

Here, we continue our analysis of the scaling limit of critical percolation and its CFT structure and complete the rigorous derivation of the results of~\cite{CamiaFeng2024logarithmic}.
In the next section, we provide some more details on the novelty and conceptual relevance of our results, their connection to the existing literature, and future directions.

\subsection{New contributions and future directions}

{In this paper, in the context of site percolation on the triangular lattice, we define a new pair of lattice fields, introduced here for the first time, and study their two- and three-point correlation functions.
One of the fields, which we call the \emph{percolation energy}, is a local lattice field, in the sense that its definition depends only on the lattice structure and lattice spacing. The other field is non-local, in the sense that its definition depends on an additional free parameter. 

The analysis of the scaling limit of the correlation functions of these fields reveals several interesting properties. More precisely, it shows that
\begin{itemize}
    \item to obtain nontrivial results, the normalization of the percolation energy field must include a logarithm of the lattice spacing,
    \item the two-point function of the energy field with itself converges to zero,
    \item the two-point function of the non-local field with itself contains a logarithm of the distance between the points,
    \item the two-point functions of the two fields have the structure predicted in~\cite{Gur93,creutzig2013logarithmic} for pairs of logarithmic partners,
    \item the three-point function of the non-local field with two percolation spin fields contains a logarithm.
\end{itemize}
These features of the scaling limit allow us to identify the non-local field with the \emph{logarithmic partner} of the energy field, and justify its definition.
To the best of our knowledge, this paper provides the first rigorous construction of a pair of CFT logarithmic partner fields and the first rigorous instance of correlation functions of the type predicted in~\cite{Gur93,creutzig2013logarithmic}.
}


A prototypical example of a lattice field that converges to a CFT field, when appropriately rescaled, is the Ising spin (magnetization) field \cite{ChelkakHonglerIzyurovConformalInvarianceCorrelationIsing,CamiaGarbanNewman}. In this and other known cases (such as the Gaussian free field and the percolation spin (density) field \cite{Cam23}), the lattice field is rescaled by a power of the lattice spacing, or by a quantity that scales like a power.
{As mentioned above, an unusual and interesting} aspect of the percolation energy field is that it needs to be rescaled by a factor containing the logarithm of the lattice spacing.
This is related to the fact that, in the scaling limit, the energy-energy two-point function is zero (see \eqref{eqn::energy_two_point} or \eqref{eqn::energy_two_point_new}).
This unusual feature does not signal triviality or independence, but is instead a manifestation of the logarithmic nature of the percolation CFT.

The logarithmic partner of the energy field is also unusual in that its definition requires an additional cutoff parameter on top of the lattice spacing, making it, in some sense, non-local.
This second cutoff is initially kept fixed and is sent to zero only after the lattice spacing.
(A similar situation appears in the definition of the \emph{edge counting operator} in~\cite{CamiaFoitGandolfiKlebanCFTBLS}.)
{We emphasize that the introduction of an additional parameter, independent of the lattice spacing, is not a technical trick but reflects the non-locality of the field.}
An interesting feature of the logarithmic partner of the percolation energy field, {related to its non-locality,} is that its two-point function contains a logarithmic term (see \eqref{eqn::hat_hat_two_point}). {This} cannot happen for \emph{primary} (pure scaling) CFT fields, whose two- and three-point functions are fully determined (up to a multiplicative constant) by conformal covariance and are products of powers.

{We note that our definitions and results are consistent with the findings of Vasseur, Jacobsen and Saleur~\cite{VJS12}, who used both analytical methods and Monte Carlo simulations to argue that certain $Q$-Potts model two-point functions involving the Potts energy field/operator should scale logarithmically in the (singular) limit $Q\to 1$. This observation supports the belief that the (formal) $Q\to 1$ limit of the Potts CFT gives the percolation CFT, despite the fact that the limit leads to singular expressions and therefore does not make proper mathematical sense.}

{The main goal and the main new contribution of this paper is the construction of a logarithmic pair via the scaling limit of a critical lattice model. Being the first mathematically rigorous result of this type, this represents a significant conceptual advancement in the study of the scaling limit of critical percolation and, more generally, in the mathematical analysis of conformal field theory.
In order to identify our percolation energy field and its partner as a logarithmic pair, it is sufficient to analyze certain two- and three-point functions and show that they have the structure predicted in~\cite{Gur93,creutzig2013logarithmic}. This is precisely what is done in this paper.

At the same time, our results suggest other natural questions, which are beyond the scope of this paper but are certainly interesting and can probably be tackled with similar methods.
One such question concerns higher-order correlation functions that involve the logarithmic partner of the energy field. In principle, these can be studied using arguments similar to those used in this paper, but the analysis would be more complicated due to the larger number of points involved. Moreover, higher-order correlation functions would probably not reveal anything fundamentally new, since the logarithmic nature of the CFT is already apparent in the two- and three-point functions calculated here.}

Another natural and interesting open problem is the existence of the scaling limit of the percolation energy field and its logarithmic partner beyond their correlation functions. A positive result in this direction would be analogous to the results of~\cite{CamiaGarbanNewman} for the Ising spin (magnetization) field and of~\cite{Cam23} for the percolation spin (density) field.
The interested reader is referred to Section~1.2 of~\cite{CF24} for some CFT background and to~\cite{GarbanKupiainenEnergyField} for an insightful discussion on the (non-)existence and nature of CFT fields, including the energy field of the Ising model.
The latter, unlike the spin field, has never been shown to exist as a generalized function, despite being well-defined within correlation functions~\cite{HonglerSmirnovIsingEnergy,ChelkakHonglerIzyurovCorrelationsPrimaryFieldsIsing}. This has prompted speculations that the Ising energy field can only be defined within correlation functions (see~\cite{GarbanKupiainenEnergyField}).
Tackling this existence problem in the case of the percolation energy field and its logarithmic partner is likely to require a combination of the tools used in this paper and in~\cite{Cam23}, and possibly new ideas, due to the particular nature of the fields involved.

\subsection{Definitions and main results}
\label{subsec::def_result}
Let $\mathcal{T}$ denote the triangular lattice and let $\mathcal{H}$ denote the hexagonal lattice dual to $\mathcal{T}$. Each vertex of $\mathcal{T}$ corresponds to a face (that is, a hexagon) of $\mathcal{H}$ in a natural way and we often identify them. Assume that $\mathcal{T}$ and $\mathcal{H}$ are embedded in $\mathbb{C}$ in such a way that one of the vertices of $\mathcal{T}$ coincides with the origin of $\mathbb{C}$ (see Figure~\ref{fig::embedding}).
We consider critical site percolation on the rescaled triangular lattice $a\mathcal{T}$, where $a>0$ is a scaling factor and each vertex of $a\mathcal{T}$ is assigned a black or white label independently, with equal probability. We denote the resulting probability measure on percolation configurations by $\mathbb{P}^a$. 
We call a sequence of vertices $(v_1,\ldots,v_{n+1})$ of $a\mathcal{T}$ a black (respectively, white) \emph{path} if $v_1,\ldots,v_{n+1}$ are all black (resp., white) vertices and $v_j \sim v_{j+1}$ for $j=1,\ldots,n$, where $v_{j}\sim v_{j+1}$ denotes that $v_j$ and $v_{j+1}$ are adjacent in $a\mathcal{T}$.

For $z\in \mathbb{C}$ and $\delta>0$, we define $B_\delta(z)=\{w\in \mathbb{C}: |z-w|<\delta\}$ and write
\begin{equation*}
	\pi_a:=\mathbb{P}^a\Big[0\xleftrightarrow{B}\partial B_{1}(0)\Big],
\end{equation*}
where $\{0\xleftrightarrow{B}\partial B_{1}(0)\}$ denotes a \emph{one-arm event}, i.e., the event that there exists a black path connecting $0$ to $\partial B_{1}(0)$. For distinct $z_1,z_2,z_3\in \mathbb{C}$, define 
\begin{align} \label{eqn::def_F}
	F(z_1,z_2,z_3):=\vert z_1-z_2\vert^{-\frac{5}{4}} \vert z_1-z_3\vert^{-\frac{5}{4}}|z_3-z_2|^{\frac{25}{24}}.
\end{align}

\begin{figure}
	\includegraphics[width= 0.3\textwidth]{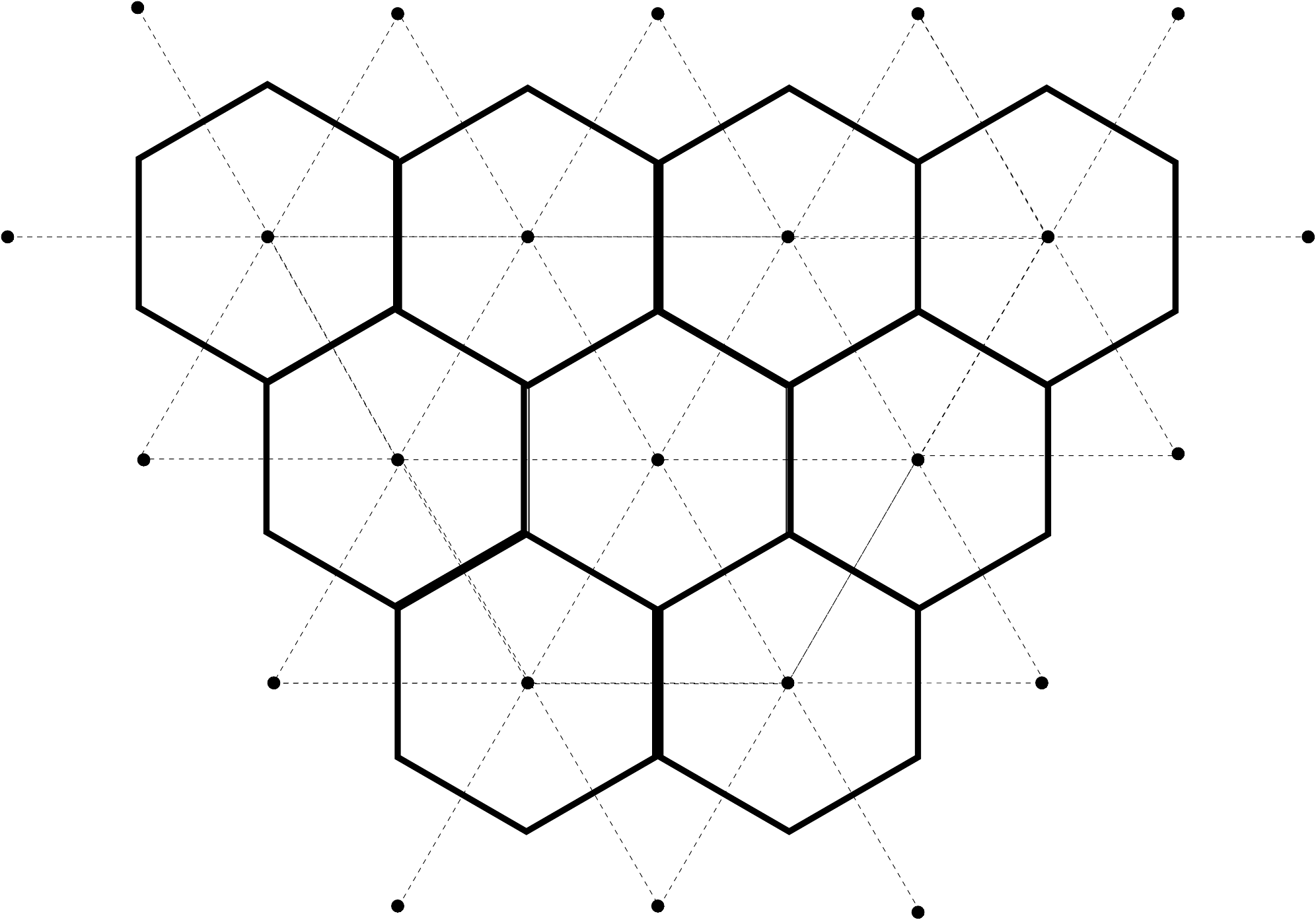}
	\caption{Embedding of the triangular and hexagonal lattices in $\mathbb{C}$.}
	\label{fig::embedding}
\end{figure}

Let $\{\mathcal{C}_j^a\}_j$ denote the collection of black clusters {(maximal connected components of the graph consisting of black vertices, with the adjacency relation $\sim$)} on $a\mathcal{T}$ and assign to each cluster $\mathcal{C}_j^a$ a random sign (spin) $\sigma_j=\pm 1$, where $\{\sigma_j\}_j$ is a collection of symmetric, $(\pm1)$-valued, i.i.d. random variables. Then for each $z^a\in a\mathcal{T}$, we let 
\begin{equation} \label{def:lattice-field}
	S_{z^a}=\begin{cases}
		\sigma_j,& \text{ if } z^a\in \mathcal{C}_j^a,\\
		0, & \text{ if }z^a \text{ is white}. 
	\end{cases}
\end{equation} 
We denote by $\langle \cdot \rangle^a$ the expectation with respect to the distribution of $\{S_{z^a}\}_{z^a\in a\mathcal{T}}$. 

The interest and relevance of the lattice field \eqref{def:lattice-field} stems from the fact that it provides an explicit version of the percolation \emph{density} or \emph{spin field} discussed in the physics literature.
As shown by one of us in the recent work \cite{Cam23}, its correlation functions, when appropriately rescaled, have a conformally covariant scaling limit as does the field itself: suppose that $z_1,\ldots,z_{2n}\in \mathbb{C}$ are distinct points and that $z_1^a,\ldots, z_{2n}^a\in a\mathcal{T}$ satisfy $\lim_{a\to 0}z_j^a=z_j$ for $1\leq j\leq 2n$, then we have 
\begin{align} \label{eqn::cvg_spin_correlation}
	\langle \psi(z_1)\cdots \psi(z_{2n})\rangle:=\lim_{a\to 0} \pi_a^{-2n}\times \langle S_{z_1^a}\cdots S_{z_{2n}^a}\rangle^a\in (0,\infty);
\end{align}
moreover, let $\varphi$ be any non-constant M\"obius transformation such that $\varphi(z_j)\neq \infty$ for $1\leq j\leq 2n$, then we have
\begin{align} \label{eqn::cov_psi}
	\big\langle  \psi\big(\varphi(z_1)\big)\cdots \psi\big(\varphi(z_{2n})\big)\big\rangle= \langle \psi(z_1)\cdots \psi(z_{2n})\rangle\times \prod_{j=1}^{2n}\vert \varphi'(z_j)\vert^{-\frac{5}{48}}. 
\end{align}
In another recent work~\cite{CF24}, we proved that there exist constants $C_1,C_L\in(0,\infty)$ such that, as $z_1,z_2\to z\in \mathbb{C}\setminus \{z_3,z_4\}$, 
\begin{align} \label{eqn::OPE1}
	\begin{split}
		\langle \psi(z_1)\psi(z_2)\psi(z_3)\psi(z_4)\rangle
		= &C_1|z_2-z_1|^{-\frac{5}{24}}\left(|z_4-z_3|^{-\frac{5}{24}}- C_L |z_2-z_1|^{\frac{5}{4}} F(z,z_3,z_4)\log\left|z_2-z_1\right|\right)\\
		&\qquad+o\big(|z_2-z_1|^{\frac{25}{24}}\vert \log|z_2-z_1|\vert\big),
	\end{split}
\end{align}
where $F(z,z_3,z_4)$ is the function defined in~\eqref{eqn::def_F} and the constant $C_1>0$ is such that
\begin{align} \label{eqn::C_1}
		\langle\psi(z_1)\psi(z_2)\rangle=\sqrt{C_1}\vert z_2-z_1 \vert^{-\frac{5}{24}}.
\end{align}

Indeed, the ideas in the proof of~\eqref{eqn::OPE1} given in~\cite{CF24} are sufficient to prove the following strengthening of Theorem~1.1 of~\cite{CF24}.
{\begin{theorem}
   As $z_1,z_2\to z\in \mathbb{C}\setminus \{z_3,z_4\}$,
\begin{align} \label{eqn::OPE2}
	\begin{split}
		\langle \psi(z_1)\psi(z_2)\psi(z_3)\psi(z_4)\rangle =&\, C_1|z_2-z_1|^{-\frac{5}{24}}\Big(|z_4-z_3|^{-\frac{5}{24}}- C_L |z_2-z_1|^{\frac{5}{4}} F(z,z_3,z_4)  \log\left|\tfrac{(z_2-z_1)(z_4-z_3)}{(z_3-z_1)(z_4-z_2)}\right|\\
		&\qquad\qquad\qquad\qquad \qquad+C_0 |z_2-z_1|^{\frac{5}{4}} F(z,z_3,z_4)\Big)+o\big(|z_2-z_1|^{\frac{25}{24}}\big),
	\end{split}
\end{align}
for some $C_0\in \mathbb{R}$.
\end{theorem}}
For completeness, we sketch the proof of~\eqref{eqn::OPE2} in Section~\ref{sec::sketch_four_spin}. 

\begin{remark}
We point out that the constants $C_L$ and $C_0$ are expected to be universal, in the sense that they should not depend on the lattice.
\end{remark}

\paragraph*{The percolation energy field $\phi$.} For each vertex $z^a\in a\mathcal{T}$, we write $z^{(a,-)}$ for $z^a-a$ and $z^{(a,+)}$ for $z^a+a$, and define 
\begin{align} \label{def::energy}
	\mathcal{E}_{z^a}:=S_{z^{(a,-)}}S_{z^{(a,+)}}-\langle S_{z^{(a,-)}}S_{z^{(a,+)}}\rangle^a,
\end{align}
which we call the lattice \emph{percolation energy field}. The name is motivated by the analogy to the energy field of the Ising model.

\begin{remark}
Obviously, we could have defined the energy field using any of the three axes of symmetry of $\mathcal{T}$.
\end{remark}

\begin{theorem} \label{thm::energy_spin}
	Let $z_1,\ldots,z_n\in \mathbb{C}$ be $n$ distinct points. Suppose that $z_1^a,\ldots,z_n^a\in a\mathcal{T}$ satisfy $\lim_{a\to 0}z_j^a=z_j$ for $1\leq j\leq n$. Then there exists a constant $C_2\in (0,\infty)$ such that
\begin{align}
		\langle \phi(z_1)\psi(z_2) \psi(z_3)\rangle :=&\lim_{a\to 0} \big(a^{\frac{5}{4}}|\log a|\big)^{-1} \pi_a^{-2}\times \langle \mathcal{E}_{z_1^a}S_{z_2^a}S_{z_3^a}\rangle^a= C_2 F(z_1,z_2,z_3), \label{eqn::three_point_energy_density}\\
		\langle \phi(z_1)\phi(z_2)\rangle:=&\lim_{a\to 0} \big(a^{\frac{5}{4}}\vert \log a\vert\big)^{-2} \times \langle \mathcal{E}_{z_1^a}\mathcal{E}_{z_2^a}\rangle^a=0, \label{eqn::energy_two_point}
\end{align}
	where $F(z_1,z_2,z_3)$ is the function defined in~\eqref{eqn::def_F}. If $n=\ell+2k$ for some integers $\ell,k\geq 1$, then
	\begin{align}
		\langle \phi(z_1)\cdots\phi(z_\ell)\psi(z_{\ell+1})\ldots \psi(z_{\ell+2k})\rangle:=&\lim_{a\to 0}\big(a^{\frac{5}{4}}\vert \log a\vert\big)^{-\ell} \pi_a^{-2k}\times \langle \mathcal{E}_{z_1^a}\ldots \mathcal{E}_{z_\ell^a}S_{z_{\ell+1}^a}\ldots S_{z_{\ell+2k}^a}\rangle^a\in (0,\infty),\label{eqn::mixed_correlation_spin_energy}\\
		\langle \phi(z_1)\cdots \phi(z_\ell)\rangle:=&\lim_{a\to 0} \big(a^{\frac{5}{4}}\vert\log a\vert\big)^{-\ell}\times \langle \mathcal{E}_{z_1^a}\cdots \mathcal{E}_{z_\ell^a}\rangle^a=0; \label{eqn::energy_n_point}
	\end{align} 
	moreover, let $\varphi$ be any non-constant M\"obius transformation such that $\varphi(z_j)\neq \infty$ for $1\leq j\leq n$, then 
    we have 
	\begin{align}\label{eqn::cov_mixed_correlation_spin_energy}
		\begin{split}
			& \big\langle \phi\big(\varphi(z_1)\big)\cdots \phi\big(\varphi(z_\ell)\big)\psi\big(\varphi(z_{\ell+1})\big)\cdots \psi\big(\varphi(z_{\ell+2k})\big)\big\rangle \\
            & \qquad = \langle \phi(z_1)\cdots\phi(z_\ell)\psi(z_{\ell+1})\cdots \psi(z_{\ell+2k})\rangle \times \Big(\prod_{i=1}^{\ell}\vert \varphi'(z_i)\vert^{-\frac{5}{4}}\Big) \Big(\prod_{j=1}^{2k}\vert \varphi'(z_{\ell+j})\vert^{-\frac{5}{48}}\Big).
		\end{split}
	\end{align}
\end{theorem}

Equation~\eqref{eqn::cvg_spin_correlation} and Theorem~\ref{thm::energy_spin} reveal that the lattice fields $S_{z^a}$ and $\mathcal{E}_{z^a}$ have a natural normalization in terms of the lattice spacing $a$. It is therefore convenient to introduce rescaled versions of those fields:
\begin{align}
    & \psi_a(z^a) := \pi_a^{-1} S_{z^a}, \label{def::renormalized-spin} \\
    & \phi_a(z^a) := \big(a^{5/4}\vert\log a \big\vert)^{-1} \mathcal{E}_{z^a}. \label{def::normalized-energy}
\end{align}
For example, using this notation, equations \eqref{eqn::three_point_energy_density} and \eqref{eqn::energy_two_point} become
\begin{align}
    & \langle \phi(z_1)\psi(z_2) \psi(z_3)\rangle := \lim_{a \to 0} \langle \phi_a(z^a_1) \psi_a(z^a_2) \psi_a(z^a_3)\rangle^a = C_2 F(z_1,z_2,z_3), \label{eqn::three_point_energy_density_new} \\
    & \langle \phi(z_1)\phi(z_2)\rangle := \lim_{a\to 0} \langle \phi_a(z_1^a)\phi_a(z_2^a)\rangle^a=0.
    \label{eqn::energy_two_point_new}
\end{align}
These equations suggest that, in some appropriate sense, the lattice fields $\psi_a$ and $\phi_a$ converge to continuum fields $\psi$ and $\phi$.
The convergence of $\psi_a$ to a continuum field (a generalized function) $\psi$ in an appropriate Sobolev space is proved in~\cite{Cam23}.
The convergence of $\phi_a$ beyond its correlations is an interesting question, but we will not address it here.
For recent results on the convergence of other lattice fields, including the Ising spin and energy fields, the reader is referred to \cite{HonglerSmirnovIsingEnergy,ChelkakHonglerIzyurovConformalInvarianceCorrelationIsing,CamiaGarbanNewman,CamiaGarbanNewmanFieldII,CamiaJiangNewmanCMEReview} and~\cite{GarbanKupiainenEnergyField}.

Equation~\eqref{eqn::energy_two_point_new} may suggest that the field $\phi$ is trivial, but this is clearly not the case. This unusual behavior is due to the logarithm in the normalization of $\phi_a$, which is related to the logarithmic nature of the percolation CFT (see \cite{CF24,CamiaFeng2024logarithmic}).

\begin{remark}
        To simplify the notation, we write $f\lesssim g$ if $f/g$ is bounded by a finite constant from above, and $f \asymp g$ if $f \lesssim g$ and $g \lesssim f$.
        We emphasize that, using the same strategy as in our proof of Theorem~\ref{thm::energy_spin}, one can show that, with the notation of Theorem~\ref{thm::energy_spin},
{		\begin{align*}
			\langle \mathcal{E}_{z_1^a}\mathcal{E}_{z_2^a}\rangle^a \asymp a^{\frac{5}{2}} \vert \log a \vert \quad \text{as }a\to 0. 
		\end{align*}        
        The mechanism that leads to the presence of $\vert\log a \vert$ is shown in Figure~\ref{fig::log_geometry_energy}.
        More generally, for any $\ell \geq 2$,} 
		\begin{align*}
			\langle \mathcal{E}_{z_1^a}\cdots \mathcal{E}_{z_\ell^a}\rangle^a \asymp a^{\frac{5}{4}\ell} \vert \log a \vert^{\ell-1} \quad \text{as }a\to 0,
		\end{align*}
        and this asymptotic behavior is the reason for~\eqref{eqn::energy_n_point}.
\end{remark}

\paragraph*{The logarithmic partner $\hat{\phi}$ of the energy field $\phi$.} Let $\delta>0$ and  assume that $\delta$ is not a multiple of $a$.
For $z^a\in a\mathcal{T}$, let $z^{(a,\delta,-)}$ (resp., $z^{(a,\delta,+)}$) be the vertex in $a\mathcal{T}$ closest to $z^a-\delta$ (resp., $z^a+\delta$), and define 
\begin{align} \label{def::Edelta}
\mathcal{E}^{\delta}_{z^a}:=S_{z^{(a,\delta,-)}}S_{z^{(a,\delta,+)}}-\langle S_{z^{(a,\delta,-)}}S_{z^{(a,\delta,+)}}\rangle^a.
\end{align}
This is a non-local version of the energy operator defined above.

For $z\in \mathbb{C}$, we also write $z^{(\delta,-)}$ for $z-\delta$ and $z^{(\delta,+)}$ for $z+\delta$. 
%
Note that, for $z_1^a,z_2^a,z_3^a\in a\mathcal{T}$ and $\delta_1,\delta_2>0$, we have 
\begin{align*}
	\begin{split}
		\langle \mathcal{E}_{z_1^a}^{\delta_1} S_{z_2^a}S_{z_3^a}\rangle^a=& \langle S_{z_1^{(a,\delta_1,-)}}S_{z_1^{(a,\delta_1,+)}}S_{z_2^a}S_{z_3^a}\rangle^a- \langle S_{z_1^{(a,\delta_1,-)}}S_{z_1^{(a,\delta_1,+)}}\rangle^a\langle S_{z_2^a}S_{z_3^a}\rangle ^a,\\
		\langle \mathcal{E}_{z_1^a}^{\delta_1}\mathcal{E}_{z_2^a}\rangle^a= &\langle S_{z_1^{(a,\delta_1,-)}} S_{z_1^{(a,\delta_1,+)}}\mathcal{E}_{z_2^a} \rangle^a,\\
		\langle \mathcal{E}^{\delta_1}_{z_1^a}\mathcal{E}^{\delta_2}_{z_2^a}\rangle^a=& \langle S_{z_1^{(a,\delta_1,-)}} S_{z_1^{(a,\delta_1,+)}}
S_{z_2^{(a,\delta_2,-)}} S_{z_2^{(a,\delta_2,+)}}\rangle^a-\langle S_{z_1^{(a,\delta_1,-)}} S_{z_1^{(a,\delta_1,+)}}
\rangle^a\langle
S_{z_2^{(a,\delta_2,-)}} S_{z_2^{(a,\delta_2,+)}}\rangle^a.
	\end{split}
\end{align*}

Assume that $z_1,z_2,z_3\in \mathbb{C}$ are distinct points and that $z_j^a\in a\mathcal{T}$ satisfy $\lim_{a\to 0} z_j^a=z_j$ for $1\leq j\leq 3$. 
Then, according to~\eqref{eqn::cvg_spin_correlation} and~\eqref{eqn::three_point_energy_density}, we have 
\begin{align} 
&\lim_{a\to 0}\pi_a^{-4} \times \langle \mathcal{E}^{\delta}_{z_1^a}S_{z_2^a}S_{z_3^a}\rangle^a = \langle \psi(z_1^{(\delta,-)})\psi(z_1^{(\delta,+)})\psi(z_2)\psi(z_3)\rangle -\langle \psi(z_1^{(\delta,-)})\psi(z_1^{(\delta,+)})\rangle \langle\psi(z_2)\psi(z_3)\rangle, \label{eq::Eps_delta_S_S} \\
&\lim_{a\to 0} a^{-\frac{5}{4}}|\log a|^{-1} \pi_a^{-2} \langle \mathcal{E}^{\delta}_{z_1^a}\mathcal{E}_{z_2^a}\rangle^a = \langle \psi(z_1^{(\delta,-)})\psi(z_1^{(\delta,+)})\phi(z_2)\rangle=C_2F(z_2,z_1^{(\delta,-)},z_1^{(\delta,+)}), \label{eq::Eps_delta_Eps} \\
&\lim_{a\to 0}\pi_a^{-4}\times \langle \mathcal{E}_{z_1^a}^{\delta_1}\mathcal{E}^{\delta_2}_{z_2^a}\rangle^a \nonumber \\
& \qquad = \langle \psi (z_1^{(\delta_1,-)}) \psi(z_1^{(\delta_1,+)})\psi(z_2^{(\delta_2,-)})\psi(z_2^{(\delta_2,+)})\rangle -\langle \psi(z_1^{(\delta_1,-)})\psi(z_1^{(\delta_1,+)})\rangle 
\langle \psi(z_2^{(\delta_2,-)})\psi(z_2^{(\delta_2,+)})\rangle. \label{eq::Eps_delta_Eps_delta}	
\end{align}


Now let
\begin{align} \label{def::Edelta-renormalized}
    \eta^{\delta}_a(z^a) := \pi_a^{-2} \mathcal{E}^{\delta}_{z^a}.
\end{align}
and define the lattice \emph{logarithmic partner} of the energy field to be the following ``mixture'' of $\phi_a$ and $\eta^{\delta}_a$:
\begin{align} \label{def::log_partner}
    \hat{\phi}^{\delta}_a(z^a) := (2\delta)^{-\frac{25}{24}} \eta^{\delta}_a(z^a) + \frac{C_1C_L}{C_2}\log(2\delta)\phi_a(z^a).
\end{align}
{ We remark that the combination $\frac{C_1C_L}{C_2}$ is chosen for convenience and could be replaced by any other constant, as explained in Remark~\ref{rem::log_partner} below.}

\begin{theorem} \label{thm::mixed}
We have 
    \begin{align}
		& \langle \hat{\phi} (z_1)\psi (z_2)\psi(z_3)\rangle := \lim_{\delta\to 0}\lim_{a \to 0} \langle \hat{\phi}^{\delta}_a(z^a_1)\psi_a(z^a_2)\psi_a(z^a_3)\rangle^a = C_1\left(C_0-C_L \log \frac{|z_3-z_2|}{|z_3-z_1||z_2-z_1|}\right)F(z_1,z_2,z_3), 	\label{eqn::hat_psi_three_point}\\
		& \langle \hat{\phi}(z_1)\phi(z_2)\rangle := \lim_{\delta\to 0}\lim_{a \to 0}\langle \hat{\phi}^{\delta}_a(z^a_1)\phi_a(z^a_2)\rangle^a=C_2 |z_2-z_1|^{-\frac{5}{2}}, \label{eqn::hat_energy_two_point}\\
		& \langle \hat{\phi}(z_1)\hat{\phi}(z_2)\rangle := \lim_{\delta_1,\delta_2\to 0}\lim_{a \to 0} \langle \hat{\phi}^{\delta_1}_a(z^a_1)\hat{\phi}^{\delta_2}_a(z^a_2)\rangle^a = C_1 \frac{C_0+2C_L\log |z_2-z_1|}{|z_2-z_1|^{\frac{5}{2}}}, \label{eqn::hat_hat_two_point}
	\end{align}
where $C_1$, $C_0$ and $C_L$ are the constants in~\eqref{eqn::C_1} and~\eqref{eqn::OPE2}, and $C_2$ is the constant in~\eqref{eqn::three_point_energy_density}.  
\end{theorem}
\begin{proof}
Equation~\eqref{eqn::hat_psi_three_point} follows from~\eqref{eqn::C_1},~\eqref{eqn::OPE2},~\eqref{eqn::three_point_energy_density} and~\eqref{eq::Eps_delta_S_S}, while~\eqref{eqn::hat_energy_two_point} is a direct consequence of~\eqref{eqn::energy_two_point} and~\eqref{eq::Eps_delta_Eps}.
It remains to prove~ \eqref{eqn::hat_hat_two_point}.

Using the strategy in the proof of~\eqref{eqn::OPE2}, which is sketched in Section~\ref{sec::sketch_four_spin}, one can show that there exist constants $C_L^*\in (0,\infty)$ and $C_0^{*}\in \mathbb{R}$ such that, as $\delta_1,\delta_2\to 0$, 
\begin{align} \label{eqn::OPE3}
\begin{split}
		&	\big\langle\psi\big(z_1^{(\delta_1,-)}\big)\psi\big(z_1^{(\delta_1,+)}\big) \psi\big(z_2^{(\delta_2,-)}\big)\psi\big(z_2^{(\delta_2,+)}\big)\big\rangle\\
			&\qquad= C_1 \big\vert z_1^{(\delta_1,+)}-z_1^{(\delta_1,-)}\big\vert^{-\frac{5}{24}} \big\vert z_2^{(\delta_2,+)}-z_2^{(\delta_2,-)}\big\vert^{-\frac{5}{24}}\\
	&\qquad	\quad-C_1C_L^* \big\vert z_1^{(\delta_1,+)}-z_1^{(\delta_1,-)}\big\vert^{\frac{25}{24}}F\Big(z_1,z_2^{(\delta_2,+)},z_2^{(\delta_2,-)}\Big) \log \Big\vert\tfrac{\big(z_1^{(\delta_1,+)}-z_1^{(\delta_1,-)}\big)\big(z_2^{(\delta_2,+)}-z_2^{(\delta_2,-)}\big)}{\big(z_2^{(\delta_2,+)}-z_1^{(\delta_1,+)}\big)\big(z_2^{(\delta_2,-)}-z_1^{(\delta_1,-)}\big)}\Big\vert\\
	&\qquad\quad+C_1C_0^*|z_1^{(\delta_1,+)}-z_1^{(\delta_1,-)}|^{\frac{25}{24}}F\Big(z_1,z_2^{(\delta_2,+)},z_2^{(\delta_2,-)}\Big)\\
	&\qquad\quad+o\Big(\big\vert z_1^{(\delta_1,+)}-z_1^{(\delta_1,-)}\big\vert^{\frac{25}{24}}\big\vert z_2^{(\delta_2,+)}-z_2^{(\delta_2,-)}\big\vert^{\frac{25}{24}}\Big).
\end{split}
\end{align}
Compared with the proof of~\eqref{eqn::OPE2}, the main difference here is that we have to use two collections of disks, $B_m^{(1)}:=\{w: |w-z_1|\leq 2^m(2\delta_1)\}$ for $1\leq m\leq  M_1$ and $B_m^{(2)}:=\{w: |w-z_2|\leq 2^m(2\delta_2)\}$ for $1\leq m\leq M_2$, instead of one, where $M_1\sim -\log \delta_1$ and $M_2\sim -\log \delta_2$.

Comparing~\eqref{eqn::OPE3} with~\eqref{eqn::OPE2}, we obtain
\begin{align}
C_L^*=C_L,\quad \text{and}\quad C_0^*=C_0.
\end{align}
Combining these two observations with~\eqref{eqn::C_1},~\eqref{eqn::energy_two_point},~\eqref{eq::Eps_delta_Eps} and~\eqref{eq::Eps_delta_Eps_delta} gives~\eqref{eqn::hat_hat_two_point}, as desired. 
\end{proof}

\begin{remark}
    Theorems~\ref{thm::energy_spin} and~\ref{thm::mixed} show that the correlations of the energy field and its logarithmic partner have the structure expected of the correlations of a \emph{logarithmic pair} in a logarithmic CFT (see~\cite{Gur93} and~\cite{creutzig2013logarithmic}). {From an algebraic perspective, a primary field is an eigenvector of the dilation operator corresponding to the highest-weight vector of a representation of the Virasoro algebra, and its two- and three-point functions are fixed by conformal invariance, up to multiplicative constants. In the case of a logarithmic CFT, there are highest-weight representations where the dilation operator is not diagonalizable and, therefore, when expressed as a matrix in Jordan form, contains Jordan blocks. In such a situation, an eigenvector, corresponding to a primary field, is accompanied by one or more Jordan partners. In the simplest case of a single Jordan partner, it is possible to check that conformal invariance fixes the functional form of the two-point functions of the primary field and its partner to be as in Theorems~\ref{thm::energy_spin} and~\ref{thm::mixed} (see, for example, Section~1.1 of~\cite{creutzig2013logarithmic}). The appearance of a logarithm in~\eqref{eqn::hat_hat_two_point} explains why the Jordan partner of the primary field is called logarithmic partner.}
\end{remark}

\begin{remark} \label{rem::log_partner}
    {Because of~\eqref{eqn::energy_two_point}, the functional form of the two-point functions~\eqref{eqn::hat_energy_two_point} and~\eqref{eqn::hat_hat_two_point} would not change if we changed the constant $\frac{C_1C_L}{C_0}$ in~\eqref{def::log_partner}. Therefore, the logarithmic partner is not unique: it is defined up to a constant that determines the amount of ``mixing'' with the energy field (see, for example, the discussion after Eq.~(1.10) of~\cite{creutzig2013logarithmic}). Our choice in~\eqref{def::log_partner} is one of convenience, since it gives a simple expression for the three-point function in~\eqref{eqn::hat_psi_three_point}.}
\end{remark}

For the reader's convenience, we end this section with a heuristic discussion of the meaning of the fields introduced above.
The density or spin field defined in~\eqref{def:lattice-field} and~\eqref{def::renormalized-spin} corresponds to the presence of a one-arm event at a particular location. In CFT parlance, one would say that a one-arm event is ``inserted'' at a point in space.
This explains the factor $\pi_a^{-1}$ in the renormalized version~\eqref{def::renormalized-spin} of the field.
To understand this, consider the $n$-point function $\langle\psi_a(z^a_1)\ldots\psi_a(z^a_n)\rangle^a$.
If the cluster of at least one of the hexagons $z^a_i$ contains no other $z^a_j$, then by symmetry, averaging over the values of $S_{z^a_i}$ gives $0$ (and $n$ must be even).
Hence, for the $n$-point function to be nonzero, each hexagon $z^a_i$ must be connected to at least one other hexagon $z^a_j$, producing a one-arm event at each hexagon involved in the $n$-point function.

In the case of an $n$-point function containing the energy field (see~\eqref{def::energy} and~\eqref{def::normalized-energy}), the contribution coming from configurations in which $z^{(a,-)}$ and $z^{(a,+)}$ are connected by a ``microscopic'' path in the scaling limit, which would produce a divergence, is canceled by removing the mean. The effect of this is that the configurations that contribute to the $n$-point function contain a four-arm event at $z^a$, which explains the factor $a^{-5/4}$ in~\eqref{def::normalized-energy}, where $\frac{5}{4}$ is the four-arm exponent~\cite{SmirnovWernerCriticalExponents}.
The presence of the logarithm in the normalization factor is less intuitive and more complex; it can be explained, as shown in the proof of~\eqref{eqn::three_point_energy_density} given in Section~\ref{subsec::mixed_cor} below, by the same mechanism identified in~\cite{CamiaFeng2024logarithmic,CF24} to explain the logarithmic divergence in the four-point function $\langle \psi(z_1)\psi(z_2)\psi(z_3)\psi(z_4)\rangle$ (see~\eqref{eqn::OPE1}).

The lattice field defined in~\eqref{def::Edelta} (or its renormalized version~\eqref{def::Edelta-renormalized}) is a non-local version of the energy field. The additional cutoff $\delta>0$ is at first kept fixed while $a \to 0$. This corresponds to the ``insertion'' of two spin fields, or equivalently two one-arm events, as explained above, at $z^{(a,\delta,-)}$ and $z^{(a,\delta,+)}$, which explains the renormalization factor $\pi_a^{-2}$ in~\eqref{def::Edelta-renormalized}. 

The factor $(2\delta)^{-\frac{25}{24}}$ in the normalization of $\hat\phi^{\delta}_a$ in front of $\eta^{\delta}_a$ can be understood by observing that $\langle \psi(z^{(\delta,-)})\psi(z^{(\delta,+)})\rangle \sim (2\delta)^{-\frac{5}{24}}$, while the ``insertion'' of two separate one-arm events at $z^{(\delta,-)}$ and $z^{(\delta,+)}$ produces a four-arm event whose probability goes to zero like $(2\delta)^{\frac{5}{4}}$ as $\delta \to 0$. 
The additional term containing the energy field and the logarithm of $2\delta$ is needed to obtain the correlations in Theorem~\ref{thm::mixed} and is related to the asymptotic expansion~\eqref{eqn::OPE2}.
In other words, while the ``mixing'' of $\eta^{\delta}_a$ and $\phi_a$ in~\eqref{def::log_partner}, with the corresponding coefficients, may appear arbitrary and complicated, it is natural given the expansion~\eqref{eqn::OPE2} and is precisely what is needed to yield the expected structure of mixed correlation functions between a field (the energy field, in this case) and its logarithmic partner (see, e.g., \cite{creutzig2013logarithmic}).

\section{Preliminary results}
\label{sec::sketch_four_spin}

\subsection{Percolation interfaces and their scaling limit}
We briefly recall the definition of percolation interfaces, which are curves that separate black and white clusters. 
For critical site percolation on the triangular lattice, {building on Smirnov's seminal paper~\cite{Smirnov:Critical_percolation_in_the_plane}, the full collection of such interfaces was shown in~\cite{CamiaNewmanPercolationFull}} to have a weak limit as the lattice spacing is sent to zero.

We let $a>0$ and consider critical Bernoulli site percolation on $a\mathcal{T}$. Given a percolation configuration, the percolation interfaces between black and white clusters are polygonal circuits (with probability one) composed of edges of the hexagonal lattice $a\mathcal{H}$ dual to the triangular lattice $a\mathcal{T}$. We view these circuits as curves in the complex plane $\mathbb{C}$ and give them an orientation in such a way that they wind counterclockwise around black clusters and clockwise around white clusters (in other words, they are oriented in such a way that black hexagons are on the left and white hexagons are on the right). Note that the interfaces form a nested collection of loops with alternating orientations and a natural tree structure. 

In order to state the weak convergence of the collection of percolation interfaces, we need to specify a topology on the space of collections of loops. First, we introduce a distance function $\Delta$ on $\mathbb{C}\times \mathbb{C}$,
\begin{equation*}
	\Delta(u,v):=\inf_{f} \int_0^1 \frac{|f'(t)|}{1+|f(t)|^2}\ud t,
\end{equation*}
where the infimum is over all differentiable curves $f: [0,1]\to \mathbb{C}$ with $f(0)=u$ and $f(1)=v$. 
Second, for two planar oriented curves $\gamma_1,\gamma_2:[0,1]\to \mathbb{C}$, we define
\begin{equation}  \label{eqn::curve_metric}
	\dist\left(\gamma_1,\gamma_2\right):=\inf_{\psi,\tilde{\psi}}\sup_{t\in[0,1]} \Delta\left( \gamma_1(\psi(t)), \gamma_2(\tilde{\psi}(t))\right),
\end{equation}
where the infimum is taken over all increasing homeomorphisms $\psi,\tilde{\psi}:[0,1]\to [0,1]$. Note that planar oriented loops can be viewed as planar oriented curves. Third, we define a distance between two closed sets of loops, $\Gamma_1$ and $\Gamma_2$, as follows:
\begin{equation} \label{eqn::loop_metric}
	\mathrm{Dist}(\Gamma_1,\Gamma_2):=\inf\{\epsilon>0: \forall \gamma_1\in \Gamma_1 \enspace \exists \gamma_2\in \Gamma_2 \text{ such that }\dist(\gamma_1,\gamma_2)\leq \epsilon \text{ and vice versa}\}. 
\end{equation}
The space $X$ of collections of loops with this distance is a separable metric space.

It was shown in~\cite{CamiaNewmanPercolationFull} that, as $a\to 0$, the collection of percolation interfaces has a unique limit in distribution in the topology induced by~\eqref{eqn::loop_metric}. We denote this limit by $\Lambda$ and call it the \textit{full scaling limit of percolation}.
As explained in~\cite{SLECLE6}, $\Lambda$ is distributed like the full-plane, nested $\mathrm{CLE}_6$. It is invariant, in a distributional sense, under all non-constant M\"obius transformations~\cite{CamiaNewmanPercolationFull,GwynneMillerQianCLE}.

For $z\in \mathbb{C}$ and $0<r<R$, we denote by $\mathcal{F}_{r,R}(z)$ the event that there are four {monochromatic paths that cross the annulus $B_{R}(z)\setminus B_{r}(z)$ producing an alternating pattern (black/white/black/white).} One can express the event $\mathcal{F}_{r,R}(z)$ in terms of percolation interface loops. More precisely, $\mathcal{F}_{r,R}(z)$ means that there are four distinct segments of interface loops, with alternating orientations, crossing the annulus $B_{R}(z)\setminus B_{r}(z)$. Using the loop definition of $\mathcal{F}_{r,R}(z)$ and the percolation full scaling limit $\Lambda$ in terms of interface loops given in~\cite{CamiaNewmanPercolationFull}, we can define the analog of $\mathcal{F}_{r,R}(z)$ in the continuum, which we still denote by $\mathcal{F}_{r,R}(z)$. Moreover, since the polychromatic boundary $3$-arm exponent for critical site percolation is strictly larger than $1$~\cite{SmirnovWernerCriticalExponents}, the event $\mathcal{F}_{r,R}(z)$ is a continuity event for the law of $\Lambda$.

\subsection{The four-point function of the density field $\psi$}

In~\cite{CF24}, the asymptotic expansion~\eqref{eqn::OPE1} is proved, but the same techniques can be used to prove the slightly stronger version~\eqref{eqn::OPE2}, which is used in the proof of Theorem~\ref{thm::mixed}. Since the ideas needed to prove~\eqref{eqn::OPE2} will also be used in the proof of Theorem~\ref{thm::energy_spin}, below we provide a sketch of the proof of~\eqref{eqn::OPE2}. For details of some technical steps, we refer to~\cite[Section~3]{CF24}.

Let $z_1,z_2,z_3,z_4\in \mathbb{C}$ be distinct points and assume that $z_j^a\in a\mathcal{T}$ satisfy $\lim_{a\to 0}z_j^a=z_j$ for $1\leq j\leq 4$. A simple calculation shows that 
\begin{align} \label{eq::4pf-connection_probs}
		\langle S_{z_1^a}\cdots S_{z_4^a}\rangle^a=&\mathbb{P}^a\Big[z_1^a\xleftrightarrow{B}z_2^a,\; z_3^a\xleftrightarrow{B}z_4^a\Big]+\mathbb{P}^a\Big[z_1^a\xleftrightarrow{B}z_3^a\centernot{\xleftrightarrow{B}}z_2^a\xleftrightarrow{B} z_4^a\Big]+\mathbb{P}^a\Big[z_1^a\xleftrightarrow{B}z_4^a\centernot{\xleftrightarrow{B}}z_2^a\xleftrightarrow{B}z_3^a\Big], 
\end{align}
where $\{ z_i^a\xleftrightarrow{B}z_j^a \}$ denotes the event that $z^a_i$ and $z^a_j$ belong to the same black cluster and $\{ z_i^a\centernot{\xleftrightarrow{B}}z_j^a \}$ denotes its complement.
Then, one can write
\begin{align}\label{eqn::spin_four_connec_discrete}
\begin{split}
	& \langle \psi(z_1)\cdots \psi(z_4)\rangle=\lim_{a\to 0}\pi_a^{-4}\times \langle S_{z_1^a}\cdots S_{z_4^a}\rangle^a\\
	& \quad = P\Big[z_1\xleftrightarrow{B}z_3,z_2\xleftrightarrow{B}z_4\Big]+ P\Big[z_1\xleftrightarrow{B}z_3\centernot{\xleftrightarrow{B}}z_2\xleftrightarrow{B}z_4\Big]+ P\Big[z_1\xleftrightarrow{B}z_4\centernot{\xleftrightarrow{B}} z_2\xleftrightarrow{B}z_3\Big],
\end{split}
\end{align}
where the existence of the limits
\begin{align*}    P\Big[z_1\xleftrightarrow{B}z_2,z_3\xleftrightarrow{B}z_4\Big]:=&\lim_{a\to 0} \pi_a^{-4}\times \mathbb{P}^a\Big[z_1^a\xleftrightarrow{B} z_2^a,\;z_3^a\xleftrightarrow{B}z_4^a\Big],\\
P\Big[z_1\xleftrightarrow{B}z_3\centernot{\xleftrightarrow{B}}z_2\xleftrightarrow{B}z_4\Big]:=&\lim_{a\to 0}\pi_a^{-4}\times\mathbb{P}^a\Big[z_1^a\xleftrightarrow{B}z_3^a\centernot{\xleftrightarrow{B}}z_2^a\xleftrightarrow{B} z_4^a\Big],\\
P\Big[z_1\xleftrightarrow{B}z_4\centernot{\xleftrightarrow{B}} z_2\xleftrightarrow{B}z_3\Big]:=&\lim_{a\to 0} \pi_a^{-4}\times \mathbb{P}^a\Big[z_1^a\xleftrightarrow{B}z_4^a\centernot{\xleftrightarrow{B}}z_2^a\xleftrightarrow{B}z_3^a\Big],
\end{align*}
was proved in~\cite[Proof of Theorem~1.5]{Cam23}.


We fix $z_3,z_4$ and $z\in \mathbb{C}\setminus \{z_3,z_4\}$. We will eventually let $z_1,z_2\to z$. Choose $\epsilon>0$ such that $B_{2\epsilon}(z)\cap \{z_3,z_4\}=\emptyset$.
Note that, when $|z_2-z_1|$ is small enough, the event $\{z_1\xleftrightarrow{B}z_3\centernot{\xleftrightarrow{B}}z_2\xleftrightarrow{B}z_4\}$ implies that there are four paths with alternating labels (black/white) crossing the annulus $B_{\epsilon}(z)\setminus B_{|z_2-z_1|}(z)$.
Since the four-arm exponent equals $\frac{5}{4}$ and the one-arm exponent equals $\frac{5}{48}$~\cite{SmirnovWernerCriticalExponents}, one can show that 
\begin{align} \label{eq::13_24}
P\Big[z_1\xleftrightarrow{B}z_3\centernot{\xleftrightarrow{B}}z_2\xleftrightarrow{B}z_4\Big] \asymp |z_2-z_1|^{-\frac{5}{24}+\frac{5}{4}},\quad \text{as }z_1,z_2\to z.
\end{align}
This can be understood, for example, using (2.47) of \cite{Cam23} with $\varepsilon$ of the order of $\vert z_2-z_1 \vert$ and $G$ representing a four-arm event, or using the arguments sketched in~\cite[Section~2.3]{Cam23}. Intuitively, the exponent $\frac{5}{4}$ arises because the event
\[
\{z_1\xleftrightarrow{B}z_3\centernot{\xleftrightarrow{B}}z_2\xleftrightarrow{B}z_4\}
\]
implies the occurrence of a four-arm event centered at $\tfrac{z_1+z_2}{2}$, whereas the exponent $\frac{5}{24}$ comes from the normalization by two one-arm events, centered at $z_1$ and $z_2$, respectively.
In order to improve~\eqref{eqn::OPE1} and obtain~\eqref{eqn::OPE2}, we need to strengthen~\cite[Lemma~3.2]{CF24} and show that
\begin{align} \label{eqn::four_spin_unpair}
	\lim_{z_1,z_2\to z}|z_2-z_1|^{-\frac{25}{24}}\times 	P\Big[z_1\xleftrightarrow{B}z_3\centernot{\xleftrightarrow{B}}z_2\xleftrightarrow{B}z_4\Big] \in (0,\infty).
\end{align}
To that end, for any $\epsilon>0$ such that $B_{\epsilon}(\frac{z_1+z_2}{2})\cap\{ z_3,z_4 \}=\emptyset$, we define
{\small\begin{align*}
P\Big[\mathcal{F}_{2|z_1-z_2|,\epsilon}\Big(\frac{z_1+z_2}{2}\Big)\Big]:=&\,\lim_{a\to 0}\mathbb{P}^a\Big[\mathcal{F}_{2|z_1^a-z_2^a|,\epsilon}\Big(\frac{z_1^a+z_2^a}{2}\Big)\Big],\\
P\Big[z_1\xleftrightarrow{B}z_3\centernot{\xleftrightarrow{B}}z_2\xleftrightarrow{B}z_4 \big| \mathcal{F}_{2| z_1-z_2 |,\epsilon}\Big(\frac{z_1+z_2}{2}\Big)\Big]:=&\,\lim_{a\to 0} \pi_a^{-4}\times\mathbb{P}^a\Big[z_1^a\xleftrightarrow{B}z_3^a\centernot{\xleftrightarrow{B}}z_2^a\xleftrightarrow{B}z_4^a \big| \mathcal{F}_{2| z_1^a-z_2^a |,\epsilon}\Big(\frac{z_1^a+z_2^a}{2}\Big)\Big]\\
	=& \, P\Big[z_1\xleftrightarrow{B}z_3\centernot{\xleftrightarrow{B}}z_2\xleftrightarrow{B}z_4\Big]/ P\Big[\mathcal{F}_{2|z_1-z_2|,\epsilon}\Big(\frac{z_1+z_2}{2}\Big)\Big],
\end{align*}}where the existence of the first limit follows directly from the convergence of the collection of percolation interfaces~\cite{CamiaNewmanPercolationFull}. Then we can write 
\begin{align}
\begin{split} \label{eqn::four_spin_unpair_aux}
	& \lim_{z_1,z_2\to z}|z_2-z_1|^{-\frac{25}{24}} \times P\Big[z_1\xleftrightarrow{B}z_3\centernot{\xleftrightarrow{B}}z_2\xleftrightarrow{B}z_4\Big] \\
    & \quad = \lim_{z_1,z_2\to z}|z_2-z_1|^{\frac{5}{24}} \times P\Big[z_1\xleftrightarrow{B}z_3\centernot{\xleftrightarrow{B}}z_2\xleftrightarrow{B}z_4 \vert \mathcal{F}_{2| z_1-z_2 |,\epsilon}\Big(\frac{z_1+z_2}{2}\Big)\Big] \\
    & \qquad \times\lim_{z_1,z_2\to z}|z_2-z_1|^{-\frac{5}{4}} \times P\Big[ \mathcal{F}_{2| z_1-z_2 |,\epsilon}\Big(\frac{z_1+z_2}{2}\Big) \Big],
\end{split}
\end{align}
where the existence of the last limit follows from~\cite[Corollary~2.4]{CF24}.
The existence of the other limit on the right-hand side of~\eqref{eqn::four_spin_unpair_aux} can be derived as follows. 

\begin{lemma}\label{lem::four_arm_coupling_argument}
  Let $\{z_1^{(n)}\}_{n\geq 1}$ and $\{z_2^{(n)}\}_{n\geq 1}$ be two sequences of complex numbers satisfying $z_1^{(n)}\neq z_2^{(n)}$ and $z_1^{(n)}\to z$, $z_2^{(n)}\to z$ as $n\to \infty$. Then
    	\begin{align*}
		\left\{|z_2^{(n)}-z_1^{(n)}|^{\frac{5}{24}}\times P\Big[z^{(n)}_1\xleftrightarrow{B}z_3\centernot{\xleftrightarrow{B}}z^{(n)}_2\xleftrightarrow{B}z_4 \vert \mathcal{F}_{2| z^{(n)}_1-z^{(n)}_2 |,\epsilon}\Big(\frac{z^{(n)}_1+z^{(n)}_2}{2}\Big)\Big]\right\}_{n=1}^{\infty}
	\end{align*}
    is a Cauchy sequence.
\end{lemma}
\noindent Lemma~\ref{lem::four_arm_coupling_argument} can be proved by adapting the proof of~\cite[Lemma~2.3]{CF24}. Since this argument will be used repeatedly later in the article, we provide a proof of Lemma~\ref{lem::four_arm_coupling_argument} in the next section.
 
Lemma~\ref{lem::four_arm_coupling_argument} gives the existence of the limit
\begin{equation*}
\lim_{z_1,z_2\to z}|z_2-z_1|^{\frac{5}{24}} \times P\Big[z_1\xleftrightarrow{B}z_3\centernot{\xleftrightarrow{B}}z_2\xleftrightarrow{B}z_4 \vert \mathcal{F}_{2| z_1-z_2 |,\epsilon}\Big(\frac{z_1+z_2}{2}\Big)\Big].    
\end{equation*}
Similarly, one can show that
\begin{align} \label{eqn::lim1-4/2-3}
	\lim_{z_1,z_2\to z} |z_2-z_1|^{-\frac{25}{24}} \times P\Big[z_1\xleftrightarrow{B}z_4\centernot{\xleftrightarrow{B}}z_2\xleftrightarrow{B} z_3\Big] \in (0,\infty). 
\end{align}

It remains to treat the term $P\Big[z_1\xleftrightarrow{B}z_2,z_3\xleftrightarrow{B}z_4\Big]$.
The necessary ingredients to complete the proof of~\eqref{eqn::OPE2} are in~\cite[Proof of Lemma~3.1]{CF24}. Below we provide some details for the reader's convenience. It is shown in~\cite{GarbanPeteSchrammPivotalClusterInterfacePercolation,Cam23} that there exists a constant $C_1\in (0,\infty)$ such that 
\begin{align}\label{eqn::two_point_spin}
		P\Big[z_1\xleftrightarrow{B}z_2\Big]:=\lim_{a\to 0}\pi_a^{-2}\times \mathbb{P}^a\Big[z_1^a\xleftrightarrow{B}z_2^a\Big]=\langle\psi(z_1)\psi(z_2)\rangle=\sqrt{C_1} |z_2-z_1|^{-\frac{5}{24}}.
\end{align}
Using~\eqref{eqn::two_point_spin}, we can write
\begin{align} \label{eqn::difference_four_spin}
	P\Big[z_1\xleftrightarrow{B}z_2,\;z_3\xleftrightarrow{B}z_4\Big]- C_1 |z_2-z_1|^{-\frac{5}{24}} |z_4-z_3|^{-\frac{5}{24}}= P\Big[z_1\xleftrightarrow{B}z_2,\;z_3\xleftrightarrow{B}z_4\Big]- P\Big[z_1\xleftrightarrow{B}z_2\Big] P\Big[z_3\xleftrightarrow{B}z_4\Big].
\end{align}
	
Given two subsets of the plane, $C$ and $D$,  we consider the following events:
\begin{itemize}
	\item $\{z_1\xleftrightarrow{B;C}z_2\}$: there is a black path connecting $z_1$ to $z_2$ contained in $C$;
	\item $\{z_1\xleftrightarrow[D]{B}z_2\}$: $z_1$ and $z_2$ belong to the same black cluster but there is no black path fully contained in $D$;
	\item $\{z_1\xleftrightarrow[D]{B;C}z_2\}$: there is a black path connecting $z_1$ to $z_2$ contained in $C$ but no black path fully contained in $D$. 
\end{itemize}
Now consider disks $B_m=B_m(z_1,z_2)=\{w: |w-\frac{z_1+z_2}{2}|\leq 2^m |z_2-z_1|\}$ for $m=1,\ldots, M$, where $M$ is chosen so that $2^M\sim 1/|z_2-z_1|$, that is, $M\sim -\log| z_2-z_1|$, and so that $z_3$ and $z_4$ are outside $B_M$. 

On the one hand, using the independence of labels at different vertices in percolation, we can write
\begin{align}\label{eqn::four_spin_decom1}
\begin{split}
	P\Big[z_1\xleftrightarrow{B}z_2,\; z_3\xleftrightarrow{B}z_4\Big]=&\,P\Big[z_1\xleftrightarrow{B;B_1}z_2,\; z_3\xleftrightarrow{B}z_4\Big]+\sum_{m=2}^M P\Big[z_1\xleftrightarrow[B_{m-1}]{B;B_m}z_2,\; z_3\xleftrightarrow{B}z_4\Big]\\
	&+P\Big[z_1\xleftrightarrow[B_M]{B}z_2,\; z_3\xleftrightarrow{B}z_4\Big]\\
	=&\,P\Big[z_1\xleftrightarrow{B;B_1}z_2,\; z_3\xleftrightarrow[B_1^c]{B}z_4\Big]+P\Big[z_1\xleftrightarrow{B;B_1}z_2\Big]P\Big[ z_3\xleftrightarrow{B;B_1^c}z_4\Big]\\
	&+\sum_{m=2}^M \left(P\Big[z_1\xleftrightarrow[B_{m-1}]{B;B_m}z_2,\; z_3\xleftrightarrow[B_m^c]{B}z_4\Big]+P\Big[z_1\xleftrightarrow[B_{m-1}]{B;B_m}z_2\Big]P\Big[ z_3\xleftrightarrow{B;B_m^c}z_4\Big]\right)\\
	&+P\Big[z_1\xleftrightarrow[B_M]{B}z_2,\; z_3\xleftrightarrow{B}z_4\Big],
\end{split}
\end{align}
where
\begin{align*}
P\Big[z_1\xleftrightarrow{B;B_1}z_2,\; z_3\xleftrightarrow{B}z_4\Big]:=&\lim_{a\to 0} \pi_a^{-4}\times  \mathbb{P}^a\Big[z_1^a\xleftrightarrow{B;B_1}z_2^a,\, z_3^a\xleftrightarrow{B}z_4^a\Big],\\
P\Big[z_1\xleftrightarrow[B_{m-1}]{B;B_m}z_2,\; z_3\xleftrightarrow{B}z_4\Big]:=&\lim_{a\to 0}\pi_a^{-4}\times \mathbb{P}^a\Big[z_1^a\xleftrightarrow[B_{m-1}]{B;B_m}z_2^a,\, z_3^a\xleftrightarrow{B}z_4^a\Big],\\
P\Big[z_1\xleftrightarrow[B_M]{B}z_2,\,z_3\xleftrightarrow{B}z_4\Big]:=&\lim_{a\to 0} \pi_a^{-4}\times  \mathbb{P}^a\Big[z_1^a\xleftrightarrow[B_M]{B}z_2^a,\, z_3^a\xleftrightarrow{B}z_4^a\Big],\\
P\Big[z_1\xleftrightarrow{B;B_1}z_2,\, z_3\xleftrightarrow[B_1^c]{B}z_4\Big]:=&\lim_{a\to 0}\pi_a^{-4}\times \mathbb{P}^a\Big[z_1^a\xleftrightarrow{B;B_1}z_2^a,\, z_3^a\xleftrightarrow[B_1^c]{B}z_4^a\Big],\\
P\Big[z_1\xleftrightarrow{B;B_1}z_2\Big]:=&\lim_{a\to0} \pi_a^{-2}\times \mathbb{P}^a\Big[z_1^a\xleftrightarrow{B;B_1}z_2^a\Big],\\
P\Big[z_3\xleftrightarrow{B;B_1^c}z_4\Big]:=&\lim_{a\to 0}\pi_a^{-2}\times \mathbb{P}^a\Big[z_3^a\xleftrightarrow{B;B_1^c}z_4^a\Big],\\
P\Big[z_1\xleftrightarrow[B_{m-1}]{B;B_m}z_2,\, z_3\xleftrightarrow[B_m^c]{B}z_4\Big]:=&\lim_{a\to 0}\pi_a^{-4}\times \mathbb{P}^a\Big[z_1^a\xleftrightarrow[B_{m-1}]{B;B_m}z_2^a,\, z_3^a\xleftrightarrow[B_m^c]{B}z_4^a\Big],\\
P\Big[z_1\xleftrightarrow[B_{m-1}]{B;B_m}z_2\Big]:=&\lim_{a\to 0}\pi_a^{-2}\times \mathbb{P}^a\Big[z_1^a\xleftrightarrow[B_{m-1}]{B;B_m}z_2^a\Big],\\
P\Big[z_3\xleftrightarrow{B;B_m^c}z_4\Big]:=&\lim_{a\to 0}\pi_a^{-2}\times \mathbb{P}^a\Big[z_3^a\xleftrightarrow{B;B_m^c}z_4^a\Big].
\end{align*}
The existence of these limits can be shown using the strategy in~\cite[Proof of Theorem~1.5]{Cam23}, without additional essential difficulties. On the other hand, we can write
\begin{align} \label{eqn::four_spin_decom2}
\begin{split}
	P\Big[z_1\xleftrightarrow{B}z_2\Big]P\Big[z_3\xleftrightarrow{B}z_4\Big]=&P\Big[z_1\xleftrightarrow{B;B_1}z_2\Big]P\Big[ z_3\xleftrightarrow[B_1^c]{B}z_4\Big]+P\Big[z_1\xleftrightarrow{B;B_1}z_2\Big]P\Big[ z_3\xleftrightarrow{B;B_1^c}z_4\Big]\\
&+\sum_{m=2}^M \left(P\Big[z_1\xleftrightarrow[B_{m-1}]{B;B_m}z_2\Big]P\Big[ z_3\xleftrightarrow[B_m^c]{B}z_4\Big]+P\Big[z_1\xleftrightarrow[B_{m-1}]{B;B_m}z_2\Big]P\Big[ z_3\xleftrightarrow{B;B_m^c}z_4\Big]\right)\\
&+P\Big[z_1\xleftrightarrow[B_M]{B}z_2\Big]P\Big[ z_3\xleftrightarrow{B}z_4\Big],
\end{split}
\end{align}
where
\begin{align*}
P\Big[z_3\xleftrightarrow[B_1^c]{B}z_4\Big]:=&\lim_{a\to 0} \pi_a^{-2}\times \mathbb{P}^a\Big[z_3^a\xleftrightarrow[B_1^c]{B}z_4^a\Big], \\
P\Big[z_1\xleftrightarrow[B_M]{B}z_2\Big]:=&\lim_{a\to 0} \pi_a^{-2}\times \mathbb{P}^a\Big[z_1^a\xleftrightarrow[B_M]{B}z_2^a\Big],
\end{align*}
and the existence of the limits can again be shown using the strategy in~\cite[Proof of Theorem~1.5]{Cam23}.
Combining~\eqref{eqn::four_spin_decom1} with~\eqref{eqn::four_spin_decom2} yields
\begin{align} \label{eqn::four_spin_decom}
\begin{split}
		&P\Big[z_1\xleftrightarrow{B}z_2,\;z_3\xleftrightarrow{B}z_4\Big]- P\Big[z_1\xleftrightarrow{B}z_2\Big] P\Big[z_3\xleftrightarrow{B}z_4\Big]\\
	&\qquad\qquad=\underbrace{P\Big[z_1\xleftrightarrow{B;B_1}z_2,\; z_3\xleftrightarrow[B_1^c]{B}z_4\Big]-P\Big[z_1\xleftrightarrow{B;B_1}z_2\Big]P\Big[ z_3\xleftrightarrow[B_1^c]{B}z_4\Big]}_{T_1}\\
	&\qquad\qquad\quad +\sum_{m=2}^M\bigg(\underbrace{P\Big[z_1\xleftrightarrow[B_{m-1}]{B;B_m}z_2,\ z_3\xleftrightarrow[B_m^c]{B}z_4\Big]-P\Big[z_1\xleftrightarrow[B_{m-1}]{B;B_m}z_2\Big]P\Big[z_3\xleftrightarrow[B_m^c]{B}z_4\Big]}_{T_m}\bigg)\\
	&\qquad\qquad\quad+\underbrace{P\Big[z_1\xleftrightarrow[B_M]{B}z_2,\; z_3\xleftrightarrow{B}z_4\Big]-P\Big[z_1\xleftrightarrow[B_M]{B}z_2\Big]P\Big[ z_3\xleftrightarrow{B}z_4\Big]}_{T_{M+1}}.
\end{split}
\end{align}

For the term $T_1$ (resp., $T_{M+1}$), note that the event $\{z_3\xleftrightarrow[B_1^c]{B}z_4\}$ (resp., $\{z_1\xleftrightarrow[B_M]{B}z_2\}$)
implies that there exist four monochromatic paths, with alternating labels (black/white/black/white), crossing the annulus $B_{2^M |z_2-z_1|}(\frac{z_1+z_2}{2})\setminus B_{2|z_2-z_1|}(\frac{z_1+z_2}{2})$. Consequently, as in the case of~\eqref{eqn::four_spin_unpair} above, one can show that
\begin{align}\label{eqn::four_spin_decom_aux0}
	\lim_{z_1,z_2\to z} |z_2-z_1|^{-\frac{25}{24}} \, T_1\in \mathbb{R} \quad \text{and}\quad \lim_{z_1,z_2\to z} |z_2-z_1|^{-\frac{25}{24}} \,  T_{M+1}\in \mathbb{R}. 
\end{align}

\begin{figure}
	\includegraphics[width= 0.5\textwidth]{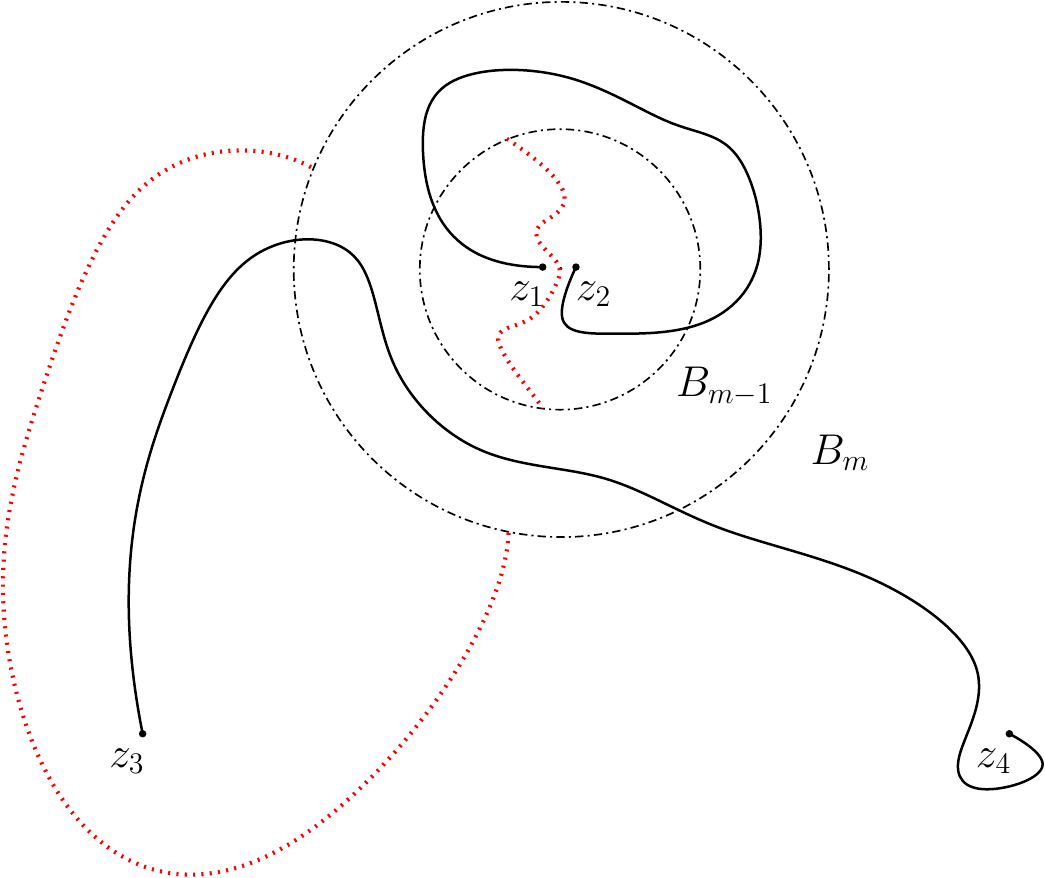}
	\caption{The event $\{ z_1 \xleftrightarrow[B_{m-1}]{B;B_m} z_2,\; z_3 \xleftrightarrow[B_m^c]{B} z_4 \}$ in the term $T_m$ in~\eqref{eqn::four_spin_decom}. The black, solid lines denote black paths, while the red, dotted lines denote white paths. $z_1$ and $z_2$ are contained in $B_{m-1}$. They are not connected by a black path within the disk $B_{m-1}$, but are connected within the larger disk $B_{m}$, with radius twice that of $B_{m-1}$. $z_3$ and $z_4$ are connected by a black path, but not outside $B_m$.  The number, $M$, of disks one can insert between the two groups of points $\{z_1,z_2\}$ and $\{z_3,z_4\}$ is of order $-\log{\vert z_2-z_1 \vert}$.}
	\label{fig::log_geometry}
\end{figure}

Let $m\in \{2,\ldots,M\}$ and consider the terms $T_m$ in~\eqref{eqn::four_spin_decom} (see Figure~\ref{fig::log_geometry} for an illustration of the event $\{ z_1 \xleftrightarrow[B_{m-1}]{B;B_m} z_2,\; z_3 \xleftrightarrow[B_m^c]{B} z_4 \}$). We write
\begin{align*}
	T_m=&\bigg(P\Big[z_1\xleftrightarrow[B_{m-1}]{B}z_2,\ z_3\xleftrightarrow[B_m^c]{B}z_4\big|z_1\xleftrightarrow{B;B_m}z_2\Big]  -P\Big[z_1\xleftrightarrow[B_{m-1}]{B}z_2 \big| z_1\xleftrightarrow{B;B_m}z_2\Big]P\Big[z_3\xleftrightarrow[B_m^c]{B}z_4\Big]\bigg) \\
&\qquad\qquad\qquad	\times P\Big[z_1\xleftrightarrow{B;B_m}z_2\Big],
\end{align*}
where 
\begin{align*}
	P\Big[z_1\xleftrightarrow[B_{m-1}]{B}z_2,\ z_3\xleftrightarrow[B_m^c]{B}z_4\big|z_1\xleftrightarrow{B;B_m}z_2\Big] :=&\,\frac{P\Big[z_1\xleftrightarrow[B_{m-1}]{B;B_m}z_2,\ z_3\xleftrightarrow[B_m^c]{B}z_4\Big] }{P\Big[z_1\xleftrightarrow{B;B_m}z_2\Big]}\\
	=&\,\pi_a^{-2}\times\lim_{a\to 0} \mathbb{P}^a\Big[ z_1^a\xleftrightarrow[B_{m-1}]{B}z_2^a,\; z_3^a\xleftrightarrow[B_m^c]{B}z_4^a \big | z_1^a\xleftrightarrow{B;B_m}z_2^a\Big],\\
	P\Big[z_1\xleftrightarrow[B_{m-1}]{B}z_2 \big| z_1\xleftrightarrow{B;B_m}z_2\Big]:=&\, \frac{P\Big[z_1\xleftrightarrow[B_{m-1}]{B;B_m}z_2\Big]}{P\Big[z_1\xleftrightarrow{B;B_m}z_2\Big]}=\lim_{a\to 0} \mathbb{P}^a\Big[z_1^a\xleftrightarrow[B_{m-1}]{B}z_2^a \big| z_1^a\xleftrightarrow{B;B_m}z_2^a\Big]
\end{align*}
and the existence of the limits follows from the same arguments as in the proof of~\cite[Theorem~1.5]{Cam23}.

On the one hand, the event $\{z_1\xlongleftrightarrow[B_{m-1}]{B}z_2\}$ implies that the annulus $B_{m-1}\setminus B_1$ is crossed by two black paths and two white paths. Since the four-arm exponent is equal to $\frac{5}{4}$~\cite{SmirnovWernerCriticalExponents}, using the up-to-constant estimate of the percolation four-arm probability from~\cite[Theorem~1.1]{ZhanGreen2SLE}, we conclude that
\[
P\big(z_1\xlongleftrightarrow[B_{m-1}]{B}z_2 \,\big|\, z_1\xlongleftrightarrow{B;B_m}z_2\big)
\asymp \big((1/2)^{m-2}\big)^{5/4}.
\] On the other hand, the event $\{z_3\xlongleftrightarrow[B_{m}^c]{B}z_4\}$ implies that the annulus $B_{M}\setminus B_m$ is crossed by two black paths and two white paths, alternating in color. Note, again using the four-arm exponent, that the probability of the latter event is of order $(2^m|z_2-z_1|)^{5/4}$. Recall that we will eventually let $z_1,z_2\to z$ for some $z\not\in \{z_3,z_4\}$ and that $\epsilon>0$ is chosen so that $B_{2\epsilon}(z)\cap \{z_3,z_4\}=\emptyset$.
Therefore, we can further write 
\begin{align}\label{eqn::Tm_decom}
\begin{split}
		& T_m = \bigg(P\Big[z_1\xleftrightarrow[B_{m-1}]{B}z_2,\ z_3\xleftrightarrow[B_m^c]{B}z_4\big|z_1\xleftrightarrow{B;B_m}z_2,\, \mathcal{F}_{2|z_2-z_1|,2^{m-1}|z_2-z_1|}(\tfrac{z_1+z_2}{2}),\, \mathcal{F}_{2^m|z_2-z_1|,\epsilon}(\tfrac{z_1+z_2}{2})\Big]  \\
		& \quad - P\Big[z_1\xleftrightarrow[B_{m-1}]{B}z_2 \big| z_1\xleftrightarrow{B;B_m}z_2,\, \mathcal{F}_{2|z_2-z_1|,2^{m-1}|z_2-z_1|}(\tfrac{z_1+z_2}{2})\Big]P\Big[z_3\xleftrightarrow[B_m^c]{B}z_4\big|\mathcal{F}_{2^m|z_2-z_1|,\epsilon}(\tfrac{z_1+z_2}{2})\Big]\bigg) \\
	&\qquad	\times \underbrace{P\Big[z_1\xleftrightarrow{B;B_m}z_2\Big]}_{\sim |z_2-z_1|^{-\frac{5}{24}}}\underbrace{P\Big[\mathcal{F}_{2|z_2-z_1|,2^{m-1}|z_2-z_1|}(\tfrac{z_1+z_2}{2})\big | z_1\xleftrightarrow{B;B_m}z_2\Big] P\Big[\mathcal{F}_{2^m|z_2-z_1|,\epsilon}(\tfrac{z_1+z_2}{2})\Big]}_{\sim |z_2-z_1|^{\frac{5}{4}}},
\end{split}
\end{align}
where
\begin{align*}
	P\Big[\mathcal{F}_{2|z_2-z_1|,2^{m-1}|z_2-z_1|}(\tfrac{z_1+z_2}{2})\big| z_1\xleftrightarrow{B;B_m}z_2\Big]:=&\,\frac{P\Big[\mathcal{F}_{2|z_2-z_1|,2^{m-1}|z_2-z_1|}(\tfrac{z_1+z_2}{2}),\;  z_1\xleftrightarrow{B;B_m}z_2\Big]}{P\Big[z_1\xleftrightarrow{B;B_m}z_2\Big]},\\
	P\Big[z_3\xleftrightarrow[B_m^c]{B}z_4\big|\mathcal{F}_{2^m|z_2-z_1|,\epsilon}(\tfrac{z_1+z_2}{2})\Big]:= &\,\frac{P\Big[z_3\xleftrightarrow[B_m^c]{B}z_4,\; \mathcal{F}_{2^m|z_2-z_1|,\epsilon}(\tfrac{z_1+z_2}{2})\Big]}{P\Big[\mathcal{F}_{2^m|z_2-z_1|,\epsilon}(\tfrac{z_1+z_2}{2})\Big]},\\
	P\Big[z_1\xleftrightarrow[B_{m-1}]{B}z_2 \big| z_1\xleftrightarrow{B;B_m}z_2,\, \mathcal{F}_{2|z_2-z_1|,2^{m-1}|z_2-z_1|}(\tfrac{z_1+z_2}{2})\Big]:=& \,\frac{P\Big[z_1\xleftrightarrow[B_{m-1}]{B;B_m}z_2, \, \mathcal{F}_{2|z_2-z_1|,2^{m-1}|z_2-z_1|}(\tfrac{z_1+z_2}{2})\Big]}{P\Big[z_1\xleftrightarrow{B;B_m}z_2, \, \mathcal{F}_{2|z_2-z_1|,2^{m-1}|z_2-z_1|}(\tfrac{z_1+z_2}{2})\Big]},
\end{align*}
\begin{align*}
	&P\Big[z_1\xleftrightarrow[B_{m-1}]{B}z_2, z_3\xleftrightarrow[B_m^c]{B}z_4\big|z_1\xleftrightarrow{B;B_m}z_2,\, \mathcal{F}_{2|z_2-z_1|,2^{m-1}|z_2-z_1|}(\tfrac{z_1+z_2}{2}),\, \mathcal{F}_{2^m|z_2-z_1|,\epsilon}(\tfrac{z_1+z_2}{2})\Big] \\
	&\qquad\qquad\qquad	:=\frac{P\Big[z_1\xleftrightarrow[B_{m-1}]{B;B_m}z_2, z_3\xleftrightarrow[B_m^c]{B}z_4,\, \mathcal{F}_{2|z_2-z_1|,2^{m-1}|z_2-z_1|}(\tfrac{z_1+z_2}{2}),\, \mathcal{F}_{2^m|z_2-z_1|,\epsilon}(\tfrac{z_1+z_2}{2})\Big]}{P\Big[z_1\xleftrightarrow{B;B_m}z_2,\, \mathcal{F}_{2|z_2-z_1|,2^{m-1}|z_2-z_1|}(\tfrac{z_1+z_2}{2}),\, \mathcal{F}_{2^m|z_2-z_1|,\epsilon}(\tfrac{z_1+z_2}{2})\Big]} ,
\end{align*}
and where
\begin{align*}
	P\Big[\mathcal{F}_{2|z_2-z_1|,2^{m-1}|z_2-z_1|}(\tfrac{z_1+z_2}{2}),\;  z_1\xleftrightarrow{B;B_m}z_2\Big]:=&\lim_{a\to 0} \pi_a^{-2}\times \mathbb{P}^a\Big[\mathcal{F}_{2|z_2^a-z_1^a|,2^{m-1}|z_2^a-z_1^a|}\big(\tfrac{z_1^a+z_2^a}{2}\big),\;  z_1^a\xleftrightarrow{B;B_m}z_2^a\Big],\\
	P\Big[z_3\xleftrightarrow[B_m^c]{B}z_4,\; \mathcal{F}_{2^m|z_2-z_1|,\epsilon}(\tfrac{z_1+z_2}{2})\Big]:=&\lim_{a\to 0}\pi_a^{-2}\times \mathbb{P}^a\Big[z_3^a\xleftrightarrow[B_m^c]{B}z_4^a,\; \mathcal{F}_{2^m|z^a_2-z^a_1|,\epsilon}\big(\tfrac{z^a_1+z^a_2}{2}\big)\Big],\\
    P\Big[\mathcal{F}_{2|z_2-z_1|,2^{m-1}|z_2-z_1|}\Big(\frac{z_1+z_2}{2}\Big)\Big] :=& \lim_{a \to 0} \mathbb{P}^a\Big[\mathcal{F}_{2|z^a_2-z^a_1|,2^{m-1}|z^a_2-z^a_1|}\Big(\frac{z^a_1+z^a_2}{2}\Big)\Big], \\
	P\Big[z_1\xleftrightarrow[B_{m-1}]{B;B_m}z_2, \, \mathcal{F}_{2|z_2-z_1|,2^{m-1}|z_2-z_1|}(\tfrac{z_1+z_2}{2})\Big]:=&\lim_{a\to 0}\pi_a^{-2}\times \mathbb{P}^a \Big[z^a_1\xleftrightarrow[B_{m-1}]{B;B_m}z^a_2, \, \mathcal{F}_{2|z^a_2-z^a_1|,2^{m-1}|z^a_2-z^a_1|}\big(\tfrac{z^a_1+z^a_2}{2}\big)\Big],\\
	P\Big[z_1\xleftrightarrow{B;B_m}z_2, \, \mathcal{F}_{2|z_2-z_1|,2^{m-1}|z_2-z_1|}(\tfrac{z_1+z_2}{2})\Big]:=&\lim_{a\to 0}\pi_a^{-2}\times \mathbb{P}^a \Big[z^a_1\xleftrightarrow{B;B_m}z^a_2, \, \mathcal{F}_{2|z^a_2-z^a_1|,2^{m-1}|z^a_2-z^a_1|}\big(\tfrac{z^a_1+z^a_2}{2}\big)\Big],
\end{align*}
\begin{align*}
	&P\Big[z_1\xleftrightarrow[B_{m-1}]{B;B_m}z_2, z_3\xleftrightarrow[B_m^c]{B}z_4,\, \mathcal{F}_{2|z_2-z_1|,2^{m-1}|z_2-z_1|}(\tfrac{z_1+z_2}{2}),\, \mathcal{F}_{2^m|z_2-z_1|,\epsilon}(\tfrac{z_1+z_2}{2})\Big]\\
	&\qquad\qquad:=\lim_{a\to 0} \pi_a^{-4}\times \mathbb{P}^a\Big[z^a_1\xleftrightarrow[B_{m-1}]{B;B_m}z^a_2, z^a_3\xleftrightarrow[B_m^c]{B}z^a_4,\, \mathcal{F}_{2|z^a_2-z^a_1|,2^{m-1}|z^a_2-z^a_1|}\big(\tfrac{z^a_1+z^a_2}{2}\big),\, \mathcal{F}_{2^m|z^a_2-z^a_1|,\epsilon}\big(\tfrac{z^a_1+z^a_2}{2}\big)\Big],\\
	&P\Big[z_1\xleftrightarrow{B;B_m}z_2,\, \mathcal{F}_{2|z_2-z_1|,2^{m-1}|z_2-z_1|}(\tfrac{z_1+z_2}{2}),\, \mathcal{F}_{2^m|z_2-z_1|,\epsilon}(\tfrac{z_1+z_2}{2})\Big]\\
	&\qquad\qquad:=\lim_{a\to 0} \pi_a^{-2}\times \mathbb{P}^a\Big[z^a_1\xleftrightarrow{B;B_m}z^a_2,\, \mathcal{F}_{2|z^a_2-z^a_1|,2^{m-1}|z^a_2-z^a_1|}\big(\tfrac{z^a_1+z^a_2}{2}\big),\, \mathcal{F}_{2^m|z^a_2-z^a_1|,\epsilon}\big(\tfrac{z^a_1+z^a_2}{2}\big)\Big].
\end{align*}
The existence of these limits can be shown using the strategy in~\cite[Proof of Theorem~1.5]{Cam23}. Note that the power of the normalization term $\pi_a$ is determined by the number of one-arm events involved.

To set up a Cauchy sequence argument, let $\{z_1^{(n)}\}_{n\geq 1}$ and $\{z_2^{(n)}\}_{n\geq 1}$ be two sequences of complex numbers satisfying $z_1^{(n)}\neq z_2^{(n)}$ and $z_1^{(n)}\to z$, $z_2^{(n)}\to z$ as $n\to \infty$. To keep track of the influence of $n$, we write $B^{(n)}_m=B^{(n)}_m\big(z_1^{(n)},z_2^{(n)}\big)=\{w: |w-\frac{z^{(n)}_1+z^{(n)}_2}{2}|\leq 2^m |z^{(n)}_2-z^{(n)}_1|\}$ for $m=1,\ldots, M_n$, where $M_n$ is chosen so that $2^{M_n}\sim 1/|z^{(n)}_2-z^{(n)}_1|$,  that is, $M_n\sim -\log| z^{(n)}_2-z^{(n)}_1|$, and so that $z_3$ and $z_4$ are outside $B_{M_n}$. We also write $T_m^{(n)}$ for the term $T_m$ associated with $z_1^{(n)}$ and $z_2^{(n)}$.
\begin{lemma} \label{lem::OPE2}
    There exists a constant $c\in (0,\infty)$, which depends only on $z$, $z_1$, $z_4$ and $\epsilon$, such that
    \begin{enumerate}
        \item for any fixed $m\geq 1$ and all $n_1,n_2$ satisfying $\min \{M_{n_1},M_{n_2}\}\geq m$, we have 
	\begin{align}\label{eqn::generic_m_limit_aux1}
		\big| |z_2^{(n_1)}-z_1^{(n_1)}|^{-\frac{25}{24}}T_m^{(n_1)}-|z_2^{(n_2)}-z_1^{(n_2)}|^{-\frac{25}{24}}T_m^{(n_2)}\big| \lesssim 2^{-c \min \{M_{n_1}-m, M_{n_2}-m\}},
	\end{align} \label{item::lem::OPE1_1}
        \item for any $m_1,m_2\geq 1$ and all $n$ satisfying $M_n\geq \max\{m_1,m_2\}$, we have 
	\begin{align}\label{eqn::generic_m_limit_aux2}
		\big|  |z_2^{(n)}-z_1^{(n)}|^{-\frac{25}{24}}T_{m_1}^{(n)}-  |z_2^{(n)}-z_1^{(n)}|^{-\frac{25}{24}}T_{m_2}^{(n)} \big| \lesssim 2^{-c\min\{m_1, m_2, M_{n}-m_1, M_{n}-m_2\}}.
	\end{align} \label{item::lem::OPE1_2}
    \end{enumerate}
\end{lemma} 

We postpone the proof of Lemma~\ref{lem::OPE2} to the next section.
It follows from~\eqref{eqn::generic_m_limit_aux1} that, for each $m\geq 1$, there exists a function $f_m(z,z_3,z_4)\in \mathbb{R}$ such that, for all $n\geq 1$ with $M_n\geq m$, we have 
\begin{align}\label{eqn::generic_m_limit}
	\big||z_2^{(n)}-z_1^{(n)}|^{-\frac{25}{24}} \, T_m^{(n)}-f_m(z,z_3,z_4)\big| \lesssim & \; 2^{-c(M_n-m)}.
\end{align}
Letting $n\to \infty$ in~\eqref{eqn::generic_m_limit_aux2} and using~\eqref{eqn::generic_m_limit}, the triangular inequality gives
	\begin{align*}
		\big|f_{m_1}(z,z_3,z_4)-f_{m_2}(z,z_3,z_4)\big|\lesssim  2 \times 2^{-c \min\{m_1,m_2\}}, \quad \text{for all }m_1,m_2\geq 1,
	\end{align*}
	which implies the existence of the limit $\lim_{m\to \infty} f_{m}(z,z_3,z_4):=f(z,z_3,z_4)\in (0,\infty)$ and the estimate
	\begin{align} \label{eqn::generic_m_limit_aux0}
		\big| f_m(z,z_3,z_4)-f(z,z_3,z_4)\big|\lesssim  2\times 2^{-c m},\quad \text{for all }m\geq 1. 
	\end{align}
    The fact that $f(z,z_3,z_4)>0$ is subtle and its proof requires several steps. Intuitively, the FKG inequality suggests that $f(z,z_3,z_4)\geq 0$, while strict positivity requires a more detailed analysis. The argument is the same as that leading to (3.19) of~\cite{CF24}, so we only list the main steps, referring to~\cite{CF24} for the details.
	\begin{itemize}
		\item First, using a standard application of RSW estimates (see e.g., the proofs of Lemmas 2.1 and 2.2 of~\cite{CamiaNewman2009ising}), the two quantities 
{\small		\begin{align*}
	& |z_2-z_1|^{\frac{5}{24}}P\Big[z_1\xleftrightarrow{B;B_m}z_2\Big],\\
		& |z_2-z_1|^{-\frac{5}{4}}P\Big[\mathcal{F}_{2|z_2-z_1|,2^{m-1}|z_2-z_1|}(\tfrac{z_1+z_2}{2})\big | z_1\xleftrightarrow{B;B_m}z_2\Big] P\Big[\mathcal{F}_{2^m|z_2-z_1|,\epsilon}(\tfrac{z_1+z_2}{2})\Big]
\end{align*}}
 are bounded away from $0$ uniformly in $m$ as $z_1,z_2 \to z$. 
 \item Second, using the coupling results concerning four-arm events in~\cite[Lemmas~2.1 and~2.7]{CF24} and the strategy in the proof of Lemma~\ref{lem::four_arm_coupling_argument}, one can show that the limit below exists:
 \begin{align*}
 	&\lim_{z_1,z_2\to z} \bigg(P\Big[z_1\xleftrightarrow[B_{m-1}]{B}z_2,\ z_3\xleftrightarrow[B_m^c]{B}z_4\big|z_1\xleftrightarrow{B;B_m}z_2,\, \mathcal{F}_{2|z_2-z_1|,2^{m-1}|z_2-z_1|}(\tfrac{z_1+z_2}{2}),\, \mathcal{F}_{2^m|z_2-z_1|,\epsilon}(\tfrac{z_1+z_2}{2})\Big]  \\
 	& \quad - P\Big[z_1\xleftrightarrow[B_{m-1}]{B}z_2 \big| z_1\xleftrightarrow{B;B_m}z_2,\, \mathcal{F}_{2|z_2-z_1|,2^{m-1}|z_2-z_1|}(\tfrac{z_1+z_2}{2})\Big]P\Big[z_3\xleftrightarrow[B_m^c]{B}z_4\big|\mathcal{F}_{2^m|z_2-z_1|, \epsilon}(\tfrac{z_1+z_2}{2})\Big]\bigg)\\
 	&\quad = \lim_{z_1,z_2\to z} P\Big[z_1\xleftrightarrow[B_{m-1}]{B}z_2 \big| z_1\xleftrightarrow{B;B_m}z_2,\, \mathcal{F}_{2|z_2-z_1|,2^{m-1}|z_2-z_1|}(\tfrac{z_1+z_2}{2})\Big]\\
 &	\qquad\times \lim_{z_1,z_2\to z} \bigg(P\Big[z_3\xleftrightarrow[B_m^c]{B}z_4\big|z_1\xleftrightarrow[B_{m-1}]{B;B_m}z_2,\,  \mathcal{F}_{2^m|z_2-z_1|,\epsilon}(\tfrac{z_1+z_2}{2})\Big]  \\
 	& \qquad \qquad \qquad - P\Big[z_3\xleftrightarrow[B_m^c]{B}z_4\big|\mathcal{F}_{2^m|z_2-z_1|,\epsilon}(\tfrac{z_1+z_2}{2})\Big]\bigg),
 \end{align*}
 where
 \begin{align*}
 	P\Big[z_3\xleftrightarrow[B_m^c]{B}z_4\big|z_1\xleftrightarrow[B_{m-1}]{B;B_m}z_2,\,  \mathcal{F}_{2^m|z_2-z_1|,\epsilon}(\tfrac{z_1+z_2}{2})\Big]:=\frac{P\Big[z_1\xleftrightarrow[B_{m-1}]{B;B_m}z_2,\, z_3\xleftrightarrow[B_m^c]{B}z_4,\, \mathcal{F}_{2^m|z_2-z_1|,\epsilon}(\tfrac{z_1+z_2}{2})\Big] }{P\Big[z_1\xleftrightarrow[B_{m-1}]{B;B_m}z_2,\,  \mathcal{F}_{2^m|z_2-z_1|,\epsilon}(\tfrac{z_1+z_2}{2})\Big] } ,
 \end{align*}
 and where 
{\small \begin{align*}
		P\Big[z_1\xleftrightarrow[B_{m-1}]{B;B_m}z_2,\, z_3\xleftrightarrow[B_m^c]{B}z_4,\, \mathcal{F}_{2^m|z_2-z_1|,\epsilon}(\tfrac{z_1+z_2}{2})\Big]:=&\lim_{a\to 0}\pi_a^{-4} \\
        & \quad\times \mathbb{P}^a\Big[z^a_1\xleftrightarrow[B_{m-1}]{B;B_m}z^a_2,\, z^a_3\xleftrightarrow[B_m^c]{B}z^a_4,\, \mathcal{F}_{2^m|z^a_2-z^a_1|,\epsilon}\big(\tfrac{z^a_1+z^a_2}{2}\big)\Big],\\
		P\Big[z_1\xleftrightarrow[B_{m-1}]{B;B_m}z_2,\,  \mathcal{F}_{2^m|z_2-z_1|,\epsilon}(\tfrac{z_1+z_2}{2})\Big] :=&\lim_{a\to 0}\pi_a^{-2}\times  \mathbb{P}^a\Big[z^a_1\xleftrightarrow[B_{m-1}]{B;B_m}z^a_2,\,  \mathcal{F}_{2^m|z^a_2-z^a_1|,\epsilon}\big(\tfrac{z^a_1+z^a_2}{2}\big)\Big].
\end{align*}}The existence of these limits can be shown using the strategy in~\cite[Proof of Theorem~1.5]{Cam23}.
Moreover, a standard application of RSW estimates (see, e.g., the proofs of Lemmas 2.1 and 2.2 of~\cite{CamiaNewman2009ising}) implies that the limit
\begin{align*}
 \lim_{z_1,z_2\to z} P\Big[z_1\xleftrightarrow[B_{m-1}]{B}z_2 \big| z_1\xleftrightarrow{B;B_m}z_2,\, \mathcal{F}_{2|z_2-z_1|,2^{m-1}|z_2-z_1|}(\tfrac{z_1+z_2}{2})\Big]
\end{align*}
is bounded away from $0$ uniformly in $m$.

\item Third, according to~\cite[Lemma~3.6]{CF24} and its proof, we have
\begin{align*}
	&\lim_{m\to \infty} \lim_{z_1,z_2\to z} \bigg(P\Big[ z_3\xleftrightarrow[B_m^c]{B}z_4\big|z_1\xleftrightarrow[B_{m-1}]{B;B_m}z_2,\,  \mathcal{F}_{2^m|z_2-z_1|,\epsilon}(\tfrac{z_1+z_2}{2})\Big]  \\
	&\qquad - P\Big[z_3\xleftrightarrow[B_m^c]{B}z_4\big|\mathcal{F}_{2^m|z_2-z_1|,\epsilon}(\tfrac{z_1+z_2}{2})\Big]\bigg) >0,
\end{align*}
 This last step is the most subtle of the three.
	\end{itemize}
It follows from the above three steps, combined with~\eqref{eqn::Tm_decom}, \eqref{eqn::generic_m_limit} and~\eqref{eqn::generic_m_limit_aux0}, that
\begin{equation} \label{eqn::f}
    f(z,z_3,z_4)=\lim_{m\to \infty}f_m(z,z_3,z_4)=\lim_{m\to \infty}\lim_{z_1,z_2\to z} |z_2-z_1|^{-\frac{25}{24}}T_m>0.
\end{equation}    
	
Combining~\eqref{eqn::spin_four_connec_discrete}, \eqref{eqn::two_point_spin} and~\eqref{eqn::four_spin_decom} with~\eqref{eqn::four_spin_unpair}, \eqref{eqn::lim1-4/2-3}, \eqref{eqn::four_spin_decom_aux0} and~\eqref{eqn::f},
we have the existence of the following limits:
	{\small\begin{align}\label{eqn::OPE2_limit}
			\begin{split}
				&\lim_{z_1,z_2\to z } \Big(\langle \psi(z_1)\psi(z_2)\psi(z_3)\psi(z_4)\rangle -C_1|z_2-z_1|^{-\frac{5}{24}} |z_4-z_3|^{-\frac{5}{24}} - M|z_2-z_1|^{\frac{25}{24}} f(z,z_3,z_4)\Big) |z_2-z_1|^{-\frac{25}{24}}\\
				& = \lim_{z_1,z_2\to z } \Big(\langle \psi(z_1)\psi(z_2)\psi(z_3)\psi(z_4)\rangle -C_1|z_2-z_1|^{-\frac{5}{24}} |z_4-z_3|^{-\frac{5}{24}} +\frac{\log |z_2-z_1|}{\log 2}|z_2-z_1|^{\frac{25}{24}} f(z,z_3,z_4)\Big) |z_2-z_1|^{-\frac{25}{24}},
			\end{split}
	\end{align}}where we have also used the observation that $M \sim -\log_2 |z_2-z_1|$ to get the equality.	

    \begin{remark} \label{rem::truncated_four_point}
		In particular, we have
		\begin{align*}
			\lim_{z_1,z_2\to z} \frac{P\Big[z_1\xleftrightarrow{B}z_2,\; z_3\xleftrightarrow{B}z_4\Big]-P\Big[z_1\xleftrightarrow{B}z_2\Big]P\Big[z_3\xleftrightarrow{B}z_4\Big]}{|z_2-z_1|^{\frac{25}{24}}|\log |z'-z||}= \frac{1}{\log 2} \lim_{m\to \infty}\lim_{z_1,z_2\to z} |z_2-z_1|^{-\frac{25}{24}} \times T_m.
		\end{align*}
\end{remark}

	Since the four-point function $\langle \psi(z_1)\psi(z_2)\psi(z_3)\psi(z_4)\rangle $ satisfies the M\"obius covariance rule~\eqref{eqn::cov_psi}, we conclude that, for any non-constant M\"obius transformation $\varphi$ with $\varphi(z_j)\neq \infty$ for $1\leq j\leq 4$, we have
	\begin{align} \label{eqn::OPE2_COV}
		f\big(\varphi(z),\varphi(z_3),\varphi(z_4)\big)=f(z,z_3,z_4)\times |\varphi'(z)|^{-\frac{5}{4}} |\varphi'(z_3)|^{-\frac{5}{48}}|\varphi'(z_4)|^{-\frac{5}{48}},
	\end{align} 
	which implies that 
	\begin{align}\label{eqn::OPE2_limit_aux1}
		f(z,z_3,z_4)= CF(z,z_3,z_4)\quad  \text{for some constant }C\in (0,\infty),
	\end{align}
    where $F$ is the function defined in~\eqref{eqn::def_F}.
    Plugging~\eqref{eqn::OPE2_limit_aux1} into~\eqref{eqn::OPE2_limit} and comparing it with~\eqref{eqn::OPE1}, we obtain
	\begin{align*}
		\frac{C}{\log 2}=C_1C_L ,
	\end{align*}
    where the factor $\log 2$ comes from the dyadic annulus decomposition.

    The considerations above imply the existence of the following limit:
	\begin{align*}
		&\Big(\langle \psi (z_1)\psi(z_2)\psi (z_3) \psi (z_4)\rangle -C_1 |z_2-z_1|^{-\frac{5}{24}} |z_4-z_3|^{-\frac{5}{24}} + C_1C_L |z_2-z_1|^{\frac{25}{24}} F(z,z_3,z_4) \log\left|\tfrac{(z_2-z_1)(z_4-z_3)}{(z_3-z_1)(z_4-z_2)}\right|\Big)\\
		&\qquad\qquad\qquad \times |z_2-z_1|^{-\frac{25}{24}} \enspace \longrightarrow \enspace \hat{f}(z,z_3,z_4) \in \mathbb{R},\quad \text{as }z_1,z_2\to z. 
	\end{align*}
    Moreover, since $\langle \psi (z_1)\psi(z_2)\psi (z_3) \psi (z_4)\rangle$ and $F(z,z_3,z_4)$ are conformally covariant and the cross-ratio is a conformally invariant quantity, the expression within parentheses is manifestly conformally covariant. Therefore, the function $\hat{f}(z,z_3,z_4)$ satisfies the same M\"obius covariance rule as the function $f$ above (see~\eqref{eqn::OPE2_COV}), which implies that
	\begin{align*}
		\hat{f}(z,z_3,z_4)= C_1C_0 F(z,z_3,z_4) \quad \text{for some }C_0\in \mathbb{R}
	\end{align*}
	and completes the proof of~\eqref{eqn::OPE2}.

\subsection{Proof of technical lemmas}

In this section, we prove Lemmas~\ref{lem::four_arm_coupling_argument} and~\ref{lem::OPE2}, used in the proof of~\eqref{eqn::OPE2}, as well as Lemma~\ref{lem::limit_two_point}, which will be used in Section~\ref{subsec::mixed_cor} below.

We begin with the proof of Lemma~\ref{lem::four_arm_coupling_argument}.
\begin{proof}[Proof of Lemma~\ref{lem::four_arm_coupling_argument}]
Using the translation invariance of percolation, it suffices to prove the following claim. Suppose that $z_3^{(n)}\to z_3$ and $z_4^{(n)}\to z_4$ as $n\to \infty$, and that $z_1^{(n)}\neq z_2^{(n)}$ satisfy
\[
z_1^{(n)}+z_2^{(n)}=0
\quad\text{and}\quad
z_1^{(n)},z_2^{(n)}\to 0
\qquad\text{as } n\to\infty.
\]
Then the sequence
\begin{align*}
\left\{
|z_2^{(n)}-z_1^{(n)}|^{\frac{5}{24}}
\times
P\Big[
z_1^{(n)}\xleftrightarrow{B}z_3^{(n)}
\centernot{\xleftrightarrow{B}}
z_2^{(n)}\xleftrightarrow{B}z_4^{(n)}
\,\Big|\,
\mathcal{F}_{2|z_1^{(n)}-z_2^{(n)}|,\epsilon}(0)
\Big]
\right\}_{n=1}^{\infty}
\end{align*}
is Cauchy. This reduction simplifies the argument, since the annuli
$B_{\epsilon}(\tfrac{z_1^{(n)}+z_2^{(n)}}{2})\setminus B_{2|z_2^{(n)}-z_1^{(n)}|}(\tfrac{z_1^{(n)}+z_2^{(n)}} {2})$ now
have the same center for all $n$. Without loss of generality, we may assume that $|z_j^{(n)}-z_j|\leq \epsilon/20$ for $1\leq j\leq 4$ and all $n\geq 1$. For each $n\geq 1$ and $1\leq j\leq 4$, let $z_j^{(a,n)}$ be a vertex in $a\mathcal{T}$ such that $|z_j^{(a,n)}-z_j^{(n)}|\leq a$. 

Recall that
\begin{align*}
&    P\Big[
z_1^{(n)}\xleftrightarrow{B}z_3^{(n)}
\centernot{\xleftrightarrow{B}}
z_2^{(n)}\xleftrightarrow{B}z_4^{(n)}
\,\Big|\,
\mathcal{F}_{2|z_1^{(n)}-z_2^{(n)}|,\epsilon}(0)
\Big]\\
&\qquad=\lim_{a\to 0} \pi_a^{-4}\times \mathbb{P}^a\Big[
z_1^{(a,n)}\xleftrightarrow{B}z_3^{(a,n)}
\centernot{\xleftrightarrow{B}}
z_2^{(a,n)}\xleftrightarrow{B}z_4^{(a,n)}
\,\Big|\,
\mathcal{F}_{2|z_1^{(a,n)}-z_2^{(a,n)}|,\epsilon}(0)
\Big].
\end{align*}
For $0<r_1<r_2$, define events
\begin{align*}
    \hat{\mathcal{F}}_{r_2}(z_1^{(a,n)},z_2^{(a,n)}):=&\Big\{z_1^{(a,n)}\xleftrightarrow{B}\partial B_{r_2}(0),\; z_2^{(a,n)}\xleftrightarrow{B}\partial B_{r_2}(0),\; \big(z_1^{(a,n)}\xleftrightarrow{B;B_{r_2}(0)}z_2^{(a,n)}\big)^c\Big\},\\
        \mathring{\mathcal{F}}_{r_1,r_2}(0):= &\mathcal{F}_{r_1,r_2}(0)\setminus \{\text{there exist five disjoint paths, not all with the same label,} \\
        & \qquad\qquad\qquad\text{ crossing the annulus }B_{r_2}(0)\setminus B_{r_1}(0)\}. 
\end{align*}
    Since the five-arm exponent for critical percolation is strictly larger than its four-arm exponent~\cite{SmirnovWernerCriticalExponents}, we have
    \begin{align*}
        {\limsup_{a\to 0}} \, \mathbb{P}^a\Big[\mathcal{F}_{r_1,r_2}(0)\setminus \mathring{\mathcal{F}}_{r_1,r_2}(0)\big|\mathcal{F}_{r_1,r_2}(0) \Big]=o(1), \quad \text{as }\frac{r_1}{r_2} \to 0.
    \end{align*}
    Therefore, we can write
    \begin{equation} \label{eqn::four_arm_coupling_argument_aux1}
\begin{split}
           & \pi_a^{-4}\times \mathbb{P}^a\Big[
z_1^{(a,n)}\xleftrightarrow{B}z_3^{(a,n)}
\centernot{\xleftrightarrow{B}}
z_2^{(a,n)}\xleftrightarrow{B}z_4^{(a,n)}
\,\Big|\,
\mathcal{F}_{2|z_1^{(a,n)}-z_2^{(a,n)}|,\epsilon}(0)
\Big]\\
&= \pi_a^{-2}\times  \mathbb{P}^a\Big[
 \hat{\mathcal{F}}_{\epsilon}(z_1^{(a,n)},z_2^{(a,n)}) \,\Big|\,
\mathcal{F}_{2|z_1^{(a,n)}-z_2^{(a,n)}|,\epsilon}(0)
\Big]  \\
 &\quad\times \Big( \pi_a^{-2}\times \mathbb{P}^a\Big[
z_1^{(a,n)}\xleftrightarrow{B}z_3^{(a,n)}
\centernot{\xleftrightarrow{B}}
z_2^{(a,n)}\xleftrightarrow{B}z_4^{(a,n)}
\,\Big|\,
\hat{\mathcal{F}}_{\epsilon}(z_1^{(a,n)},z_2^{(a,n)})\cap \mathring{\mathcal{F}}_{2|z_1^{(a,n)}-z_2^{(a,n)}|,\epsilon}(0)
\Big]+o(1)\Big) + o(1),
\end{split}
    \end{equation}
    as $a\to 0$ and then $n\to \infty$. 

We will adopt the four-arm event coupling argument developed in~\cite{GarbanPeteSchrammPivotalClusterInterfacePercolation}. To this end, we will use the notion, introduced in~\cite{GarbanPeteSchrammPivotalClusterInterfacePercolation}, of (outward) ``faces" induced by arms of alternating labels. We briefly recall its definition.

 Let $r>0$ and, for $1\leq j\leq 4$, let $y_j\in\partial B_r(0)$ be on the boundary of some hexagon $\tilde{y}_j\in a\mathcal{H}$ which intersects $\partial B_r(0)$, where $y_1,\ldots,y_4$ are chosen in counterclockwise order. A \textit{configuration of faces} $\boldsymbol{\eta}$ around the circle $\partial B_r(0)$ with endpoints $y_1,\ldots,y_4$ is a collection of four oriented simple paths $(\eta_1,\eta_2,\eta_3,\eta_4)$ consisting of hexagons of $a\mathcal{H}$ such that, for $j=1,2,3,4$ (with the convention that $y_5=y_1$),
		\begin{itemize}
			\item $y_j$ is on the boundary of the first hexagon of the path $\eta_j$ and $y_{j+1}$ is on the boundary of the last hexagon of $\eta_j$;
			\item $\eta_j$ is a path consisting of black (resp., white) hexagons if $j$ is odd (resp., even);
			\item there are no hexagons in $\eta_j$ that are entirely contained in $\big(B_r(0)\big)^c$.
		\end{itemize}

\noindent We claim that, on the event $\mathring{\mathcal{F}}_{2|z_1^{(a,n)}-z_2^{(a,n)}|,\epsilon}(0)$, the four arms crossing the annulus $B_{\epsilon}(0)\setminus B_{2|z_1^{(a,n)}-z_2^{(a,n)}|}(0)$ naturally determine a configuration of faces around the circle $\partial B_{\epsilon}(0)$. Indeed, let $H$ be the union of all hexagons {of $a\mathcal{H}$ that intersect} $B_{\epsilon}(0)\setminus B_{2|z_1^{(a,n)}-z_2^{(a,n)}|}(0)$ with the following property: if a hexagon is black (resp., white), then it is connected to both $\partial B_{\epsilon}(0)$ and $\partial B_{2|z_1^{(a,n)}-z_2^{(a,n)}|}(0)$ by black (resp., white) paths. Then the hexagons in $a\mathcal{H}$ lying on the boundary of the connected component of $\mathbb{C}\setminus H$ that is connected to infinity form a natural configuration of faces around the circle $\partial B_{\epsilon}(0)$, which we denote by
\[\bs{\eta}^{(n)}= (\eta_1^{(n)},\eta_2^{(n)},\eta_3^{(n)},\eta_4^{(n)}).\]
In particular, on the event 
\[\hat{\mathcal{F}}_{\epsilon}(z_1^{(a,n)},z_2^{(a,n)})\cap \mathring{\mathcal{F}}_{2|z_1^{(a,n)}-z_2^{(a,n)}|,\epsilon}(0),\]
we fix the order of faces $\bs{\eta}^{(n)}= (\eta_1^{(n)},\eta_2^{(n)},\eta_3^{(n)},\eta_4^{(n)})$ so that $\eta_1^{(n)}$ is in the same black cluster as $z_1^{(a,n)}$ and that $\eta_3^{(n)}$ is in the same black cluster as $z_2^{(a,n)}$.

{We now continue the proof and deal first} with the last conditional probability in~\eqref{eqn::four_arm_coupling_argument_aux1}. We write
{\small\begin{equation}\label{eqn::four_arm_coupling_argument_aux2}
\begin{split}
        &\pi_a^{-2}\times \mathbb{P}^a\Big[
z_1^{(a,n)}\xleftrightarrow{B}z_3^{(a,n)}
\centernot{\xleftrightarrow{B}}
z_2^{(a,n)}\xleftrightarrow{B}z_4^{(a,n)}
\,\Big|\,
\hat{\mathcal{F}}_{\epsilon}(z_1^{(a,n)},z_2^{(a,n)})\cap \mathring{\mathcal{F}}_{2|z_1^{(a,n)}-z_2^{(a,n)}|,\epsilon}(0)
\Big]\\
& = \underbrace{\mathbb{P}^a\Big[
z_1^{(a,n)}\xleftrightarrow{B}z_3^{(a,n)}
\centernot{\xleftrightarrow{B}}
z_2^{(a,n)}\xleftrightarrow{B}z_4^{(a,n)}
\,\Big|\,
\hat{\mathcal{F}}_{\epsilon}(z_1^{(a,n)},z_2^{(a,n)})\cap \mathring{\mathcal{F}}_{2|z_1^{(a,n)}-z_2^{(a,n)}|,\epsilon}(0),\; z_j^{(a,n)}\xleftrightarrow{B} \partial B_{\epsilon}(z_j^{(a,n)}), \; \;{j=3,4}\Big]}_{T^{(a,n)}}\\
&\quad \times \pi_a^{-2}\times \mathbb{P}^a\Big[ z_3^{(a,n)}\xleftrightarrow{B} \partial B_{\epsilon}(z_3^{(a,n)})\Big]\times \mathbb{P}^a\Big[ z_4^{(a,n)}\xleftrightarrow{B} \partial B_{\epsilon}(z_4^{(a,n)})\Big].
\end{split}
\end{equation}}It follows from~\cite{GarbanPeteSchrammPivotalClusterInterfacePercolation} (see the first limit in the third displayed equation on page 999) that 
\begin{align*} 
    \lim_{a\to 0}\pi_a^{-2}\times \mathbb{P}^a\Big[ z_3^{(a,n)}\xleftrightarrow{B} \partial B_{\epsilon}(z_3^{(a,n)})\Big]\times \mathbb{P}^a\Big[ z_4^{(a,n)}\xleftrightarrow{B} \partial B_{\epsilon}(z_4^{(a,n)})\Big] =\epsilon^{-\frac{5}{24}}.
\end{align*}

We now analyze the conditional probability $T^{(a,n)}$ on the right hand side of~\eqref{eqn::four_arm_coupling_argument_aux2}. Using the four-arm event coupling argument in~\cite[Proof of Proposition~3.1]{GarbanPeteSchrammPivotalClusterInterfacePercolation} and the one-arm event coupling argument in~\cite[Proof of Lemma~2.1]{Cam23}, one can show the existence of a coupling $\mathring{\mathbb{P}}^a[\cdot]$ between the conditional measures 
\[\mathbb{P}^a\Big[\cdot
\,\Big|\,
\hat{\mathcal{F}}_{\epsilon}(z_1^{(a,n)},z_2^{(a,n)})\cap \mathring{\mathcal{F}}_{2|z_1^{(a,n)}-z_2^{(a,n)}|,\epsilon}(0),\; z_j^{(a,n)}\xleftrightarrow{B} \partial B_{\epsilon}(z_j^{(a,n)}), \; \;{j=3,4}\Big],\quad \text{for }n\geq 1,\]
that is, a joint distribution on $(\Lambda^{(a,1)},\ldots, \Lambda^{(a,n)},\ldots)$ such that 
\[\Lambda^{(a,n)}\sim \mathbb{P}^a\Big[\cdot
\,\Big|\,
\hat{\mathcal{F}}_{\epsilon}(z_1^{(a,n)},z_2^{(a,n)})\cap \mathring{\mathcal{F}}_{2|z_1^{(a,n)}-z_2^{(a,n)}|,\epsilon}(0),\; z_j^{(a,n)}\xleftrightarrow{B} \partial B_{\epsilon}(z_j^{(a,n)}), \; \;j=1,2\Big],\quad \text{for }n\geq 1,\]
with the following property. There exists a constant $c\in (0,\infty)$ that is independent of $a>0$ and $n\geq 1$ and events $\{\mathcal{O}^{(a,n)}\}_{n\geq 1}$ such that the following holds.
\begin{itemize}
    \item The probabilities of the events $\{\mathcal{O}_n\}_{n\geq 1}$ decay to $0$ as $a\to 0$ and then $n\to \infty$:
    \[\mathring{\mathbb{P}}^a[\mathcal{O}^{(a,n)}] \lesssim \Big(\max \{|z_3^{(a,n)}-z_3|, |z_4^{(a,n)}-z_4|, |z_1^{(a,n)}|, |z_2^{(a,n)}|\}\Big)^c;\]
    \item if the event $\mathcal{O}^{(a,n)}$ happens, then 
    \begin{itemize}
        \item for all $n_1,n_2\geq n$, we have $\bs{\eta}^{(n_1)}=\bs{\eta}^{(n_2)}$;
        \item there exist two closed paths of black hexagons, $\mathtt{C}^{(a,n)}_1\subseteq B_{\epsilon/3}(z_3)\setminus B_{2|z_3^{(a,n)}-z_3|}(z_3)$ and $\mathtt{C}^{(a,n)}_2\subseteq B_{\epsilon/3}(z_4)\setminus B_{2|z_4^{(a,n)}-z_4|}(z_4)$, surrounding $z_3$ and $z_4$, respectively, in all $\Lambda^{(a,m)}$ for {all $m\geq n$;}
        \item for any $n_1,n_2\geq n$, the percolation configurations on hexagons that are not surrounded\footnote{Note that the concatenation of the four faces in $\bs{\eta}^{(n)}$ forms a closed path {of hexagons around the origin}.} by any of $\{\mathtt{C}_1^{(a,n)},\mathtt{C}_2^{(a,n)}, \bs{\eta}^{(n)}\}$ are the same under $\Lambda^{(a,n_1)}$ and $\Lambda^{(a,n_2)}$.
    \end{itemize}
\end{itemize}
Note that, on the event $\mathcal{O}^{(a,n)}$, for any $m\geq n$, the event $\{z_1^{(m)}\xleftrightarrow{B}z_3^{(m)}
\centernot{\xleftrightarrow{B}}
z_2^{(m)}\xleftrightarrow{B}z_4^{(m)}\}$ {implies} $\{\mathtt{C}_1^{(a,m)}\xleftrightarrow{B} \eta_1^{(m)}\centernot{\xleftrightarrow{B}} \mathtt{C}_2^{(a,m)}\xleftrightarrow{B}\eta_3^{(m)}\}$. Therefore, for any $n_1,n_2\geq n$, we have 
\begin{align*}
    \Big| T^{(a,n_1)}-T^{(a,n_2)}\Big|\leq \mathring{\mathbb{P}}^a[\mathcal{O}^{(a,n)}]\lesssim \Big(\max \{|z_3^{(a,n)}-z_3|, |z_4^{(a,n)}-z_4|, |z_1^{(a,n)}|, |z_2^{(a,n)}|\}\Big)^c.
\end{align*}
Letting $a\to 0$ shows that $\{\lim_{a\to 0}T^{(a,n)}\}_{n\geq 1}$
is a Cauchy sequence, where the existence of the limits as $a\to 0$ can be shown using the strategy in~\cite[Proof of Theorem~1.5]{Cam23}, without additional essential difficulties.

Next, one can use a similar argument to show that  
\[\Big\{|z_2^{(n)}-z_1^{(n)}|^{\frac{5}{24}}\times \lim_{a\to 0} \pi_a^{-2}\times  \mathbb{P}^a\Big[
 \hat{\mathcal{F}}_{\epsilon}(z_1^{(a,n)},z_2^{(a,n)}) \,\Big|\,
\mathcal{F}_{2|z_1^{(a,n)}-z_2^{(a,n)}|,\epsilon}(0)
\Big]\Big\}_{n\geq 1}\]
 is a Cauchy sequence, where the existence of the limits as $a\to 0$ can also be shown using the strategy in~\cite[Proof of Theorem~1.5]{Cam23}, without additional essential difficulties.

 Combining these conclusions with~\eqref{eqn::four_arm_coupling_argument_aux1} concludes the proof.
\end{proof}

Next, we prove Lemma~\ref{lem::OPE2}.
\begin{proof}[Proof of Lemma~\ref{lem::OPE2}]
We decompose the term $T_m=T_m^{(n)}$ as in~\eqref{eqn::Tm_decom}.

    We begin with the proof of claim~\ref{item::lem::OPE1_1}. 
\begin{itemize}
\item First, we write
\begin{align*}
h(z_1,z_2):= P\Big[z_1\xleftrightarrow{B;B_m}z_2\Big]= P\Big[z_1\xleftrightarrow{B;B_{2^m|z_2-z_1|}\big(\tfrac{z_1+z_2}{2}\big)} z_2\Big].
\end{align*}
It follows from~\cite[Proof of Theorem~1.4]{Cam23} that, for any non-constant M\"obius transformation $\varphi$ which is a composition of translation and rotation, we have 
\begin{align*}
h(\varphi(z_1),\varphi(z_2))=|\varphi'(z_1)|^{-\frac{5}{48}} |\varphi'(z_2)|^{-\frac{5}{48}}\times h(z_1,z_2).
\end{align*}
Therefore, the quantity $h(z_1,z_2)/|z_2-z_1|^{-\frac{5}{24}}$ is a constant, which implies that
\begin{align*}
\frac{P\Big[z_1^{(n_1)}\xleftrightarrow{B;{B^{(n_1)}_m}}z_2^{(n_1)}\Big]}{|z_2^{(n_1)}-z_1^{(n_1)}|^{-\frac{5}{24}}} = \frac{P\Big[z_1^{(n_2)}\xleftrightarrow{B;{B^{(n_2)}_m}}z_2^{(n_2)}\Big]}{|z_2^{(n_2)}-z_1^{(n_2)}|^{-\frac{5}{24}}}.
\end{align*}
\item Second, we write (to simplify the notation, we write $z_1$ for $z_1^{(n)}$, $z_2$ for $z_2^{(n)}$  and $B_m$ for $B_m^{(n)}$ when there is no ambiguity)
\begin{align*}
  & |z_2-z_1|^{-\frac{5}{4}}\times P\Big[\mathcal{F}_{2|z_2-z_1|,2^{m-1}|z_2-z_1|}(\tfrac{z_1+z_2}{2})\big | z_1\xleftrightarrow{B;B_m}z_2\Big] P\Big[\mathcal{F}_{2^m|z_2-z_1|,\epsilon}(\tfrac{z_1+z_2}{2})\Big]\\
  &=\underbrace{\frac{P\Big[\mathcal{F}_{2|z_2-z_1|,2^{m-1}|z_2-z_1|}(\tfrac{z_1+z_2}{2})\big | z_1\xleftrightarrow{B;B_m}z_2\Big]}{P\Big[\mathcal{F}_{2|z_2-z_1|,2^{m-1}|z_2-z_1|}(\tfrac{z_1+z_2}{2})\Big]} }_{T_{m,1}^{(n)}}\times \underbrace{\frac{P\Big[\mathcal{F}_{2|z_2-z_1|,2^{m-1}|z_2-z_1|}(\tfrac{z_1+z_2}{2})\Big]P\Big[\mathcal{F}_{2^m|z_2-z_1|,\epsilon}(\tfrac{z_1+z_2}{2})\Big]}{P\Big[\mathcal{F}_{2|z_2-z_1|,\epsilon}(\tfrac{z_1+z_2}{2})\Big]}}_{T^{(n)}_{m,2}}\\
  &\qquad\times \underbrace{P\Big[\mathcal{F}_{2|z_2-z_1|,\epsilon}(\tfrac{z_1+z_2}{2})\Big]\times |z_2-z_1|^{-\frac{5}{4}}}_{T^{(n)}_{m,3}}.
\end{align*}
As for $h(z_1,z_2)/|z_2-z_1|^{-\frac{5}{24}}$ in the previous step, $T_{m,1}^{(n)}$ is independent of $n$, so that $T_{m,1}^{(n_1)}=T_{m,1}^{(n_2)}$.
The term $T^{(n)}_{m,3}$ is independent of $m$ and satisfies $\lim_{n\to \infty}T_{m,3}^{(n)}\in (0,\infty)$ (see, e.g.,~\cite[Corollary~2.4]{CF24}). For the term $T_{m,2}^{(n)}$, we write
\begin{eqnarray*}
    \frac{1}{T_{m,2}^{(n)}} & = & P\Big[\mathcal{F}_{2|z_2-z_1|,\epsilon}(\tfrac{z_1+z_2}{2}) \big|\mathcal{F}_{2|z_2-z_1|,2^{m-1}|z_2-z_1|}(\tfrac{z_1+z_2}{2}),\; \mathcal{F}_{2^m|z_2-z_1|,\epsilon}(\tfrac{z_1+z_2}{2})\Big] \\
    & = & P\Big[ \mathcal{F}_{2, \epsilon/|z_2-z_1|}(0) \Big|\mathcal{F}_{2,2^{m-1}}(0),\; \mathcal{F}_{2^m, \epsilon/|z_2-z_1|}(0)\Big]\\
    & = & \lim_{a\to 0} \mathbb{P}^a\Big[\mathcal{F}_{2,\epsilon/|z_2^a-z_1^a|}(0) \big| \mathcal{F}_{2,2^{m-1}}(0), \; \mathcal{F}_{2^m, \epsilon/|z_2^a-z_1^a|}(0)\Big],
\end{eqnarray*}
where {the events in the first and second lines can be expressed in terms of the full-plane, nested $\mathrm{CLE}_6$ and therefore} the second equality follows from the conformal invariance of $\mathrm{CLE}_6$~\cite{CamiaNewmanPercolationFull,GwynneMillerQianCLE}.
For each fixed $n$, let $z_1^{(a,n)}$ and $ z_2^{(a,n)}$ be vertices in $a\mathcal{T}$ such that $z_1^{(a,n)}\to z_1^{(n)}$ and $z_2^{(a,n)}\to z_2^{(n)}$ as $a\to 0$. Then, as in the proof of Lemma~\ref{lem::four_arm_coupling_argument}, for each $a>0$, one can couple the two conditional measures
\[\mathbb{P}^a\Big[ \cdot \big|\mathcal{F}_{2,2^{m-1}}(0),\; \mathcal{F}_{2^m, \epsilon/\big|z_2^{(a,n_1)}-z_1^{(a,n_1)}\big|}(0)\Big] \text{ and }\mathbb{P}^a\Big[ \cdot \big|\mathcal{F}_{2,2^{m-1}}(0),\; \mathcal{F}_{2^m, \epsilon/\big|z_2^{(a,n_2)}-z_1^{(a,n_2)}\big|}(0)\Big]\]
and show (by letting $a\to 0$) that there is a constant $c_*\in (0,\infty)$ such that, for all $n_1,n_2\geq 1$ with $\min\{M_{n_1}, M_{n_2}\}\geq m$,
\begin{align} \label{eqn::four_arm_quasi_multiply}
    \Big|T_{m,2}^{(n_1)}-T_{m,2}^{(n_2)}\Big|\lesssim \left(\min\left\{\frac{\epsilon}{|z_2^{(n_1)}-z_1^{(n_1)}|}, \frac{\epsilon}{|z_{2}^{(n_2)}-z_1^{(n_2)}|}\right\}\right)^{-c_*}.
\end{align}
\item Third, we write
\begin{align*}
   & P\Big[z_1\xleftrightarrow[B_{m-1}]{B}z_2,\ z_3\xleftrightarrow[B_m^c]{B}z_4\big|z_1\xleftrightarrow{B;B_m}z_2,\, \mathcal{F}_{2|z_2-z_1|,2^{m-1}|z_2-z_1|}(\tfrac{z_1+z_2}{2}),\, \mathcal{F}_{2^m|z_2-z_1|,\epsilon}(\tfrac{z_1+z_2}{2})\Big]  \\
		& \qquad - P\Big[z_1\xleftrightarrow[B_{m-1}]{B}z_2 \big| z_1\xleftrightarrow{B;B_m}z_2,\, \mathcal{F}_{2|z_2-z_1|,2^{m-1}|z_2-z_1|}(\tfrac{z_1+z_2}{2})\Big]P\Big[z_3\xleftrightarrow[B_m^c]{B}z_4\big|\mathcal{F}_{2^m|z_2-z_1|,\epsilon}(\tfrac{z_1+z_2}{2})\Big]\\
        &\quad=\underbrace{P\Big[z_1\xleftrightarrow[B_{m-1}]{B}z_2 \big| z_1\xleftrightarrow{B;B_m}z_2,\, \mathcal{F}_{2|z_2-z_1|,2^{m-1}|z_2-z_1|}(\tfrac{z_1+z_2}{2})\Big]}_{T_{m,4}^{(n)}}\\
        &\qquad \times \bigg( \underbrace{P\Big[\ z_3\xleftrightarrow[B_m^c]{B}z_4\big|z_1\xleftrightarrow[B_{m-1}]{B;B_m}z_2,\,  \mathcal{F}_{2^m|z_2-z_1|,\epsilon}(\tfrac{z_1+z_2}{2})\Big]}_{T_{m,5}^{(n)}}-\underbrace{P\Big[z_3\xleftrightarrow[B_m^c]{B}z_4\big|\mathcal{F}_{2^m|z_2-z_1|,\epsilon}(\tfrac{z_1+z_2}{2})\Big]}_{T_{m,6}^{(n)}}\bigg).
\end{align*}
As in the analysis of $h(z_1,z_2)/|z_2-z_1|^{-\frac{5}{24}}$ in the first step, one can show that, for any $n_1,n_2\geq 1$ with $\min\{M_{n_1},M_{n_2}\}\geq m$, we have $T_{m,4}^{(n_1)}=T_{m,4}^{(n_2)}$.
 For the term $T_{m,5}^{(n)}$, as in the proof of Lemma~\ref{lem::four_arm_coupling_argument}, for each $a>0$, one can couple the two conditional measures\footnote{Note that the event $\mathcal{F}_{2\big|z^{(a,n_1)}_2-z^{(a,n_1)}_1\big|,2^{m-1}\big|z^{(a,n_1)}_2-z^{(a,n_1)}_1\big|}\Big(\tfrac{z^{(a,n_1)}_1+z^{(a,n_1)}_2}{2}\Big)$ is implied by the event $\{z_1^{(a,n_1)}\xleftrightarrow[B^{(n_1)}_{m-1}]{B;B^{(n_1)}_m}z_2^{(a,n_1)}\}$. The same conclusion holds if $n_1$ is replaced by $n_2$.}
\begin{align*}
&\mathbb{P}^a\Big[\cdot\big|z_1^{(a,n_1)}\xleftrightarrow[B^{(n_1)}_{m-1}]{B;B^{(n_1)}_m}z_2^{(a,n_1)},\, \mathcal{F}_{2\big|z^{(a,n_1)}_2-z^{(a,n_1)}_1\big|,2^{m-1}\big|z^{(a,n_1)}_2-z^{(a,n_1)}_1\big|}\Big(\tfrac{z^{(a,n_1)}_1+z^{(a,n_1)}_2}{2}\Big),\\
&\qquad\qquad\qquad\qquad\qquad\qquad\qquad	\qquad\qquad\qquad\qquad \mathcal{F}_{2^m\big|z^{(a,n_1)}_2-z^{(a,n_1)}_1\big|,\epsilon}\Big(\tfrac{z^{(a,n_1)}_1+z^{(a,n_1)}_2}{2}\Big)\Big]
\end{align*}
and
\begin{align*}
&\mathbb{P}^a\Big[\cdot\big|z_1^{(a,n_2)}\xleftrightarrow[B^{(n_2)}_{m-1}]{B;B^{(n_2)}_m}z_2^{(a,n_2)},\, \mathcal{F}_{2\big|z^{(a,n_2)}_2-z^{(a,n_2)}_1\big|,2^{m-1}\big|z^{(a,n_2)}_2-z^{(a,n_2)}_1\big|}\Big(\tfrac{z^{(a,n_2)}_1+z^{(a,n_2)}_2}{2}\Big),\\
&\qquad\qquad\qquad\qquad\qquad\qquad\qquad	\qquad\qquad\qquad\qquad \mathcal{F}_{2^m\big|z^{(a,n_2)}_2-z^{(a,n_2)}_1\big|,\epsilon}\Big(\tfrac{z^{(a,n_2)}_1+z^{(a,n_2)}_2}{2}\Big)\Big],
\end{align*}
and show (by letting $a\to 0$) that there is a constant $c_*\in (0,\infty)$ such that, for all $n_1,n_2\geq 1$ with $\min \{M_{n_1},M_{n_2}\}\geq 1$,
\[\big|T_{m,5}^{(n_1)}-T_{m,5}^{(n_2)}\big|\lesssim \left(\min\left\{\frac{\epsilon}{\Big|z_2^{(n_1)}-z_1^{(n_1)}\Big|}, \frac{\epsilon}{\Big|z_{2}^{(n_2)}-z_1^{(n_2)}\Big|}\right\}\right)^{-c_*}.\]
Similarly, for all $n_1,n_2\geq 1$ with $\min \{M_{n_1},M_{n_2}\}\geq 1$, we have
\[|T_{m,6}^{(n_1)}-T_{m,6}^{(n_2)}|\lesssim \left(\min\left\{\frac{\epsilon}{\Big|z_2^{(n_1)}-z_1^{(n_1)}\Big|}, \frac{\epsilon}{\Big|z_{2}^{(n_2)}-z_1^{(n_2)}\Big|}\right\}\right)^{-c_*},\]
for some constant $c_*\in (0,\infty)$.

Finally, a standard application of RSW estimates (see, e.g., the proofs of Lemmas 2.1 and 2.2 of~\cite{CamiaNewman2009ising}) shows that 
\[T_{m,j}^{(n)}\asymp 1,\quad \text{for } j\in \{1,2,\ldots,6\}\setminus \{3\},\]
while~\cite[Corollary~2.4]{CF24} shows that $T_{m,3}^{(n)}\asymp 1$.
\end{itemize}
Combining the above observations with~\eqref{eqn::Tm_decom}, we get~\eqref{eqn::generic_m_limit_aux1}. 

We now prove claim~\ref{item::lem::OPE1_2}. First, note that 
\begin{align*}
& \left|1-\frac{P\Big[z_1\xleftrightarrow{B;B_m}z_2\Big]}{\sqrt{C_1}|z_2-z_1|^{-\frac{5}{24}}}\right|=\left|1-\frac{P\Big[z_1\xleftrightarrow{B;B_m}z_2\Big]}{P\Big[z_1\xleftrightarrow{B}z_2\Big]}\right|=\frac{P\Big[z_1\xleftrightarrow[B_m]{B}z_2\Big]}{P\Big[z_1\xleftrightarrow{B}z_2\Big]} \\
& \qquad = \frac{P\Big[z_1\xleftrightarrow[B_m]{B}z_2,\; \mathcal{F}_{2|z_2-z_1|,2^m|z_2-z_1|}(\frac{z_1+z_2}{2})\Big]}{P\Big[z_1\xleftrightarrow{B}z_2\Big]} \\
& \qquad\leq \frac{P\Big[z_1\xleftrightarrow{B}\partial B_{\frac{|z_2-z_1|}{10}}(z_1),\;  z_2\xleftrightarrow{B}\partial B_{\frac{|z_2-z_1|}{10}}(z_2),\; \mathcal{F}_{2|z_2-z_1|,2^m|z_2-z_1|}(\frac{z_1+z_2}{2})\Big]}{P\Big[z_1\xleftrightarrow{B}z_2\Big]} \\
& \qquad = P\Big[ \mathcal{F}_{2|z_2-z_1|,2^m|z_2-z_1|}\Big(\frac{z_1+z_2}{2}\Big)\Big]\times \frac{P\Big[z_1\xleftrightarrow{B}\partial B_{\frac{|z_2-z_1|}{10}}(z_1),\;  z_2\xleftrightarrow{B}\partial B_{\frac{|z_2-z_1|}{10}}(z_2)\Big]}{P\Big[z_1\xleftrightarrow{B}z_2\Big]} \\
& \qquad \lesssim P\Big[\mathcal{F}_{2|z_2-z_1|,2^m|z_2-z_1|}\Big(\frac{z_1+z_2}{2}\Big)\Big]\asymp 2^{-\frac{5}{4}m},
\end{align*}
where the inequality in the last line follows from a standard application of RSW estimates (see, e.g., the proofs of Lemmas 2.1 and 2.2 of~\cite{CamiaNewman2009ising}), and the last estimate follows from ~\cite[Corollary~2.4]{CF24}. The remaining terms in the decomposition~\eqref{eqn::Tm_decom} can be treated in a similar way as in the proof of claim $1$. We then get~\eqref{eqn::generic_m_limit_aux2}, as desired.
\end{proof}
We end this section with the following lemma, which will be used in Section~\ref{subsec::mixed_cor} below. 
\begin{lemma}\label{lem::limit_two_point}
Let
\begin{equation*}
    P\Big[z_1\xleftrightarrow{B;B_m}z_2 \big| \mathcal{F}_{2|z_2-z_1|,2^{m-1}|z_2-z_1|}\Big(\frac{z_1+z_2}{2}\Big)\Big] := \frac{P\Big[z_1\xleftrightarrow{B;B_m}z_2, \mathcal{F}_{2|z_2-z_1|,2^{m-1}|z_2-z_1|}\big(\frac{z_1+z_2}{2}\big)\Big]}{P\Big[\mathcal{F}_{2|z_2-z_1|,2^{m-1}|z_2-z_1|}\big(\frac{z_1+z_2}{2}\big)\Big]}.
\end{equation*}
Then, there exists a constant $C_*\in (0,\infty)$ such that
\begin{align} \label{eqn::limit_two_point}
\lim_{m\to \infty}	\lim_{z_1,z_2\to z} |z_2-z_1|^{\frac{5}{24}}\times P\Big[z_1\xleftrightarrow{B;B_m}z_2 \big| \mathcal{F}_{2|z_2-z_1|,2^{m-1}|z_2-z_1|}\Big(\frac{z_1+z_2}{2}\Big)\Big]=C_*.
	\end{align}
\end{lemma}
\begin{proof}
We recall that
\begin{equation*}
    P\Big[z_1\xleftrightarrow{B;B_m}z_2, \mathcal{F}_{2|z_2-z_1|,2^{m-1}|z_2-z_1|}\Big(\frac{z_1+z_2}{2}\Big)\Big] := \lim_{a \to 0} \pi_a^{-2} \mathbb{P}^a\Big[z^a_1\xleftrightarrow{B;B_m}z^a_2, \mathcal{F}_{2|z^a_2-z^a_1|,2^{m-1}|z^a_2-z^a_1|}\Big(\frac{z^a_1+z^a_2}{2}\Big)\Big]
\end{equation*}
and
\begin{equation*}
    P\Big[\mathcal{F}_{2|z_2-z_1|,2^{m-1}|z_2-z_1|}\Big(\frac{z_1+z_2}{2}\Big)\Big] := \lim_{a \to 0} \mathbb{P}^a\Big[\mathcal{F}_{2|z^a_2-z^a_1|,2^{m-1}|z^a_2-z^a_1|}\Big(\frac{z^a_1+z^a_2}{2}\Big)\Big].
\end{equation*}

As in the analysis of $h(z_1,z_2)/|z_2-z_1|^{-\frac{5}{24}}$ in the first step of the proof of Lemma~\ref{lem::OPE2}, we see that
\begin{align*}
& |z_2-z_1|^{\frac{5}{24}}\times P\Big[z_1\xleftrightarrow{B;B_m}z_2 \big| \mathcal{F}_{2|z_2-z_1|,2^{m-1}|z_2-z_1|}\Big(\frac{z_1+z_2}{2}\Big)\Big] \\
& \qquad = {\sqrt{C_1}} \, \frac{P\Big[z_1\xleftrightarrow{B;B_m}z_2 \big| \mathcal{F}_{2|z_2-z_1|,2^{m-1}|z_2-z_1|}(\frac{z_1+z_2}{2})\Big]}{P\Big[z_1\xleftrightarrow{B}z_2\Big]} = {\sqrt{C_1}} \, \frac{P\Big[-1\xleftrightarrow{B;B_{2^{m+1}}(0)}1 \big| \mathcal{F}_{4,2^m}(0)\Big]}{P\Big[-1\xleftrightarrow{B}1\Big]}
\end{align*}
does not depend on $z_1$ and $z_2$. As in the proof of Lemma~\ref{lem::four_arm_coupling_argument}, for each $a>0$, one can couple the measures $\mathbb{P}^a\Big[\cdot\big| \mathcal{F}_{4,2^m}(0)\Big]$ for $m\geq 1$ and show (by letting $a\to 0$) that
\[ \left\{ \frac{P\Big[-1\xleftrightarrow{B;B_m}1 \big| \mathcal{F}_{4,2^m}(0)\Big]}{P\Big[-1\xleftrightarrow{B}1\Big]} \right\}_{m=1}^{\infty}
\]
is a Cauchy sequence. We then get~\eqref{eqn::limit_two_point} for some constant $C_*\in [0,\infty)$. The fact that $C_*>0$ follows from a standard application of RSW estimates (see, e.g., the proofs of Lemmas 2.1 and 2.2 of~\cite{CamiaNewman2009ising}).
\end{proof}

\section{Correlations of the energy field $\phi$ and the density field $\psi$}
\subsection{Mixed correlations of $\phi$ and $\psi$} \label{subsec::mixed_cor}

For $z\in \mathbb{C}$ and $0<r<R$, recall that $\mathcal{F}_{r,R}(z)$ is the event that there exist four paths with alternating labels (black/white) crossing the annulus $B_{R}(z)\setminus B_{r}(z)$. For $z^a\in a\mathcal{T}$ and $r>0$, define
\begin{align*}
 	\mathcal{F}_{r}(z^a):=\{z^a\xleftrightarrow{BWBW}\partial B_{r}(z^a)\}\cap \{z^a \text{ is white}\},
\end{align*}
where $\{z^a\xleftrightarrow{BWBW}\partial B_{r}(z^a)\}$ is the event that there exist four paths with alternating labels (black/white) connecting $z^a$ to $\partial B_{r}(z^a)$.
According to~\cite[Theorem~1.3]{SharpArmExponent},
\begin{align} \label{eqn::four_arm_proba}
 \lim_{a\to 0}	a^{-\frac{5}{4}} \, \mathbb{P}^a\Big[\mathcal{F}_r(z^a)\Big] \in (0,\infty). 
\end{align}

 \begin{proof}[Proof of~\eqref{eqn::three_point_energy_density} in Theorem~\ref{thm::energy_spin}]
A simple calculation gives
\begin{align} \label{eqn::energy_density_three_express}
\begin{split}
	 	\langle \mathcal{E}_{z_1^a}S_{z_2^a}S_{z_3^a}\rangle^a=& \, \langle S_{z_1^{(a,-)}}S_{z_1^{(a,+)}}S_{z_2^a}S_{z_3^a}\rangle^a-\langle S_{z_1^{(a,-)}}S_{z_1^{(a,+)}}\rangle^a\langle S_{z_2^a}S_{z_3^a}\rangle^a\\
	=&\,\underbrace{\mathbb{P}^a\Big[z_1^{(a,-)}\xleftrightarrow{B}z_2^a\centernot{\xleftrightarrow{B}}z_1^{(a,+)}\xleftrightarrow{B}z_3^a\Big]}_{R_1^a}+\underbrace{\mathbb{P}^a\Big[z_1^{(a,-)}\xleftrightarrow{B}z_3^a\centernot{\xleftrightarrow{B}}z_1^{(a,+)}\xleftrightarrow{B}z_2^a\Big]}_{R_2^a}\\
	&+\underbrace{\mathbb{P}^a\Big[z_1^{(a,-)}\xleftrightarrow{B}z_1^{(a,+)},\; z_2^a\xleftrightarrow{B}z_3^a\Big]- \mathbb{P}^a\Big[z_1^{(a,-)}\xleftrightarrow{B}z_1^{(a,+)}\Big]\mathbb{P}^a\Big[z_2^a\xleftrightarrow{B}z_3^a\Big]}_{R_3^a}.
\end{split}
\end{align}
Choose $\epsilon>0$ such that $B_{2\epsilon}(z_i^a)\cap B_{2\epsilon}(z_j^a)=\emptyset$ for $1\leq i<j\leq 3$. 

For the term $R_1^a$ (resp., $R_2^a$), note that the event $\{z_1^{(a,-)}\xleftrightarrow{B}z_2^a\centernot{\xleftrightarrow{B}}z_1^{(a,+)}\xleftrightarrow{B}z_3^a\}$ (resp., $\{z_1^{(a,-)}\xleftrightarrow{B}z_3^a\centernot{\xleftrightarrow{B}}z_1^{(a,+)}\xleftrightarrow{B}z_2^a\}$) implies the event $\mathcal{F}_{\epsilon}(z_1^a)$ for small enough $a>0$. We then conclude from~\eqref{eqn::four_arm_proba} that 
 \begin{align*}
 	R_1^a \lesssim a^{\frac{5}{4}} \quad\text{and}\quad R_2^a\lesssim a^{\frac{5}{4}}. 
 \end{align*}

The analysis of the term $R_3^a$ is similar to that of the difference in~\eqref{eqn::difference_four_spin} presented in Section~\ref{sec::sketch_four_spin}. Let $M$ be the largest integer such that $2^M (2a)\leq \epsilon$; note that $M$ is of order $-\log a$. For $1\leq m\leq M$, define $B_m^a= \{w\in \mathbb{C}: |w-z_1^a|\leq 2^m (2a)\}$. Proceeding as in the derivation of~\eqref{eqn::four_spin_decom}, one can show that 
\begin{align}
\begin{split}
R_3^a
&=\underbrace{\mathbb{P}^a\Big[z_1^{(a,-)}\xleftrightarrow{B;B_1^a}z_1^{(a,+)},\; z_2^a\xleftrightarrow[(B_1^a)^c]{B}z_3^a\Big]-\mathbb{P}^a\Big[z_1^{(a,-)}\xleftrightarrow{B;B_1^a}z_1^{(a,+)}\Big]\mathbb{P}^a\Big[ z_2^a\xleftrightarrow[(B_1^a)^c]{B}z_3^a\Big]}_{T_1^a}\\
& +\sum_{m=2}^M\bigg(\underbrace{\mathbb{P}^a\Big[z_1^{(a,-)}\xleftrightarrow[B_{m-1}^a]{B;B_m^a}z_1^{(a,+)},\; z_2^a\xleftrightarrow[(B_m^a)^c]{B}z_3^a\Big]-\mathbb{P}^a\Big[z_1^{(a,-)}\xleftrightarrow[B_{m-1}^a]{B;B_m^a}z_1^{(a,+)}\Big]\mathbb{P}^a\Big[z_2^a\xleftrightarrow[(B_m^a)^c]{B}z_3^a\Big]}_{T_m^a}\bigg)\\
&+\underbrace{\mathbb{P}^a\Big[z_1^{(a,-)}\xleftrightarrow[B_M^a]{B}z_1^{(a,+)},\; z_2^a\xleftrightarrow{B}z_3^a\Big]-\mathbb{P}^a\Big[z_1^{(a,-)}\xleftrightarrow[B_M^a]{B}z_1^{(a,+)}\Big]\mathbb{P}^a\Big[ z_2^a\xleftrightarrow{B}z_3^a\Big]}_{T_{M+1}^a} ,
\end{split} \label{def::T^a_m}
\end{align}
which is the analog of~\eqref{eqn::four_spin_decom} with $(z_1,z_2,z_3,z_4)$ replaced by $(z_1^{(a,-)},z_1^{(a,+)},z_3^a,z_4^a)$.
 
For small enough $a$, each of the events $\{z_2^a\xlongleftrightarrow[(B_1^a)^c]{B}z_3^a\}$ and $\{z_1^{(a,-)}\xleftrightarrow[(B_M)^a]{B} z_1^{(a,+)}\}$, contained in the terms $T_1^a$ and $T_{M+1}^a$, respectively, implies $\mathcal{F}_{4a, 2^M(2a)}(z_1^a)$. 
 Consequently, thanks to~\eqref{eqn::four_arm_proba}, we have 
 \begin{align}
 	T_1^a \lesssim a^{\frac{5}{4}}\quad \text{and} \quad T_{M+1}^a\lesssim a^{\frac{5}{4}}. 
 \end{align}

Going back to $R^a_1$, for $0<a<\epsilon$ sufficiently small, we can write
\begin{align*}
    & R^a_1 = \mathbb{P}^a\Big[z_1^{(a-)}\xleftrightarrow{B}z_2^a\centernot{\xleftrightarrow{B}}z_1^{(a,+)}\xleftrightarrow{B}z_3^a \vert \mathcal{F}_{\epsilon}(z^a_1), \;z_2^a\xleftrightarrow{B}\partial B_{\epsilon}(z_2^a),\; z_3^a\xleftrightarrow{B}\partial B_{\epsilon}(z_3^a) \Big] \\
    & \qquad \quad \times \mathbb{P}^a\Big[\mathcal{F}_{\epsilon}(z_1^a)\Big] \mathbb{P}^a\Big[z_2^a\xleftrightarrow{B}\partial B_{\epsilon}(z_2^a)\Big] \mathbb{P}^a\Big[z_3^a\xleftrightarrow{B}\partial B_{\epsilon}(z_3^a)\Big].
\end{align*}
It follows from~\eqref{eqn::four_arm_proba} that
\begin{align*}
    \lim_{a \to 0} \big(a^{\frac{5}{4}}\vert\log a \vert\big)^{-1} \pi_a^{-2} R^a_1 = 0. 
\end{align*}
Similarly,
\begin{align*}
    \lim_{a \to 0} \big(a^{\frac{5}{4}}\vert\log a \vert\big)^{-1} \pi_a^{-2}  R^a_2 = 0, \; \lim_{a \to 0} \big(a^{\frac{5}{4}}\vert\log a \vert\big)^{-1} \pi_a^{-2} T^a_1 = 0, \; \lim_{a \to 0} \big(a^{\frac{5}{4}}\vert\log a \vert\big)^{-1} \pi_a^{-2} T^a_{M+1} = 0.
\end{align*}
Therefore we have
\begin{eqnarray} \label{eq::lim=sum}
    \lim_{a \to 0} \big(a^{\frac{5}{4}}\vert\log a \vert\big)^{-1} \pi_a^{-2}\langle \mathcal{E}_{z_1^a}S_{z_2^a}S_{z_3^a}\rangle^a & = & \lim_{a \to 0} \big(a^{\frac{5}{4}}\vert\log a \vert\big)^{-1} \pi_a^{-2} \Big(R^a_1 +  R^a_2 + T^a_1 + T^a_{M+1} + \sum_{m=2}^M T^a_m \Big) \nonumber \\
    & = & \lim_{a \to 0} \big(a^{\frac{5}{4}}\vert\log a \vert\big)^{-1} \pi_a^{-2} \sum_{m=2}^{M} T^a_m.
\end{eqnarray}

 Lemma~\ref{lem::three_point_spin_energy_aux} below shows that
\begin{align} \label{eqn::existence-limit-sum}
\lim_{a\to 0} \big(a^{\frac{5}{4}}|\log a|\big)^{-1}\pi_a^{-2} \sum_{m=2}^MT_m^a \in (0,\infty)
\end{align}
and, moreover, that there exists a constant $C>0$ such that
\begin{align} \label{eq::lim-sum-T_m}
	\lim_{a\to 0} \big(a^{\frac{5}{4}}|\log a|\big)^{-1} \pi_a^{-2} \, \sum_{m=2}^M T_m^a=C \lim_{z,z'\to z_1} \frac{P\Big[z\xleftrightarrow{B}z',\; z_2\xleftrightarrow{B}z_3\Big]-P\Big[z\xleftrightarrow{B}z'\Big]P\Big[z_2\xleftrightarrow{B}z_3\Big]}{|z'-z|^{\frac{25}{24}}|\log |z'-z||},
\end{align}
where the existence of the limit on the right-hand side of~\eqref{eq::lim-sum-T_m} follows from~\eqref{eqn::OPE1} and~\eqref{eqn::spin_four_connec_discrete}.

{Using~\eqref{eqn::OPE2},} \eqref{eq::4pf-connection_probs}, \eqref{eq::13_24} and~\eqref{eqn::difference_four_spin}, we can write
\begin{align}
    {-C_1C_L F(z_1,z_2,z_3)} & = \lim_{z,z' \to z_1}\frac{\langle\psi(z)\psi(z')\psi(z_2)\psi(z_3)\rangle - C_1 \vert z-z' \vert^{-\frac{5}{24}} \vert z_3-z_2 \vert^{-\frac{5}{24}}}{|z'-z|^{\frac{25}{24}}\vert\log\vert z'-z\vert\vert} \label{eq::top-lim} \\
    & = \lim_{z,z'\to z_1} \frac{P\Big[z\xleftrightarrow{B}z',\; z_2\xleftrightarrow{B}z_3\Big]-P\Big[z\xleftrightarrow{B}z'\Big]P\Big[z_2\xleftrightarrow{B}z_3\Big]}{|z'-z|^{\frac{25}{24}}|\log |z'-z||}\label{eq::bottom-lim}.
\end{align}
{Comparing \eqref{eq::bottom-lim} with~\eqref{eq::lim-sum-T_m} and using \eqref{eq::lim=sum} gives~\eqref{eqn::three_point_energy_density}, as we set out to prove.}
\end{proof}

\begin{lemma} \label{lem::three_point_spin_energy_aux}
Consider the terms $T_m^a$ defined in~\eqref{def::T^a_m} above.
Then we have~\eqref{eqn::existence-limit-sum} and~\eqref{eq::lim-sum-T_m}. 
\end{lemma} 
\begin{proof}
We adopt the same notation as in the proof of~\eqref{eqn::three_point_energy_density} just above.
Fix $y\in (0,1)$ and $m_y\in \{\lfloor My \rfloor, \lceil My\rceil\}$. 
Let $m\in \{2,\ldots,M\}$ and consider the term $T_m^a$. On the one hand, the event $\{z_1^{(a,-)}\xlongleftrightarrow[B_{m-1}^a]{B}z_1^{(a,+)}\}$ implies $\mathcal{F}_{2a, 2^{m}a}(z_1^a)$. On the other hand, the event $\{z_3\xlongleftrightarrow[(B^a_{m})^c]{B}z_4\}$ implies $\mathcal{F}_{2^{m+1}a,2^{M+1}a}(z_1^a)$. Note also that the events $\mathcal{F}_{2a,2^m a}(z_1^a)$ and $\mathcal{F}_{2^{m+1}a,2^{M+1}a}(z_1^a)$ are independent. Consequently, we can write 
\begin{align} \label{eqn::T^a_m}
\begin{split}
 	&T_{m_y}^a\\
 	&=\bigg(\mathbb{P}^a\Big[z_1^{(a,-)}\xleftrightarrow[B_{m_y-1}^a]{B}z_1^{(a,+)},\; z_2^a\xleftrightarrow[(B_{m_y}^a)^c]{B}z_3^a\big| z_1^{(a,-)} \xleftrightarrow{B;B_{m_y}^a} z_1^{(a,+)},\;   \mathcal{F}_{2a,2^{m_y}a}(z_1^a),\; \mathcal{F}_{2^{m_y+1}a,\epsilon}(z_1^a)\Big]  \\
 	&\qquad-\mathbb{P}^a\Big[z_1^{(a,-)}\xleftrightarrow[B_{m_y-1}^a]{B}z_1^{(a,+)} \big| z_1^{(a,-)} \xleftrightarrow{B;B_{m_y}^a} z_1^{(a,+)},\; \mathcal{F}_{2a,2^{m_y}a}(z_1^a) \Big]\mathbb{P}^a\Big[z_2^a\xleftrightarrow[(B_{m_y}^a)^c]{B}z_3^a \big| \mathcal{F}_{2^{m_y+1}a,\epsilon}(z_1^a)\Big]\bigg)\\
 	&\qquad	\times \mathbb{P}^a\Big[z_1^{(a,-)} \xleftrightarrow{B;B_{m_y}^a} z_1^{(a,+)}\big| \mathcal{F}_{2a,2^{m_y} a}(z_1^a)\Big]\times \mathbb{P}^a\Big[\mathcal{F}_{2a,2^{m_y}a}(z_1^a)\Big]\times  \mathbb{P}^a\Big[\mathcal{F}_{2^{m_y+1}a,\epsilon}(z_1^a)\Big].
\end{split}
\end{align}
We then analyze the terms in~\eqref{eqn::T^a_m} separately, using techniques similar to those employed in the proof of Lemma~\ref{lem::OPE2}.

First, we write 
\begin{align*}
	\begin{split} 
		&a^{-\frac{5}{4}}\times \mathbb{P}^a\Big[\mathcal{F}_{2a,2^{m_y}a}(z_1^a)\Big]\times  \mathbb{P}^a\Big[\mathcal{F}_{2^{m_y+1}a,\epsilon}(z_1^a)\Big]\\
		&\qquad =\Big(a^{-\frac{5}{4}}\times \mathbb{P}^a\left[\mathcal{F}_{\epsilon}(z_1^a)\right]\Big) \times \underbrace{\frac{\mathbb{P}^a\Big[\mathcal{F}_{2a,2^{m_y}a}(z_1^a)\Big]\times  \mathbb{P}^a\Big[\mathcal{F}_{2^{m_y+1}a,\epsilon}(z_1^a)\Big]}{\mathbb{P}^a\Big[\mathcal{F}_{\epsilon}(z_1^a)\Big]}}_{T_{m_y,1}^a}.
	\end{split}
\end{align*}
We can use~\eqref{eqn::four_arm_proba} and \cite[Proposition 4.9]{GarbanPeteSchrammPivotalClusterInterfacePercolation} (see also~\cite[Eq.~(2.4)]{CF24}) to conclude that
\begin{align} \label{eqn::C_1^*}
	\lim_{a\to 0}a^{-\frac{5}{4}}\times \mathbb{P}^a\left[\mathcal{F}_{\epsilon}(z_1^a)\right] =\lim_{a\to 0} a^{-\frac{5}{4}}\times \mathbb{P}^a\left[\mathcal{F}_{1}(z_1^a)\right]\times \frac{\mathbb{P}^a\left[\mathcal{F}_{\epsilon}(z_1^a)\right]}{\mathbb{P}^a\left[\mathcal{F}_{1}(z_1^a)\right]}=C_1^* \epsilon ^{-\frac{5}{4}}
\end{align}
for some constant $C_1^*\in (0,\infty)$. For the term $T_{m_y,1}^a$, we have
\begin{align*}
	\frac{1}{T_{m_y,1}^a}=\mathbb{P}^a\left[\mathcal{F}_{\epsilon}(z_1^a)\big| \mathcal{F}_{2a, 2^{m_y}a}(z_1^a),\; \mathcal{F}_{2^{m_y+1}a,\epsilon}(z_1^a)\right]
\end{align*}
Let $0<\eta\ll 1$. As in the proof of Lemma~\ref{lem::four_arm_coupling_argument}, one can find a coupling between the two conditional measures
\[\mathbb{P}^a\left[\,\cdot\,\big| \mathcal{F}_{2a, 2^{m_y}a}(z_1^a),\; \mathcal{F}_{2^{m_y+1}a,\epsilon}(z_1^a)\right]\text{ and }\mathbb{P}^a\left[\,\cdot\,\big| \mathcal{F}_{\eta2^{m_y}a, 2^{m_y}a}(z_1^a),\; \mathcal{F}_{2^{m_y+1}a,\;2^{m_y}a/\eta}(z_1^a)\right]\]
and show that 
\begin{align} \label{eqn::four_arm_events_estimate}
&\Big| \mathbb{P}^a\left[\mathcal{F}_{\epsilon}(z_1^a)\big| \mathcal{F}_{2a, 2^{m_y}a}(z_1^a),\; \mathcal{F}_{2^{m_y+1}a,\epsilon}(z_1^a)\right]\notag\\
&\qquad\qquad-\mathbb{P}^a\left[\mathcal{F}_{\eta 2^{m_y}a, 2^{m_y}a/\eta}(z_1^a)\big| \mathcal{F}_{\eta2^{m_y}a, 2^{m_y}a}(z_1^a),\; \mathcal{F}_{2^{m_y+1}a,\;2^{m_y}a/\eta}(z_1^a)\right]\Big|\lesssim \left(\min\left\{\frac{1}{\eta}, 2^{M-m_y},2^{-m_y}\right\}\right)^{-c_*},
\end{align}
for some constant $c_*\in (0,\infty)$. Note that, as $a\to 0$, the annulus
\[
B_{2^{m_y}a/\eta}(z_1^a)\setminus B_{\eta 2^{m_y}a}(z_1^a)
\]
that appears in the crossing event $\mathcal{F}_{\eta 2^{m_y}a,\, 2^{m_y}a/\eta}(z_1^a)$ shrinks to the single point $z_1^a$.
To handle this issue, we introduce an auxiliary scale, $a_*:=2^{-m_y}$, which also tends to $0$ as $a\to 0$. Rescaling the triangular lattice by $a_*$, we get
\begin{align*}
&\lim_{a\to 0}\mathbb{P}^a\left[\mathcal{F}_{\eta 2^{m_y}a, 2^{m_y}a/\eta}(z_1^a)\big| \mathcal{F}_{\eta2^{m_y}a, 2^{m_y}a}(z_1^a),\; \mathcal{F}_{2^{m_y+1}a,\;2^{m_y}a/\eta}(z_1^a)\right]\\
&\quad=\lim_{a_*\to 0} \mathbb{P}^{a_*}\left[\mathcal{F}_{\eta,1/\eta}(z_1^{a_*})\big| \mathcal{F}_{\eta,1}(z_1^{a_*}),\; \mathcal{F}_{2, 1/\eta}(z_1^{a_*})\right]=:P\left[\mathcal{F}_{\eta,1/\eta}(z_1)\big| \mathcal{F}_{\eta,1}(z_1),\; \mathcal{F}_{2,1/\eta}(z_1)\right],
\end{align*}
where we have used the convergence of percolation interfaces~\cite{CamiaNewmanPercolationFull} to get the existence of the second limit. Let $\{\eta_n\}_{n=1}^{\infty}$ satisfy $\eta_n\to 0$ as $n\to\infty$. Then, for each $a>0$, as in the proof of Lemma~\ref{lem::four_arm_coupling_argument}, one can couple the measures $\mathbb{P}^a\Big[\cdot\big| \mathcal{F}_{\eta_n,1}(z_1^a),\, \mathcal{F}_{2,1/\eta_n}(z_1^a)\Big]$ for different values of $n$, and show (by letting $a\to 0$) that $\{P\left[\mathcal{F}_{\eta_n,1/\eta_n}(z_1)\big| \mathcal{F}_{\eta_n,1}(z_1),\; \mathcal{F}_{2,1/\eta_n}(z_1)\right]\}_{n=1}^{\infty}$ is a Cauchy sequence. Combining this with the estimate~\eqref{eqn::four_arm_events_estimate} implies the existence of a constant $C_2^{*}\in (0,\infty)$ such that 
\begin{align*} 
	\lim_{a\to 0} \frac{\mathbb{P}^a\Big[\mathcal{F}_{2a,2^{m_y}a}(z_1^a)\Big]\times  \mathbb{P}^a\Big[\mathcal{F}_{2^{m_y+1}a,\epsilon}(z_1^a)\Big]}{\mathbb{P}^a\Big[\mathcal{F}_{\epsilon}(z_1^a)\Big]}=C_2^*.
\end{align*}
Therefore, using~\eqref{eqn::C_1^*}, we have
\begin{align} \label{eqn::four_arm_quasi}
    \lim_{a\to 0}a^{-\frac{5}{4}}\times \mathbb{P}^a\Big[\mathcal{F}_{2a,2^{m_y}a}(z_1^a)\Big]\times  \mathbb{P}^a\Big[\mathcal{F}_{2^{m_y+1}a,\epsilon}(z_1^a)\Big]= C_1^*C_2^*\epsilon^{-\frac{5}{4}}. 
\end{align}

Second, note that 
\begin{align*}
	\mathbb{P}^a\Big[z_1^{(a,-)} \xleftrightarrow{B;B_{m_y}^a} z_1^{(a,+)}\big| \mathcal{F}_{2a,2^{m_y} a}(z_1^a)\Big]=\mathbb{P}^1\Big[-1\xleftrightarrow{B;B^1_{2^{m_y+1}}(0)}1 \big| \mathcal{F}_{2,2^{m_y}}(0) \Big].
\end{align*}
Since $m_y\to \infty$ as $a\to 0$, as in the proof of Lemma~\ref{lem::four_arm_coupling_argument}, one can couple the measures $\mathbb{P}^1\left[\,\cdot\, \big| \mathcal{F}_{2,2^{m_y}}(0)\right]$ for different values of $m_y$ and show that 
\begin{equation}
	\begin{split} \label{eqn::lim=C^*_3}
		\lim_{a\to 0}\mathbb{P}^a\Big[z_1^{(a,-)} \xleftrightarrow{B;B_{m_y}^a} z_1^{(a,+)}\big| \mathcal{F}_{2a,2^{m_y} a}(z_1^a)\Big]=\lim_{m_y\to \infty}\mathbb{P}^1\Big[-1\xleftrightarrow{B;B^1_{2^{m_y+1}}(0)}1 \big| \mathcal{F}_{2,2^{m_y}}(0) \Big]=C_3^{*},
	\end{split}
\end{equation}
for some constant $C_3^*\in (0,\infty)$.

Third, we claim that
\begin{align} \label{eqn::change_of_limits}
	&	\lim_{a\to 0}	\pi_a^{-2}\times \bigg(\mathbb{P}^a\Big[z_1^{(a,-)}\xleftrightarrow[B_{m_y-1}^a]{B}z_1^{(a,+)},\; z_2^a\xleftrightarrow[(B_{m_y}^a)^c]{B}z_3^a\big| z_1^{(a,-)} \xleftrightarrow{B;B_{m_y}^a} z_1^{(a,+)},\;   \mathcal{F}_{2a,2^{m_y}a}(z_1^a),\; \mathcal{F}_{2^{m_y+1}a,\epsilon}(z_1^a)\Big] \notag \\
	&-\mathbb{P}^a\Big[z_1^{(a,-)}\xleftrightarrow[B_{m_y-1}^a]{B}z_1^{(a,+)} \big| z_1^{(a,-)} \xleftrightarrow{B;B_{m_y}^a} z_1^{(a,+)},\; \mathcal{F}_{2a,2^{m_y}a}(z_1^a) \Big]\mathbb{P}^a\Big[z_2^a\xleftrightarrow[(B_{m_y}^a)^c]{B}z_3^a \big| \mathcal{F}_{2^{m_y+1}a,\epsilon}(z_1^a)\Big]\bigg)\notag\\
	&=C_4^*\lim_{m\to \infty} \lim_{z,z'\to z_1} \bigg(P\Big[z\xleftrightarrow[\hat{B}_{m-1}]{B}z',\ z_2\xleftrightarrow[\hat{B}_{m}^c]{B}z_3\big|z\xleftrightarrow{B;\hat{B}_{m}}z',\, \mathcal{F}_{2|z'-z|,2^{m-1}|z'-z|}(\tfrac{z+z'}{2}),\, \mathcal{F}_{2^{m}|z'-z|,\epsilon}(\tfrac{z+z'}{2})\Big]  \notag\\
	& \quad - P\Big[z\xleftrightarrow[\hat{B}_{m-1}]{B}z' \big| z\xleftrightarrow{B;\hat{B}_{m}}z',\, \mathcal{F}_{2|z'-z|,2^{m-1}|z'-z|}(\tfrac{z+z'}{2})\Big]P\Big[z_2\xleftrightarrow[\hat{B}_{m}^c]{B}z_3\big|\mathcal{F}_{2^{m}|z'-z|,\epsilon}(\tfrac{z+z'}{2})\Big]\bigg),
\end{align}
for some constant $C_4^*\in(0,\infty)$, where $\hat{B}_{m}:=B_{2^m|z'-z|}(\frac{z+z'}{2})$ for $1\leq m\leq \hat{M}$ and $\hat{M}$ is the largest integer such that $2^{\hat{M}}|z'-z|\leq \epsilon$.
To prove~\eqref{eqn::change_of_limits}, we write
\begin{align}
\begin{split} \label{eqn::decomp-T234}
		& \lim_{a\to 0}	\pi_a^{-2}\times \bigg(\mathbb{P}^a\Big[z_1^{(a,-)}\xleftrightarrow[B_{m_y-1}^a]{B}z_1^{(a,+)},\; z_2^a\xleftrightarrow[(B_{m_y}^a)^c]{B}z_3^a\big| z_1^{(a,-)} \xleftrightarrow{B;B_{m_y}^a} z_1^{(a,+)},\;   \mathcal{F}_{2a,2^{m_y}a}(z_1^a),\; \mathcal{F}_{2^{m_y+1}a,\epsilon}(z_1^a)\Big] \\
	& \quad -\mathbb{P}^a\Big[z_1^{(a,-)}\xleftrightarrow[B_{m_y-1}^a]{B}z_1^{(a,+)} \big| z_1^{(a,-)} \xleftrightarrow{B;B_{m_y}^a} z_1^{(a,+)},\; \mathcal{F}_{2a,2^{m_y}a}(z_1^a) \Big]\mathbb{P}^a\Big[z_2^a\xleftrightarrow[(B_{m_y}^a)^c]{B}z_3^a \big| \mathcal{F}_{2^{m_y+1}a,\epsilon}(z_1^a)\Big]\bigg) \\
	& \quad = \lim_{a\to 0} \pi_a^{-2}\times \underbrace{\mathbb{P}^a\Big[z_1^{(a,-)}\xleftrightarrow[B_{m_y-1}^a]{B}z_1^{(a,+)} \big| z_1^{(a,-)} \xleftrightarrow{B;B_{m_y}^a} z_1^{(a,+)},\; \mathcal{F}_{2a,2^{m_y}a}(z_1^a) \Big]}_{T_{m_y,2}^a}
	\\
	&\qquad\quad\times\bigg(\underbrace{\mathbb{P}^a\Big[ z_2^a\xleftrightarrow[(B_{m_y}^a)^c]{B}z_3^a\big| z_1^{(a,-)} \xleftrightarrow[B_{m_y-1}^a]{B;B_{m_y}^a} z_1^{(a,+)},\;  \mathcal{F}_{2^{m_y+1}a,\epsilon}(z_1^a)\Big]}_{T_{m_y,3}^a} - \underbrace{\mathbb{P}^a\Big[z_2^a\xleftrightarrow[(B_{m_y}^a)^c]{B}z_3^a \big| \mathcal{F}_{2^{m_y+1}a,\epsilon}(z_1^a)\Big]}_{T_{m_y,4}^a} \bigg)
\end{split}
\end{align}
and proceed to analyze the terms of the last expression.
\begin{itemize}
	\item For the term $T_{m_y,2}^a$, note that, by scale and translation invariance,
	\begin{align*}
		T_{m_y,2}^a=\mathbb{P}^1\Big[-1\xleftrightarrow[{B_{2^{m_y}}(0)}]{B} 1\big|-1\xleftrightarrow{B;{B_{2^{m_y+1}}(0)}} 1,\; \mathcal{F}_{2,2^{m_y}}(0)\ \Big].
	\end{align*}
	Since {$m_y\to \infty$} as $a\to 0$, as in the proof of Lemma~\ref{lem::four_arm_coupling_argument}, one can find a coupling between the measures $\mathbb{P}^1\left[\,\cdot\, \big|-1\xleftrightarrow{B;{B_{2^{m_y+1}}}(0)} 1,\; \mathcal{F}_{2,2^{m_y}}(0)\right]$, for different values of $m_y$, and show that
	\begin{equation*}
		\lim_{a\to 0}T_{m_y,2}^a=\lim_{m_y\to\infty}\mathbb{P}^1\Big[-1\xleftrightarrow[{B_{2^{m_y}}(0)}]{B} 1\big|-1\xleftrightarrow{B;{B_{2^{m_y+1}}(0)}} 1,\; \mathcal{F}_{2,2^{m_y}}(0)\ \Big]\in (0,\infty).
	\end{equation*}
    A similar analysis in the scaling limit shows that
		\begin{align*}
	&	\lim_{m\to \infty}	\lim_{z,z' \to z_1} P\Big[z\xleftrightarrow[\hat{B}_{m-1}]{B}z' \big| z \xleftrightarrow{B;\hat{B}_m} z',\; \mathcal{F}_{|z'-z|,2^{m-1}|z'-z|}\Big(\frac{z+z'}{2}\Big) \Big] \\
		&\qquad=\lim_{m\to\infty} P\Big[-1\xleftrightarrow[B_{2^{m}}(0)]{B}1\big| -1\xleftrightarrow{B;B_{2^{m+1}}(0)}1,\; \mathcal{F}_{2,2^m}(0)\Big]\in (0,\infty).
		\end{align*}
		Since the two limits above exist and are not zero, we can write
		\begin{align} \label{eqn::change_of_limits_aux3}
			\lim_{a\to 0}T_{m_y,2}^a=C^*_4\lim_{m\to \infty}	\lim_{z,z' \to z_1} P\Big[z\xleftrightarrow[\hat{B}_{m-1}]{B}z' \big| z \xleftrightarrow{B;\hat{B}_m} z',\; \mathcal{F}_{|z'-z|,2^{m-1}|z'-z|}\Big(\frac{z+z'}{2}\Big) \Big],
		\end{align}
		for some $C^*_4\in(0,\infty)$.	
	
	\item We now analyze the term $T_{m_y,3}^a$. For $r>0$, we define the event 
	\[\hat{\mathcal{F}}^a_r(z_1^a,z_2^a,z_3^a):=\Big\{z_2^a\xleftrightarrow{B}\partial B_{r}(z_1^a),\; z_3^a\xleftrightarrow{B}\partial B_{r}(z_1^a),\; \big(z_2^a\xleftrightarrow{B;\big(B_r(z_1^a)\big)^c}z_3^a\big)^c\Big\}.\]
For $0<r_1<r_2$, we define the event 
    \begin{align*}
        \mathring{\mathcal{F}}_{r_1,r_2}(z_1^a):= &\mathcal{F}_{r_1,r_2}(z_1^a)\setminus \{\text{there exist five disjoint paths, not all with the same label,} \\
        & \qquad\qquad\qquad\text{ crossing the annulus }B_{r_2}(z_1^a)\setminus B_{r_1}(z_1^a)\}.
    \end{align*}
  On the event $\mathring{\mathcal{F}}_{r_1,r_2}(z_1^a)$, define $\mathcal{C}_{r_1,r_2}^a(z_1^a)$ to be the event that the two disjoint black paths in $\mathring{\mathcal{F}}_{r_1,r_2}(z_1^a)$ are connected to each other by a black path in $B_{r_1}(z_1^a)$. We define
    \[\tilde{\mathcal{F}}_{r_1,r_2}^a(z_1^a,z_2^a,z_3^a):=\hat{\mathcal{F}}^a_{r_1}(z_1^a,z_2^a,z_3^a)\cap \mathring{\mathcal{F}}_{r_1,r_2}(z_1^a).\]
    Since the five-arm exponent for critical percolation is strictly larger than its four-arm exponent~\cite{SmirnovWernerCriticalExponents}, we have 
    \begin{align} \label{eqn::five_arm_four_arm}
        {\limsup_{a\to 0}}\, \mathbb{P}^a\Big[\mathcal{F}_{r_1,r_2}(z_1^a)\setminus \mathring{\mathcal{F}}_{r_1,r_2}(z_1^a)\big|\mathcal{F}_{r_1,r_2}(z_1^a) \Big]=o(1), \quad \text{as }\frac{r_1}{r_2} \to 0 .
    \end{align}
To simplify the presentation and without loss of generality, we may assume that $z_1=0$ and $z_1^a=0$ for all $a>0$. Then we can write
\begin{align*}
T_{m_y,3}^a=&\,\mathbb{P}^a\Big[\tilde{\mathcal{F}}^a_{2^{m_y+1}a,\epsilon}(0,z_2^a,z_3^a) \big | -a\xleftrightarrow[B_{m_y-1}^a]{B;B_{m_y}^a}a,\;\mathcal{F}_{2^{m_y+1}a,\epsilon}(0)\Big]\\
&\qquad\times \mathbb{P}^a\Big[z_2^a\xleftrightarrow[(B_{m_y}^a)^c]{B}z_3^a\big| -a\xleftrightarrow[B_{m_y-1}^a]{B;B_{m_y}^a}a ,\; {\tilde{\mathcal{F}}}^a_{2^{m_y+1}a,\epsilon}(0,z_2^a,z_3^a)\Big]\\
=&\,\mathbb{P}^a\Big[\tilde{\mathcal{F}}^a_{2^{m_y+1}a,\epsilon}(0,z_2^a,z_3^a) \big | \mathcal{F}_{2^{m_y+1}a,\epsilon}(0)\Big]\\
&\qquad\times \Big(\mathbb{P}^a\Big[\mathcal{C}^a_{2^{m_y+1}a,\epsilon}(0)\big| -a\xleftrightarrow[B_{m_y-1}^a]{B;B_{m_y}^a}a ,\; {\tilde{\mathcal{F}}}^a_{2^{m_y+1}a,\epsilon}(0,z_2^a,z_3^a)\Big]+o(1)\Big), \text{ as } a\to 0,
\end{align*}
where, in the second equality, we have also used the fact that the event $\{-a \xleftrightarrow[B_{m_y-1}^a]{B;B_{m_y}^a} a\}$ is independent of both $\tilde{\mathcal{F}}^a_{2^{m_y+1}a,\epsilon}(0,z_2^a,z_3^a)$ and $\mathcal{F}_{2^{m_y+1}a,\epsilon}(0)$.

First, we deal with the outer connection $\tilde{\mathcal{F}}^a_{2^{m_y+1}a,\epsilon}(0,z_2^a,z_3^a)$. Let $0<\eta\ll 1$ and consider
\begin{align*}
&\pi_a^{-2}\times\Big| \mathbb{P}^a\Big[\tilde{\mathcal{F}}^a_{{2^{m_y+1}a},\epsilon}(0,z_2^a,z_3^a) \big | \mathcal{F}_{2^{m_y+1}a,\epsilon}(0)\Big]- \mathbb{P}^a\Big[\tilde{\mathcal{F}}^a_{\eta,\epsilon}(0,z_2^a,z_3^a) \big | \mathcal{F}_{\eta,\epsilon}(0)\Big]\Big|\\
& \quad = \pi_a^{-2}\times \mathbb{P}^a\Big[z_2\xleftrightarrow{B}\partial B_{\epsilon}(z_2^a)\Big]\mathbb{P}^a\Big[z_3\xleftrightarrow{B}\partial B_{\epsilon}(z_3^a)\Big]\\
&\qquad\times \Big|\mathbb{P}^a\Big[\tilde{\mathcal{F}}^a_{{2^{m_y+1}a},\epsilon}(0,z_2^a,z_3^a) \big | \mathcal{F}_{2^{m_y+1}a,\epsilon}(0),\, z_k^a\xleftrightarrow{B}\partial B_{\epsilon}(z_k^a),\, k=2,3\Big] \\
&\qquad\qquad\qquad -\mathbb{P}^a\Big[\tilde{\mathcal{F}}^a_{\eta,\epsilon}(0,z_2^a,z_3^a)\big | \mathcal{F}_{\eta,\epsilon}(0), \, z_k^a\xleftrightarrow{B}\partial B_{\epsilon}(z_k^a),\, k=2,3\Big]\Big|.
\end{align*}
It follows from~\cite{GarbanPeteSchrammPivotalClusterInterfacePercolation} (see the first limit in the third displayed equation on page 999) that 
\begin{align} \label{eqn::GPS999}
    \lim_{a\to 0}\pi_a^{-2}\times \mathbb{P}^a\Big[z_2\xleftrightarrow{B}\partial B_{\epsilon}(z_2^a)\Big]\mathbb{P}^a\Big[z_3\xleftrightarrow{B}\partial B_{\epsilon}(z_3^a)\Big] =\epsilon^{-\frac{5}{24}}.
\end{align}
Moreover, as in the proof ofLemma~\ref{lem::four_arm_coupling_argument}, one can find a coupling between the two measures 
\begin{align*}
\mathbb{P}^a\Big[\cdot \big| \mathcal{F}_{2^{m_y+1}a,\epsilon}(0),\, z_k^a\xleftrightarrow{B}\partial B_{\epsilon}(z_k^a),\, k=2,3\Big]\text{ and }\mathbb{P}^a\Big[\cdot | \mathcal{F}_{\eta,\epsilon}(0),\,z_k^a\xleftrightarrow{B}\partial B_{\epsilon}(z_k^a),\, k=2,3\Big]
\end{align*}
showing that
\begin{align*}
&\Big|\mathbb{P}^a\Big[\tilde{\mathcal{F}}^a_{{2^{m_y+1}a},\epsilon}(0,z_2^a,z_3^a) \big | \mathcal{F}_{2^{m_y+1}a,\epsilon}(0),\, z_k^a\xleftrightarrow{B}\partial B_{\epsilon}(z_k^a),\, k=2,3\Big] \\
&\qquad\qquad\qquad -\mathbb{P}^a\Big[\tilde{\mathcal{F}}^a_{\eta,\epsilon}(0,z_2^a,z_3^a)\big | \mathcal{F}_{\eta,\epsilon}(0), \, z_k^a\xleftrightarrow{B}\partial B_{\epsilon}(z_k^a),\, k=2,3\Big]\Big|\\
&\qquad\lesssim \left(\min\left\{\frac{1}{\eta}, 2^{M-m_y}\right\}\right)^{-c_*}
\end{align*}
for some constant $c_*\in (0,\infty)$.

Letting $a\to 0$ shows that 
\begin{align*}
&\limsup_{a\to 0} \Big|\mathbb{P}^a\Big[\tilde{\mathcal{F}}^a_{{2^{m_y+1}a},\epsilon}(0,z_2^a,z_3^a) \big | \mathcal{F}_{2^{m_y+1}a,\epsilon}(0),\, z_k^a\xleftrightarrow{B}\partial B_{\epsilon}(z_k^a),\, k=2,3\Big] \\
&\qquad\qquad\qquad -\mathbb{P}^a\Big[\tilde{\mathcal{F}}^a_{\eta,\epsilon}(0,z_2^a,z_3^a)\big | \mathcal{F}_{\eta,\epsilon}(0), \, z_k^a\xleftrightarrow{B}\partial B_{\epsilon}(z_k^a),\, k=2,3\Big]\Big|\\   &\quad=\limsup_{a\to 0} \Big|\mathbb{P}^a\Big[\tilde{\mathcal{F}}^a_{{2^{m_y+1}a},\epsilon}(0,z_2^a,z_3^a) \big | \mathcal{F}_{2^{m_y+1}a,\epsilon}(0),\, z_k^a\xleftrightarrow{B}\partial B_{\epsilon}(z_k^a),\, k=2,3\Big] \\
&\quad\qquad\qquad\qquad -P\Big[\tilde{\mathcal{F}}_{\eta,\epsilon}(0,z_2,z_3)\big | \mathcal{F}_{\eta,\epsilon}(0), \, z_k\xleftrightarrow{B}\partial B_{\epsilon}(z_k),\, k=2,3\Big]\Big| \lesssim \eta^{c_*},
\end{align*}
where 
\begin{align*}
& P\Big[\tilde{\mathcal{F}}_{\eta,\epsilon}(0,z_2,z_3)\big | \mathcal{F}_{\eta,\epsilon}(0), \, z_k\xleftrightarrow{B}\partial B_{\epsilon}(z_k),\, k=2,3\Big]\\
&\quad:=\lim_{a\to 0}\mathbb{P}^a\Big[\tilde{\mathcal{F}}^a_{\eta,\epsilon}(0,z_2^a,z_3^a)\big | \mathcal{F}_{\eta,\epsilon}(0), \, z_k^a\xleftrightarrow{B}\partial B_{\epsilon}(z_k^a),\, k=2,3\Big]
\end{align*}
and the existence of the last limit follows from the same arguments as in the proof of~\cite[Theorem~1.5]{Cam23}.
 Standard RSW arguments show that $P\Big[\tilde{\mathcal{F}}_{\eta,\epsilon}(0,z_2,z_3)\big | \mathcal{F}_{\eta,\epsilon}(0), \, z_k\xleftrightarrow{B}\partial B_{\epsilon}(z_k),\, k=2,3\Big]$ is bounded away from $0$ and $1$ uniformly in $\eta$. Therefore, there exists a sequence $\{\eta_n\}_{n=1}^{\infty}$ with $\lim_{n\to \infty}\eta_n=0$ such that the limit 
\[\lim_{n\to \infty}P\Big[\tilde{\mathcal{F}}_{\eta_n,\epsilon}(0,z_2,z_3)\big | \mathcal{F}_{\eta_n,\epsilon}(0), \, z_k\xleftrightarrow{B}\partial B_{\epsilon}(z_k),\, k=2,3\Big]\in (0,1)\]
exists.
Letting $n\to \infty$, we get
\begin{align*}
&\lim_{a\to 0}\mathbb{P}^a\Big[\tilde{\mathcal{F}}^a_{{2^{m_y+1}a},\epsilon}(0,z_2^a,z_3^a) \big | \mathcal{F}_{2^{m_y+1}a,\epsilon}(0),\, z_k^a\xleftrightarrow{B}\partial B_{\epsilon}(z_k^a),\, k=2,3\Big]\\
&\quad =\lim_{n\to \infty}P\Big[\tilde{\mathcal{F}}_{\eta_n,\epsilon}(0,z_2,z_3)\big | \mathcal{F}_{\eta_n,\epsilon}(0), \, z_k\xleftrightarrow{B}\partial B_{\epsilon}(z_k),\, k=2,3\Big].
\end{align*}
As the limit on the left-hand side does not depend on the choice of $\{\eta_n\}_{n=1}^{\infty}$, we have
\begin{equation} \label{eqn::one_arm_cap_four_arm}
\begin{split}
    &\lim_{a\to 0}\mathbb{P}^a\Big[\tilde{\mathcal{F}}^a_{{2^{m_y+1}a},\epsilon}(0,z_2^a,z_3^a) \big | \mathcal{F}_{2^{m_y+1}a,\epsilon}(0),\, z_k^a\xleftrightarrow{B}\partial B_{\epsilon}(z_k^a),\, k=2,3\Big]\\
&\quad =\lim_{\eta\to 0}P\Big[\tilde{\mathcal{F}}_{\eta,\epsilon}(0,z_2,z_3)\big | \mathcal{F}_{\eta,\epsilon}(0), \, z_k\xleftrightarrow{B}\partial B_{\epsilon}(z_k),\, k=2,3\Big].
\end{split}
\end{equation}

Combining~\eqref{eqn::GPS999} and~\eqref{eqn::one_arm_cap_four_arm} gives
\begin{align}
\begin{split} \label{eqn::limT3-part1}
&\lim_{a \to 0}\pi_a^{-2}\times \mathbb{P}^a\Big[\tilde{\mathcal{F}}^a_{{2^{m_y+1}a},\epsilon}(0,z_2^a,z_3^a) \big | \mathcal{F}_{2^{m_y+1}a,\epsilon}(0)\Big] \\
& \qquad = \epsilon^{-\frac{5}{24}} \times \lim_{\eta\to 0}P\Big[\tilde{\mathcal{F}}_{\eta,\epsilon}(0,z_2,z_3)\big | \mathcal{F}_{\eta,\epsilon}(0), \, z_k\xleftrightarrow{B}\partial B_{\epsilon}(z_k),\, k=2,3\Big].
\end{split}
\end{align}

Next, we deal with the inner connection $\mathcal{C}^a_{2^{m_y+1}a,\epsilon}(0)$. Let $\eta(a)$ be a  vertex in $a\mathcal{T}$ such that $\big| \eta(a)-\eta 2^{m_y}a\big|\leq a$.
Again, as in the proof of Lemma~\ref{lem::four_arm_coupling_argument},  one can find a coupling between the two measures 
\begin{align*}
\mathbb{P}^a\Big[\cdot \big| -a\xleftrightarrow[B_{m_y-1}^a]{B;B_{m_y}^a}a ,\; \tilde{\mathcal{F}}^a_{2^{m_y+1}a,{\epsilon}}(0,z_2^a,z_3^a) \Big] \text{ and }\mathbb{P}^a\Big[\cdot \big| -\eta(a)\xleftrightarrow[B_{m_y-1}^a]{B;B_{m_y}^a}\eta(a),\; \tilde{\mathcal{F}}^a_{2^{m_y+1}a,\epsilon}(0,z_2^a,z_3^a)\Big]
\end{align*} 
showing that 
\begin{align*}
&	\Big|\mathbb{P}^a\Big[\mathcal{C}^a_{2^{m_y+1}a,\epsilon}(0)\big| -a\xleftrightarrow[B_{m_y-1}^a]{B;B_{m_y}^a}a ,\; \tilde{\mathcal{F}}^a_{2^{m_y+1}a,\epsilon}(0,z_2^a,z_3^a)\Big]\\
&	\qquad\qquad\qquad- \mathbb{P}^a\Big[\mathcal{C}^a_{2^{m_y+1}a,\epsilon}(0)\big| -\eta(a)\xleftrightarrow[B_{m_y-1}^a]{B;B_{m_y}^a}\eta(a) ,\; \tilde{\mathcal{F}}^a_{2^{m_y+1}a,{\epsilon}}(0,z_2^a,z_3^a)\Big]\Big|\lesssim \eta^{c_*}
\end{align*}
for some constant $c_*\in (0,\infty)$. Let $a_*:=2^{-m_y}$, which tends to $0$ as $a\to 0$. Rescaling the triangular lattice by $a_*$, we get
	\begin{align*}
&\mathbb{P}^a\Big[\mathcal{C}^a_{2^{m_y+1}a,\epsilon}(0)\big| -\eta(a)\xleftrightarrow[B_{m_y-1}^a]{B;B_{m_y}^a}\eta(a) ,\; \tilde{\mathcal{F}}^a_{2^{m_y+1}a,{\epsilon}}(0,z_2^a,z_3^a)\Big]\\
	&\qquad= \mathbb{P}^{a_*}\Big[\mathcal{C}^{a_*}_{2,\epsilon/(2^{m_y}a)}(0) \big|  -\frac{\eta(a)}{2^{m_y}a} \xleftrightarrow[B_1(0)]{B;B_2(0)} \frac{\eta(a)}{2^{m_y}a},\; 
    \tilde{\mathcal{F}}^{a_*}_{2,\epsilon/(2^{m_y}a)}
    (0,z_2^a/(2^{m_y}a), z_3^a/(2^{m_y}a))\Big].
	\end{align*}
Once again, as in the proof of Lemma~\ref{lem::four_arm_coupling_argument}, one can find a coupling between the two measures 
	\begin{align*}
		\mathbb{P}^{a_*}\Big[\cdot \big|  -\frac{\eta(a)}{2^{m_y}a} \xleftrightarrow[B_1(0)]{B;B_2(0)} \frac{\eta(a)}{2^{m_y}a},\; \tilde{\mathcal{F}}^{a_*}_{2,\epsilon/(2^{m_y}a)}(0,z_2^a/(2^{m_y}a), z_3^a/(2^{m_y}a))\Big]
	\end{align*}
	and 
	\begin{align*}
		\mathbb{P}^{a_*}\Big[\cdot \big| -\frac{\eta(a)}{2^{m_y}a} \xleftrightarrow[B_1(0)]{B;B_2(0)} \frac{\eta(a)}{2^{m_y}a},\; \mathring{\mathcal{F}}_{2,1/\eta}(0)\Big]
	\end{align*}
showing that 
\begin{align*}
&	\Big| \mathbb{P}^{a_*}\Big[\mathcal{C}^{a_*}_{2,\epsilon/(2^{m_y}a)}(0) \big|  -\frac{\eta(a)}{2^{m_y}a} \xleftrightarrow[B_1(0)]{B;B_2(0)} \frac{\eta(a)}{2^{m_y}a},\; \tilde{\mathcal{F}}^{a_*}_{2,\epsilon/(2^{m_y}a)}(0,z_2^a/(2^{m_y}a), z_3^a/(2^{m_y}a))\Big]\\
&\qquad\qquad - \mathbb{P}^{a_*}\Big[\mathcal{C}^{a_*}_{2,1/\eta}(0) \big|  -\frac{\eta(a)}{2^{m_y}a} \xleftrightarrow[B_1(0)]{B;B_2(0)} \frac{\eta(a)}{2^{m_y}a},\; \mathring{\mathcal{F}}_{2,1/\eta}(0)\Big]\Big| \lesssim \left(\min\left\{\frac{1}{\eta}, 2^{M-m_y}\right\}\right)^{-c_*}
\end{align*}
for some constant $c_*\in (0,\infty)$.
Letting $a\to 0$, we get 
    	\begin{align*}
&	\Big| \limsup_{a\to 0}\mathbb{P}^a\Big[\mathcal{C}^a_{2^{m_y+1}a,\epsilon}(0)\big| -a\xleftrightarrow[B_{m_y-1}^a]{B;B_{m_y}^a}a ,\; \tilde{\mathcal{F}}^a_{2^{m_y+1}a,{\epsilon}}(0,z_2^a,z_3^a)\Big]\\
	&\qquad\qquad - 	P\Big[\mathcal{C}_{2,1/\eta}(0) \big| -\eta\xleftrightarrow[B_1(0)]{B;B_2(0)} \eta,\; \mathring{\mathcal{F}}_{2,1/\eta}(0)\Big]\Big| \lesssim \eta^{c_*},\\
    &	\Big| \liminf_{a\to 0}\mathbb{P}^a\Big[\mathcal{C}^a_{2^{m_y+1}a,\epsilon}(0)\big| -a\xleftrightarrow[B_{m_y-1}^a]{B;B_{m_y}^a}a ,\; \tilde{\mathcal{F}}^a_{2^{m_y+1}a,{\epsilon}}(0,z_2^a,z_3^a)\Big]\\
	&\qquad\qquad - 	P\Big[\mathcal{C}_{2,1/\eta}(0) \big| -\eta\xleftrightarrow[B_1(0)]{B;B_2(0)} \eta,\; \mathring{\mathcal{F}}_{2,1/\eta}(0)\Big]\Big| \lesssim \eta^{c_*},
	\end{align*}
    where 
    \begin{align*}
        P\Big[\mathcal{C}_{2,1/\eta}(0)\big| -\eta\xleftrightarrow[B_1(0)]{B;B_{2}(0)}\eta,\; \mathring{\mathcal{F}}_{2,1/\eta}(0)\Big]:=\lim_{a\to 0} \mathbb{P}^{a_*}\Big[\mathcal{C}_{2,1/\eta}^{a_*}(0) \Big|\frac{\eta(a)}{2^{m_y}a} \xleftrightarrow[B_1(0)]{B;B_2(0)} \frac{\eta(a)}{2^{m_y}a},\; \mathring{\mathcal{F}}_{2,1/\eta}(0)
    \Big]
    \end{align*}
    and the existence of the limit can be shown using the strategy in~\cite[Proof of Theorem~1.1]{Cam23}. 
    Standard RSW arguments show that $P\Big[\mathcal{C}_{2,1/\eta}(0)\big| -\eta\xleftrightarrow[B_1(0)]{B;B_{2}(0)}\eta,\; \mathring{\mathcal{F}}_{2,1/\eta}(0)\Big]$ is bounded away from $0$ and $1$, uniformly in $\eta\to 0$. Therefore, there exists a sequence $\{\eta_n\}_{n=1}^{\infty}$ with $\lim_{n\to \infty}\eta_n=0$ such that the limit $\lim_{n\to \infty}P\Big[\mathcal{C}_{2,1/\eta_n}(0) \big| -\eta_n\xleftrightarrow[B_1(0)]{B;B_2(0)} \eta_n,\; \mathring{\mathcal{F}}_{2,1/\eta_n}(0)\Big]$ exists. Letting $n\to \infty$, we get 
{\small   \begin{align*}
	\lim_{a\to 0}\mathbb{P}^a\Big[\mathcal{C}^a_{2^{m_y+1}a,\epsilon}(0)\big| -a\xleftrightarrow[B_{m_y-1}^a]{B;B_{m_y}^a}a ,\; \tilde{\mathcal{F}}^a_{2^{m_y+1}a,{\epsilon}}(0,z_2^a,z_3^a)\Big]= \lim_{n\to \infty}P\Big[\mathcal{C}_{2,1/\eta_n}(0) \big| -\eta_n\xleftrightarrow[B_1(0)]{B;B_2(0)} \eta_n,\; \mathring{\mathcal{F}}_{2,1/\eta_n}(0)\Big].
    \end{align*}}
    Since the limit on the left-hand side does not depend on the choice of $\{\eta_n\}_{n=1}^{\infty}$, we have
	\begin{align}\label{eqn::change_of_limits_aux6}
	\lim_{a\to 0}\mathbb{P}^a\Big[\mathcal{C}^a_{2^{m_y+1}a,\epsilon}(0)\big| -a\xleftrightarrow[B_{m_y-1}^a]{B;B_{m_y}^a}a ,\; \tilde{\mathcal{F}}^a_{2^{m_y+1}a,{\epsilon}}(0,z_2^a,z_3^a)\Big]= \lim_{\eta\to 0} P\Big[\mathcal{C}_{2,1/\eta}(0)\big| -\eta\xleftrightarrow[B_1(0)]{B;B_2(0)}\eta, \; \mathring{\mathcal{F}}_{2,1/\eta}(0)\Big].
	\end{align}
    
    To summarize, combining~\eqref{eqn::change_of_limits_aux6} with~\eqref{eqn::limT3-part1}, we have shown that 
	\begin{equation}
	    \begin{split}
	        \lim_{a\to 0 }\pi_{a}^{-2}\times T_{m_y,3}^a=& \; \epsilon^{-\frac{5}{24}}\lim_{\eta\to 0}P\Big[\tilde{\mathcal{F}}_{\eta,\epsilon}(0,z_2,z_3)\big | \mathcal{F}_{\eta,\epsilon}(0), \, z_k\xleftrightarrow{B}\partial B_{\epsilon}(z_k),\, k=2,3\Big] \\
    &\quad\times \lim_{\eta\to 0} P\Big[\mathcal{C}_{2,1/\eta}(0)\big| -\eta\xleftrightarrow[B_1(0)]{B;B_2(0)}\eta, \; \mathring{\mathcal{F}}_{2,1/\eta}(0)\Big].
	    \end{split}
	\end{equation}
		A similar analysis in the scaling limit shows that
	\begin{equation}\label{eqn::change_of_limits_aux4}
\begin{split}
			&	\lim_{m\to\infty}\lim_{z,z'\to z_1} P\Big[z_2\xleftrightarrow[{\hat{B}_m^c}]{B} z_3 \big| z\xleftrightarrow[{\hat{B}_{m-1}}]{B;{\hat{B}_m}}z',\; \mathcal{F}_{2^m|z'-z|,\epsilon}\Big(\frac{z+z'}{2}\Big)\Big]\\
	&\qquad\qquad= \epsilon^{-\frac{5}{24}}\lim_{\eta\to 0}P\Big[\tilde{\mathcal{F}}_{\eta,\epsilon}(0,z_2,z_3)\big | \mathcal{F}_{\eta,\epsilon}(0), \, z_k\xleftrightarrow{B}\partial B_{\epsilon}(z_k),\, k=2,3\Big] \\
    &\qquad\qquad\qquad\times \lim_{\eta\to 0} P\Big[\mathcal{C}_{2,1/\eta}(0)\big| -\eta\xleftrightarrow[B_1(0)]{B;B_2(0)}\eta, \; \mathring{\mathcal{F}}_{2,1/\eta}(0)\Big]=\lim_{a\to 0}\pi_a^{-2}\times\,T_{m_y,3}^a.
\end{split}
		\end{equation}

	\item Similarly, proceeding as in the analysis of $T_{m_y,3}^a$, one can show that 
	\begin{align}\label{eqn::change_of_limits_aux5}
		\lim_{a\to 0}\pi_a^{-2} \times T_{m_y,4}^a= &\lim_{m\to \infty}\lim_{z,z'\to z_1}  P\Big[z_2\xleftrightarrow[\hat{B}^c_m]{B}z_3 \big| \mathcal{F}_{2^{m}|z'-z|,\epsilon}\Big(\frac{z+z'}{2}\Big)\Big].
	\end{align}
\end{itemize}

To conclude the proof of~\eqref{eqn::change_of_limits}, suppose that $z^a,(z')^a\in a\mathcal{T}$ satisfy $z^a\to z$ and $(z')^a\to z'$, as $a\to 0$.
By definition and letting $\hat{B}^a_{m}:=B_{2^m|(z')^a-z^a|}\big(\frac{z^a+(z')^a}{2}\big)$, we have
\begin{align*}
   & P\Big[z\xleftrightarrow[\hat{B}_{m-1}]{B}z',\ z_2\xleftrightarrow[\hat{B}_{m}^c]{B}z_3\big|z\xleftrightarrow{B;\hat{B}_{m}}z',\, \mathcal{F}_{2|z'-z|,2^{m-1}|z'-z|}\big(\tfrac{z+z'}{2}\big),\, \mathcal{F}_{2^{m}|z'-z|,\epsilon}(\tfrac{z+z'}{2})\Big]  \notag\\
	& \quad - P\Big[z\xleftrightarrow[\hat{B}_{m-1}]{B}z' \big| z\xleftrightarrow{B;\hat{B}_{m}}z',\, \mathcal{F}_{2|z'-z|,2^{m-1}|z'-z|}\big(\tfrac{z+z'}{2}\big)\Big]P\Big[z_2\xleftrightarrow[\hat{B}_{m}^c]{B}z_3\big|\mathcal{F}_{2^{m}|z'-z|,\epsilon}\big(\tfrac{z+z'}{2}\big)\Big]\\
    &= \lim_{a\to 0}\pi_a^{-2}\times\bigg(\mathbb{P}^a\Big[z^a\xleftrightarrow[\hat{B}^a_{m-1}]{B}(z')^a,\ z^a_2\xleftrightarrow[(\hat{B}^a_{m})^c]{B}z^a_3\big|z^a\xleftrightarrow{B;\hat{B}^a_{m}}(z')^a, \mathcal{F}_{2|(z')^a-z^a|,2^{m-1}|(z')^a-z^a|}\Big(\tfrac{z^a+(z')^a}{2}\Big),\\
&\qquad\qquad\qquad\qquad\qquad\qquad\qquad\qquad\qquad\qquad\qquad\qquad\mathcal{F}_{2^{m}|(z')^a-(z)^a|,\epsilon}\Big(\tfrac{z^a+(z')^a}{2}\Big)\Big]  \notag\\
	&  - \mathbb{P}^a\Big[z^a\xleftrightarrow[\hat{B}^a_{m-1}]{B}(z')^a \big| z^a\xleftrightarrow{B;\hat{B}^a_{m}}(z')^a,\, \mathcal{F}_{2|(z')^a-z^a|,2^{m-1}|(z')^a-z^a|}\Big(\tfrac{z^a+(z')^a}{2}\Big)\Big]\\
&\qquad\qquad\qquad\qquad\qquad\qquad\qquad\qquad\qquad\qquad\qquad\qquad \mathbb{P}^a\Big[z^a_2\xleftrightarrow[(\hat{B}^a_{m})^c]{B}z_3^a\big|\mathcal{F}_{2^{m}|(z')^a-z^a|,\epsilon}\Big(\tfrac{z^a+(z')^a}{2}\Big)\Big]\bigg)\\
&=\lim_{a\to 0} \pi_a^{-2}\times \mathbb{P}^a\Big[z^a\xleftrightarrow[\hat{B}^a_{m-1}]{B}(z')^a \big| z^a \xleftrightarrow{B;\hat{B}^a_{m}} (z')^a,\; \mathcal{F}_{2|(z')^a-z^a|,2^{m-1}|(z')^a-z^a|}\Big(\tfrac{z^a+(z')^a}{2}\Big) \Big]
	\\
	&\qquad\quad\times\bigg(\mathbb{P}^a\Big[ z_2^a\xleftrightarrow[(\hat{B}^a_{m})^c]{B}z_3^a\big| z^a\xleftrightarrow[\hat{B}^a_{m-1}]{B;\hat{B}^a_{m}} (z')^a,\;  \mathcal{F}_{2^{m_y}|(z')^a-z^a|,\epsilon}\Big(\tfrac{z^a+(z')^a}{2}\Big)\Big]\\
    &\qquad\qquad\qquad\qquad\qquad\qquad\qquad\qquad\qquad\qquad - \mathbb{P}^a\Big[z_2^a\xleftrightarrow[(\hat{B}^a_{m})^c]{B}z_3^a \big| \mathcal{F}_{2^{m_y}|(z')^a-z^a|,\epsilon}\Big(\tfrac{z^a+(z')^a}{2}\Big)\Big] \bigg)\\
    &=P\Big[z\xleftrightarrow[\hat{B}_{m-1}]{B}z'\big|z\xleftrightarrow{B;\hat{B}_{m}}z',\, \mathcal{F}_{2|z'-z|,2^{m-1}|z'-z|}(\tfrac{z+z'}{2})\Big] \\
    &\qquad\quad\times \bigg(P\Big[z_2\xleftrightarrow[\hat{B}_{m}^c]{B}z_3\big|z\xleftrightarrow[\hat{B}_{m-1}]{B;\hat{B}_{m}}z',\, \mathcal{F}_{2^{m}|z'-z|,\epsilon}(\tfrac{z+z'}{2})\Big] -P\Big[z_2\xleftrightarrow[\hat{B}_{m}^c]{B}z_3\big|\mathcal{F}_{2^{m}|z'-z|,\epsilon}(\tfrac{z+z'}{2})\Big]\bigg).
\end{align*}
Combining this with~\eqref{eqn::decomp-T234}, \eqref{eqn::change_of_limits_aux3}, \eqref{eqn::change_of_limits_aux4} and~\eqref{eqn::change_of_limits_aux5} gives~\eqref{eqn::change_of_limits}, as desired.

Finally, combining~\eqref{eqn::T^a_m}, \eqref{eqn::four_arm_quasi}, \eqref{eqn::lim=C^*_3} and~\eqref{eqn::change_of_limits}, 
we have that 
\begin{align*}
& \lim_{a\to 0} a^{-\frac{5}{4}} \pi_a^{-2} \times T_{m_y}^a = C^*_1 C^*_2 C^*_3 C_4^*\epsilon^{-\frac{5}{4}} \\
& \quad\times  \lim_{m\to \infty} \lim_{z,z'\to z_1} \bigg(P\Big[z\xleftrightarrow[\hat{B}_{m-1}]{B}z',\ z_2\xleftrightarrow[\hat{B}_{m}^c]{B}z_3\big|z\xleftrightarrow{B;\hat{B}_{m}}z',\, \mathcal{F}_{2|z'-z|,2^{m-1}|z'-z|}\big(\tfrac{z+z'}{2}\big),\, \mathcal{F}_{2^{m}|z'-z|,{\epsilon}}(\tfrac{z+z'}{2})\Big]   \\
& \qquad - P\Big[z\xleftrightarrow[\hat{B}_{m-1}]{B}z' \big| z\xleftrightarrow{B;\hat{B}_{m}}z',\, \mathcal{F}_{2|z'-z|,2^{m-1}|z'-z|}\big(\tfrac{z+z'}{2}\big)\Big]P\Big[z_2\xleftrightarrow[\hat{B}_{m}^c]{B}z_3\big|\mathcal{F}_{2^{m}|z'-z|,{\epsilon}}\big(\tfrac{z+z'}{2}\big)\Big]\bigg).
\end{align*}
Furthermore, using the definitions of the terms in the right-hand side, as in the calculation just above, we obtain
\begin{equation} \label{eqn::three_point_spin_energy_aux0}
\begin{split}
& \lim_{a\to 0} a^{-\frac{5}{4}} \pi_a^{-2} \times T_{m_y}^a = C^*_1 C^*_2 C^*_3 C^*_4 \epsilon^{-\frac{5}{4}} \lim_{m\to \infty} \lim_{z,z'\to z_1}\Big( P\Big[ z\xleftrightarrow{B;\hat{B}_{m}}z' \big| \mathcal{F}_{2|z'-z|,2^{m-1}|z'-z|}\big(\tfrac{z+z'}{2}\big) \Big]\Big)^{-1}\\
&\qquad\qquad\qquad\times\Big(P\Big[\mathcal{F}_{2|z'-z|,2^{m-1}|z'-z|}\big(\tfrac{z+z'}{2}\big)\Big] P\Big[\mathcal{F}_{2^m|z'-z|,{\epsilon}}\big(\tfrac{z+z'}{2}\big)\Big]\Big)^{-1} \\
& \qquad\qquad\qquad \times \Big( P\Big[ z\xleftrightarrow[\hat{B}_{m-1}]{B;\hat{B}_{m}}z', z_2\xleftrightarrow[\hat{B}^c_m]{B}z_3 \Big] - P\Big[ z\xleftrightarrow[\hat{B}_{m-1}]{B;\hat{B}_{m}}z' \Big] P\Big[ z_2\xleftrightarrow[\hat{B}^c_m]{B}z_3 \Big] \Big),
\end{split}
\end{equation}
where the last equality follows from~\eqref{eqn::Tm_decom} replacing $(z_1,z_2,z_3,z_4)$ with $(z,z',z_2,z_3)$.

To conclude the proof, note that,
\begin{itemize}
\item according to~\eqref{eqn::limit_two_point}, there exists a constant $C_5^*=C_*\in (0,\infty)$ such that
	\begin{align} \label{eqn::C^*_4}
		\lim_{m\to \infty}\lim_{z,z'\to z_1} |z'-z|^{\frac{5}{24}} \times P\Big[z\xleftrightarrow{B;\hat{B}_{m}}z'\big| \mathcal{F}_{2|z'-z|,2^{m-1}|z'-z|}\big(\tfrac{z+z'}{2}\big) \Big]=C_5^*,
		\end{align}
\item there exists a constant $C_6^*\in (0,\infty)$ such that
        	\begin{align} \label{eqn::three_point_spin_energy_aux2}
			\begin{split}
				& \lim_{m\to \infty} \lim_{z,z'\to z_1} |z'-z|^{-\frac{5}{4}} \times P\Big[\mathcal{F}_{2|z'-z|,2^{m-1}|z'-z|}\big(\tfrac{z+z'}{2}\big)\Big] P\Big[\mathcal{F}_{2^{m}|z'-z|,\epsilon}\big(\tfrac{z+z'}{2}\big)\Big] \\	
                & \quad=\lim_{m\to \infty}P\Big[\mathcal{F}_{2^{2-m},1}(0)\Big]\lim_{z,z'\to z_1} |z'-z|^{-\frac{5}{4}}\times P\Big[\mathcal{F}_{2^m|z'-z|/\epsilon,1}(0)\Big]\\
                        &\quad = C_8 \, \epsilon^{-\frac{5}{4}} \, \lim_{m\to\infty}P\Big[\mathcal{F}_{2^{2-m},1}(0)\Big]\times 2^{\frac{5}{4}m}\\
                &\quad = C_8^2 \, 2^{\frac{5}{2}} \, \epsilon^{-\frac{5}{4}}=C_1\epsilon^{-\frac{5}{4}},
			\end{split}
		\end{align}
        where the first equality is due to the conformal invariance of the full-plane, nested $\mathrm{CLE}_6$, the second and third equalities follow from~\cite[Corollary~2.4]{CF24}, and where $C_8\in (0,\infty)$ is the constant in~\cite[Corollary~2.4]{CF24}.
\end{itemize}
Plugging~\eqref{eqn::C^*_4} and~\eqref{eqn::three_point_spin_energy_aux2} into~\eqref{eqn::three_point_spin_energy_aux0}, we conclude that
\begin{align}\label{eqn::three_point_spin_energy_aux3}
	\lim_{a\to 0} a^{-\frac{5}{4}} \pi_a^{-2} \times T_{m_y}^a=\frac{C_1^*C_2^*C_3^*C_4^*}{C_5^*C_6^*}\lim_{m\to \infty}\lim_{z,z'\to z} |z'-z|^{-\frac{25}{24}} \times T_m(z,z',z_2,z_3),
\end{align}
where $T_m(z,z',z_2,z_3)$ is the same term as in Equation~\eqref{eqn::four_spin_decom} in~Section~\ref{sec::sketch_four_spin}, but with $(z_1,z_2,z_3,z_4)$ replaced by $(z,z',z_2,z_3)$.
According to Remark~\ref{rem::truncated_four_point}, we have
\begin{align} \label{eqn::three_point_spin_energy_aux4}
	 \lim_{z,z'\to z_1} \frac{P\Big[z\xleftrightarrow{B}z',\; z_2\xleftrightarrow{B}z_3\Big]-P\Big[z\xleftrightarrow{B}z'\Big]P\Big[z_2\xleftrightarrow{B}z_3\Big]}{|z'-z|^{\frac{25}{24}}|\log |z'-z||}= \frac{1}{\log 2} \lim_{m\to \infty}\lim_{z,z'\to z} |z'-z|^{-\frac{25}{24}} \times T_m(z,z',z_2,z_3).
\end{align}
It follows from a standard application of RSW estimates (see e.g., the proofs of Lemmas 2.1 and 2.2 of~\cite{CamiaNewman2009ising}) that $a^{-5/4}\pi_a^{-2} T_m^a$ is uniformly bounded in $a$ and $m$.
Combining~\eqref{eqn::three_point_spin_energy_aux3} with~\eqref{eqn::three_point_spin_energy_aux4} and using the dominated convergence theorem, we obtain~\eqref{eqn::existence-limit-sum} and~\eqref{eq::lim-sum-T_m} with $C:=(\log 2) C_1^* C_2^*C_3^*C_4^*/(C_5^*C_6^*$), as desired.
\end{proof}

\begin{proof}[Proof of~\eqref{eqn::mixed_correlation_spin_energy} in Theorem~\ref{thm::energy_spin} with $\ell=2$ and $k=1$. ]
We choose $\epsilon>0$ such that $B_{2\epsilon}(z_i)\cap B_{2\epsilon}(z_j)=\emptyset$ for $i\neq j$. Let $M$ be the largest integer such that $2^M (2a)\leq \epsilon$; note that $M$ is of order $-\log a$. For $1\leq i\leq 2$ and $1\leq m\leq M$, define $B_{m}^{(i,a)}:=\{w: |w-z_i^a|\leq 2^m(2a)\}$.  	

\medskip

\noindent\textbf{Step 1.} We claim that, for $i=1,2$, we have 
\begin{align*}
		&\Big| \mathbb{P}^a\Big[z_i^{(a,-)}\xleftrightarrow{B}z_3^a\centernot{\xleftrightarrow{B}}z_i^{(a,+)}\xleftrightarrow{B}z_4^a\Big] \mathbb{P}^a\Big[z_{3-i}^{(a,-)}\xleftrightarrow{B}z_{3-i}^{(a,+)}\Big]\\
	&\qquad\qquad\qquad\qquad	-\mathbb{P}^a\Big[z_i^{(a,-)}\xleftrightarrow{B}z_3^a\centernot{\xleftrightarrow{B}}z_i^{(a,+)}\xleftrightarrow{B}z_4^a,\; z_{3-i}^{(a,-)}\xleftrightarrow{B}z_{3-i}^{(a,+)}\Big]\Big|\lesssim a^{\frac{5}{2}} |\log a|\pi_a^2, \\
			&\Big| \mathbb{P}^a\Big[z_i^{(a,-)}\xleftrightarrow{B}z_4^a\centernot{\xleftrightarrow{B}}z_i^{(a,+)}\xleftrightarrow{B}z_3^a\Big] \mathbb{P}^a\Big[z_{3-i}^{(a,-)}\xleftrightarrow{B}z_{3-i}^{(a,+)}\Big]\\
	&\qquad\qquad\qquad\qquad	-\mathbb{P}^a\Big[z_i^{(a-)}\xleftrightarrow{B}z_4^a\centernot{\xleftrightarrow{B}}z_i^{(a,+)}\xleftrightarrow{B}z_3^a,\; z_{3-i}^{(a,-)}\xleftrightarrow{B}z_{3-i}^{(a,+)}\Big]\Big|\lesssim a^{\frac{5}{2}} |\log a|\pi_a^2.
\end{align*}
These estimates can be proved using the same type of argument used in \cite[Section~3.2]{CF24} and in the proof of~\eqref{eqn::three_point_energy_density} above, by expressing the event $\Big\{z_{3-i}^{(a,-)}\xleftrightarrow{B}z_{3-i}^{(a,+)}\Big\}$ as the following union of mutually exclusive events:
\begin{align*}
&\Big\{z_{3-i}^{(a,-)}\xleftrightarrow{B}z_{3-i}^{(a,+)}\Big\}= \Big\{z_{3-i}^{(a,-)}\xleftrightarrow{B;B_1^{(3-i,a)}} z_{3-i}^{(a,+)}\Big\} \cup \Big(\cup_{m=2}^M \Big\{z_{3-i}^{(a,-)}\xleftrightarrow[B_{m-1}^{(3-i,a)}]{B;B_{m}^{(3-i,a)}}z_{3-i}^{(a,+)}\Big\}\Big)\\
&\qquad\qquad\qquad\qquad\qquad\qquad\qquad\cup \Big\{z_{3-i}^{(a,-)}\xleftrightarrow[B_M^{(3-i,a)}]{B}z_{3-i}^{(a,-)}\Big\}.
\end{align*}
The factor $(\pi_a)^2$ comes from the one-arm events at $z^a_3$ and $z^a_4$, the factor $a^{5/2}=\big(a^{5/4}\big)^2$ comes from the four-arm events at $z_1^a$ and $z^a_2$, and the logarithm comes from the union $\cup_{m=2}^M \Big\{z_{3-i}^{(a,-)}\xleftrightarrow[B_{m-1}^{(3-i,a)}]{B;B_{m}^{(3-i,a)}}z_{3-i}^{(a,+)}\Big\}$ of disjoint events, which results in a sum of probabilities of order $a^{5/4}$ (recall that $M$ is of order $\vert \log a \vert$.)

\medskip

\noindent\textbf{Step 2.} A simple calculation shows that we can write
\begin{align*}
	\langle \mathcal{E}_{z_1^a}\mathcal{E}_{z_2^a}S_{z_3^a}S_{z_4^a}\rangle^a =&\underbrace{\mathbb{P}^a\Big[z_1^{(a,-)}\xleftrightarrow{B}z_1^{(a,+)},\; z_2^{(a,-)}\xleftrightarrow{B}z_2^{(a,+)},\; z_3^a\xleftrightarrow{B}z_4^a\Big]}_{R_1^a}\\
	&+\underbrace{\mathbb{P}^a\Big[z_1^{(a,-)}\xleftrightarrow{B}z_1^{(a,+)}\Big] \mathbb{P}^a\Big[z_2^{(a,-)}\xleftrightarrow{B}z_2^{(a,+)}\Big] \mathbb{P}^a\Big[ z_3^a\xleftrightarrow{B}z_4^a\Big]}_{R_2^a}\\
	&-\underbrace{\mathbb{P}^a\Big[z_2^{(a,-)}\xleftrightarrow{B}z_2^{(a,+)}\Big]\mathbb{P}^a\Big[ z_1^{(a,-)}\xleftrightarrow{B}z_1^{(a,+)},\; z_3^a\xleftrightarrow{B}z_4^a\Big]}_{R_3^a}\\
	&-\underbrace{\mathbb{P}^a\Big[z_1^{(a,-)}\xleftrightarrow{B}z_1^{(a,+)}\Big]\mathbb{P}^a\Big[ z_2^{(a,-)}\xleftrightarrow{B}z_2^{(a,+)},\; z_3^a\xleftrightarrow{B}z_4^a\Big]}_{R_4^a}+R_5^a+R_6^a,
\end{align*}
where $R_5^a$ is a sum of differences studied in Step 1, and 
\begin{align*}
		& R_6^a=\mathbb{P}^a\left[z_1^{(a,-)}\xleftrightarrow{B} z_2^{(a,-)}\centernot{\xleftrightarrow{B}}z_1^{(a,+)}\xleftrightarrow{B}z_2^{(a,+)},\; z_3^a\xleftrightarrow{B}z_4^a\right]\\
        &\qquad+\mathbb{P}^a\left[z_1^{(a,-)}\xleftrightarrow{B} z_2^{(a,+)}\centernot{\xleftrightarrow{B}}z_1^{(a,+)}\xleftrightarrow{B}z_2^{(a,-)},\; z_3^a\xleftrightarrow{B}z_4^a\right]\\
        &\qquad+\mathbb{P}^a\left[z_1^{(a,-)}\xleftrightarrow{B}z_2^{(a,-)}\centernot{\xleftrightarrow{B}}z_1^{(a,+)}\xleftrightarrow{B}z_3^a\centernot{\xleftrightarrow{B}}z_2^{(a,+)}\xleftrightarrow{B}z_4^a\right]\\
        &\qquad+\mathbb{P}^a\left[z_1^{(a,-)}\xleftrightarrow{B}z_2^{(a,-)}\centernot{\xleftrightarrow{B}}z_1^{(a,+)}\xleftrightarrow{B}z_4^a\centernot{\xleftrightarrow{B}}z_2^{(a,+)}\xleftrightarrow{B}z_3^a\right]\\
        &\qquad+\mathbb{P}^a\left[z_1^{(a,-)}\xleftrightarrow{B}z_2^{(a,+)}\centernot{\xleftrightarrow{B}}z_1^{(a,+)}\xleftrightarrow{B}z_3^a\centernot{\xleftrightarrow{B}}z_2^{(a,-)}\xleftrightarrow{B}z_4^a\right]\\
        &\qquad+\mathbb{P}^a\left[z_1^{(a,-)}\xleftrightarrow{B}z_2^{(a,+)}\centernot{\xleftrightarrow{B}}z_1^{(a,+)}\xleftrightarrow{B}z_4^a\centernot{\xleftrightarrow{B}}z_2^{(a,-)}\xleftrightarrow{B}z_3^a\right]\\
        &\qquad+\mathbb{P}^a\left[z_1^{(a,-)}\xleftrightarrow{B}z_3^a\centernot{\xleftrightarrow{B}}z_1^{(a,+)}\xleftrightarrow{B}z_2^{(a,+)}\centernot{\xleftrightarrow{B}}z_2^{(a,-)}\xleftrightarrow{B}z_4^a\right]\\
        &\qquad+\mathbb{P}^a\left[z_1^{(a,-)}\xleftrightarrow{B}z_3^a\centernot{\xleftrightarrow{B}}z_1^{(a,+)}\xleftrightarrow{B}z_2^{(a,-)}\centernot{\xleftrightarrow{B}}z_2^{(a,+)}\xleftrightarrow{B}z_4^a\right]\\
        &\qquad+\mathbb{P}^a\left[z_1^{(a,-)}\xleftrightarrow{B}z_4^a\centernot{\xleftrightarrow{B}}z_1^{(a,+)}\xleftrightarrow{B}z_2^{(a,+)}\centernot{\xleftrightarrow{B}}z_2^{(a,-)}\xleftrightarrow{B}z_3^a\right]\\ &\qquad+\mathbb{P}^a\left[z_1^{(a,-)}\xleftrightarrow{B}z_4^a\centernot{\xleftrightarrow{B}}z_1^{(a,+)}\xleftrightarrow{B}z_2^{(a,-)}\centernot{\xleftrightarrow{B}}z_2^{(a,+)}\xleftrightarrow{B}z_3^a\right].
\end{align*}
It follows from Step 1 that 
\begin{align*}
	|R_5^a|\lesssim a^{\frac{5}{2}}|\log a |\pi_a^2. 
\end{align*} 
Moreover, we have 
\begin{align*}
R_6^a\lesssim  a^{\frac{5}{2}} \pi_a^2,
\end{align*}
where the factor $\pi_a^2$ comes from the one-arm events at $z_3^a$ and $z_4^a$, and the factor $a^{5/2}=(a^{5/4})^2$ comes from the four-arm events at $z_1^a$ and $z_2^a$.
Therefore, we conclude that
\begin{align} \label{eqn::three_aux1} 
		\langle \mathcal{E}_{z_1^a}\mathcal{E}_{z_2^a}S_{z_3^a}S_{z_4^a}\rangle^a=\sum_{i=1}^4 R_i^a +O\big(a^{\frac{5}{2}}|\log a|\pi_a^2\big).
\end{align}

\noindent\textbf{Step 3.} With the convention that $B_{M+1}^{(1,a)}=B_{M+1}^{(2,a)}=\mathbb{C}$ and $B_0^{(1,a)}=B_{0}^{(2,a)}=\emptyset$, we can write
\begin{align*}
\sum_{i=1}^4	R_i^a=&\sum_{m=1}^{M+1}\sum_{j=1}^{M+1} \underbrace{\mathbb{P}^a\Big[z_1^{(a,-)}\xleftrightarrow[B_{m-1}^{(1,a)}]{B;B_m^{(1,a)}} z_1^{(a,+)},\; z_2^{(a,-)}\xleftrightarrow[B_{j-1}^{(2,a)}]{B;B_j^{(2,a)}}z_2^{(a,+)},\; z_3^a\xleftrightarrow{B}z_4^a\Big]}_{T_{m,j}^{(1,a)}}\\
&\qquad+\sum_{m=1}^{M+1}\sum_{j=1}^{M+1}\underbrace{\mathbb{P}^a\Big[z_1^{(a,-)}\xleftrightarrow[B_{m-1}^{(1,a)}]{B;B_m^{(1,a)}} z_1^{(a,+)}\Big]\mathbb{P}^a\Big[ z_2^{(a,-)}\xleftrightarrow[B_{j-1}^{(2,a)}]{B;B_j^{(2,a)}}z_2^{(a,+)}\Big]\mathbb{P}^a\Big[ z_3^a\xleftrightarrow{B}z_4^a\Big]}_{T_{m,j}^{(2,a)}}\\
&\qquad-\sum_{m=1}^{M+1}\sum_{j=1}^{M+1}\underbrace{\mathbb{P}^a\Big[z_2^{(a,-)}\xleftrightarrow[B_{j-1}^{(2,a)}]{B;B_j^{(2,a)}} z_2^{(a,+)}\Big]\mathbb{P}^a\Big[ z_1^{(a,-)}\xleftrightarrow[B_{m-1}^{(1,a)}]{B;B_m^{(1,a)}}z_1^{(a,+)},\; z_3^a\xleftrightarrow{B}z_4^a\Big]}_{T_{m,j}^{(3,a)}}\\
&\qquad- \sum_{m=1}^{M+1}\sum_{j=1}^{M+1}\underbrace{\mathbb{P}^a\Big[z_1^{(a,-)}\xleftrightarrow[B_{m-1}^{(1,a)}]{B;B_m^{(1,a)}} z_1^{(a,+)}\Big]\mathbb{P}^a\Big[ z_2^{(a,-)}\xleftrightarrow[B_{j-1}^{(2,a)}]{B;B_j^{(2,a)}}z_2^{(a,+)},\; z_3^a\xleftrightarrow{B}z_4^a\Big]}_{T_{m,j}^{(4,a)}}.
\end{align*}
 
Let $C,D$ be two subsets of the plane, define
\[\Big\{z_3^a\xleftrightarrow[C\oplus D]{B}z_4^a\Big\}:=\Big\{z_3^a\xleftrightarrow[C]{B}z_4^a\Big\}\cap \Big\{z_3^a\xleftrightarrow[D]{B}z_4^a\Big\}.\]

Fix $1\leq m,j\leq M+1$ and note that we can express the event $\big\{z_3^a\xleftrightarrow{B}z_4^a\big\}$ as the following union of mutually exclusive events:
\begin{align*}
	\big\{z_3^a\xleftrightarrow{B}z_4^a\big\}=&\underbrace{\Big\{z_3^a\xleftrightarrow{B;\big(B_{m}^{(1,a)}\cup B_{j}^{(2,a)}\big)^c}z_4^a\Big\}}_{E_{m,j}^{(1,a)}}\cup \underbrace{\Big\{z_3^a\xleftrightarrow[\big(B_{j}^{(2,a)}\big)^c]{B;\big(B_{m}^{(1,a)}\big)^c}z_4^a\Big\}}_{E_{m,j}^{(2,a)}}\cup \underbrace{\Big\{z_3^a\xleftrightarrow[\big(B_{m}^{(1,a)}\big)^c]{B;\big(B_{j}^{(2,a)}\big)^c}z_4^a\Big\}}_{E_{m,j}^{(3,a)}}\\
&	\qquad\qquad\qquad\qquad\qquad\cup \underbrace{\Big\{z_3^a\xleftrightarrow[\big(B_{m}^{(1,a)}\big)^c\oplus\big(B_{j}^{(2,a)}\big)^c]{B}z_4^a\Big\}}_{E_{m,j}^{(4,a)}}.
\end{align*}
Due to the independence of labels at different vertices in percolation, for $1\leq i\leq 3$, we have 
\begin{align*}
	&\mathbb{P}^a\Big[z_1^{(a,-)}\xleftrightarrow[B_{m-1}^{(1,a)}]{B;B_m^{(1,a)}} z_1^{(a,+)},\; z_2^{(a,-)}\xleftrightarrow[B_{j-1}^{(2,a)}]{B;B_j^{(2,a)}}z_2^{(a,+)},\; E_{m,j}^{(i,a)}\Big]\\
	&+\mathbb{P}^a\Big[z_1^{(a,-)}\xleftrightarrow[B_{m-1}^{(1,a)}]{B;B_m^{(1,a)}} z_1^{(a,+)}\Big]\mathbb{P}^a\Big[ z_2^{(a,-)}\xleftrightarrow[B_{j-1}^{(2,a)}]{B;B_j^{(2,a)}}z_2^{(a,+)}\Big]\mathbb{P}^a\Big[E_{m,j}^{(i,a)}\Big]\\
	&-\mathbb{P}^a\Big[z_2^{(a,-)}\xleftrightarrow[B_{j-1}^{(2,a)}]{B;B_j^{(2,a)}} z_2^{(a,+)}\Big]\mathbb{P}^a\Big[ z_1^{(a,-)}\xleftrightarrow[B_{m-1}^{(1,a)}]{B;B_m^{(1,a)}}z_1^{(a,+)},\; E_{m,j}^{(i,a)}\Big]\\
	&-\mathbb{P}^a\Big[z_1^{(a,-)}\xleftrightarrow[B_{m-1}^{(1,a)}]{B;B_m^{(1,a)}} z_1^{(a,+)}\Big]\mathbb{P}^a\Big[ z_2^{(a,-)}\xleftrightarrow[B_{j-1}^{(2,a)}]{B;B_j^{(2,a)}}z_2^{(a,+)},\; E_{m,j}^{(i,a)}\Big]=0.
\end{align*}
Consequently, we have
\begin{align} \label{eqn::three_aux2}
\begin{split}
\sum_{i=1}^4 R_i^a&=\sum_{m=1}^{M+1}\sum_{j=1}^{M+1} \underbrace{\mathbb{P}^a\Big[z_1^{(a,-)}\xleftrightarrow[B_{m-1}^{(1,a)}]{B;B_m^{(1,a)}} z_1^{(a,+)},\; z_2^{(a,-)}\xleftrightarrow[B_{j-1}^{(2,a)}]{B;B_j^{(2,a)}}z_2^{(a,+)},\; E_{m,j}^{(4,a)}\Big]}_{\hat{T}_{m,j}^{(1,a)}}\\
	&\quad+\sum_{m=1}^{M+1}\sum_{j=1}^{M+1} \underbrace{\mathbb{P}^a\Big[z_1^{(a,-)}\xleftrightarrow[B_{m-1}^{(1,a)}]{B;B_m^{(1,a)}} z_1^{(a,+)}\Big]\mathbb{P}^a\Big[ z_2^{(a,-)}\xleftrightarrow[B_{j-1}^{(2,a)}]{B;B_j^{(2,a)}}z_2^{(a,+)}\Big]\mathbb{P}^a\Big[ E_{m,j}^{(4,a)}\Big]}_{\hat{T}_{m,j}^{(2,a)}}\\
	&\quad- \sum_{m=1}^{M+1}\sum_{j=1}^{M+1} \underbrace{\mathbb{P}^a\Big[z_2^{(a,-)}\xleftrightarrow[B_{j-1}^{(2,a)}]{B;B_j^{(2,a)}} z_2^{(a,+)}\Big]\mathbb{P}^a\Big[ z_1^{(a,-)}\xleftrightarrow[B_{m-1}^{(1,a)}]{B;B_m^{(1,a)}}z_1^{(a,+)},\; E_{m,j}^{(4,a)}\Big]}_{\hat{T}_{m,j}^{(3,a)}}\\
	&\quad- \sum_{m=1}^{M+1}\sum_{j=1}^{M+1} \underbrace{\mathbb{P}^a\Big[z_1^{(a,-)}\xleftrightarrow[B_{m-1}^{(1,a)}]{B;B_m^{(1,a)}} z_1^{(a,+)}\Big]\mathbb{P}^a\Big[ z_2^{(a,-)}\xleftrightarrow[B_{j-1}^{(2,a)}]{B;B_j^{(2,a)}}z_2^{(a,+)},\; E_{m,j}^{(4,a)}\Big]}_{\hat{T}_{m,j}^{(4,a)}}.
\end{split}
\end{align}

See Figure~\ref{fig::log_geometry_energy_three_point} for an illustration of the event
\begin{align*}
&	\Big\{z_1^{(a,-)}\xleftrightarrow[B_{m-1}^{(1,a)}]{B;B_m^{(1,a)}} z_1^{(a,+)},\; z_2^{(a,-)}\xleftrightarrow[B_{j-1}^{(2,a)}]{B;B_j^{(2,a)}}z_2^{(a,+)},\; E_{m,j}^{(4,a)}\Big\}\\
&	\qquad\qquad\qquad =\Big\{ z_1^{(a,-)} \xleftrightarrow[B_{m-1}^{(1,a)}]{B;B_m^{(1,a)}} z_1^{(a,+)},\; z_2^{(a,-)}\xleftrightarrow[B_{j-1}^{(2,a)}]{B;B_j^{(2,a)}} z_2^{(a,+)},\; z_3^a\xleftrightarrow[\big(B_{m}^{(1,a)}\big)^c\oplus \big(B_{j}^{(2,a)}\big)^c]{B}z_4^a \Big\}
\end{align*}
that appears in the definition of the term $\hat{T}_{m,j}^{(1,a)}$.
    \begin{figure}
	\includegraphics[width= 0.7\textwidth]{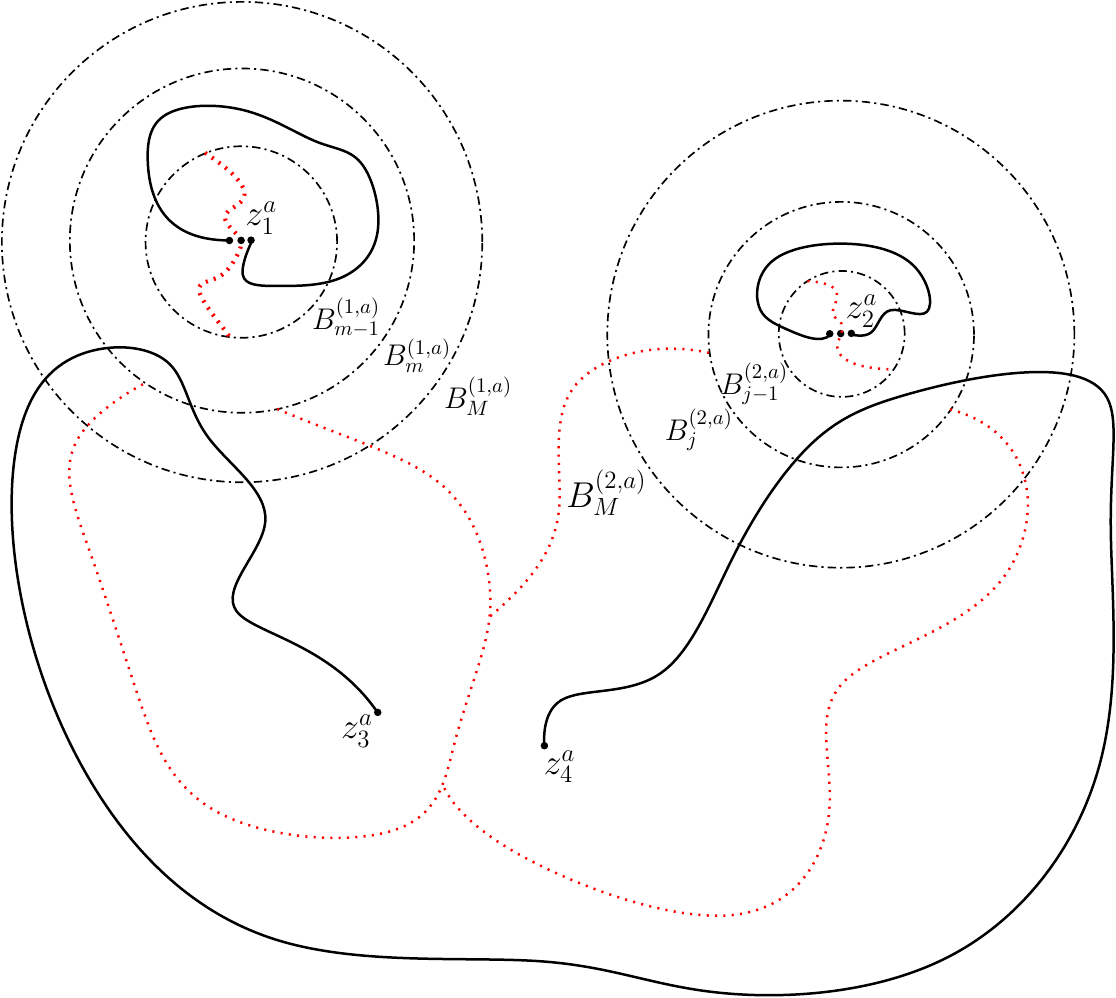}
	\caption{The event $\Big\{ z_1^{(a,-)} \xleftrightarrow[B_{m-1}^{(1,a)}]{B;B_m^{(1,a)}} z_1^{(a,+)},\; z_2^{(a,-)}\xleftrightarrow[B_{j-1}^{(2,a)}]{B;B_j^{(2,a)}} z_2^{(a,+)},\; z_3^a\xleftrightarrow[\big(B_{m}^{(1,a)}\big)^c\oplus \big(B_{j}^{(2,a)}\big)^c]{B}z_4^a \Big\}$. The black, solid lines denote black paths, while the red, dotted lines denote white paths. $z_1^{(a,-)}$ and $z_1^{(a,+)}$ are contained in $B_{m-1}^{(1,a)}$. They are not connected by a black path within the disk $B_{m-1}^{(1,a)}$, but are connected within the larger disk $B_{m}^{(1,a)}$, with radius twice that of $B_{m-1}^{(1,a)}$. 
		The points $z_2^{(a,-)}$ and $z_2^{(a,+)}$ are contained in $B_{j-1}^{(2,a)}$. They are not connected by a black path within the disk $B_{j-1}^{(2,a)}$, but are connected within the larger disk $B_{j}^{(2,a)}$, with radius twice that of $B_{j-1}^{(2,a)}$.	
		The points $z_3^a$ and $z_4^a$ are connected by a black path, but not outside $B_{m}^{(1,a)}$ nor outside $ B_j^{(2,a)}$.
		The number $M$ is of order $-\log a$.
	}
	\label{fig::log_geometry_energy_three_point}
\end{figure}

\noindent\textbf{Step 4.}
For $0<r_1,r_2<\epsilon$, we define the event $\hat{\mathcal{G}}_{r_1,r_2}^{a}(z_1^a,z_2^a,z_3^a,z_4^a)$ by
\begin{align*}
 \hat{\mathcal{G}}_{r_1,r_2}^{a}(z_1^a,z_2^a,z_3^a,z_4^a):=& \Big\{z_3^a\xleftrightarrow{B;(B_{r_1}(z_1^a)\cup B_{r_2}(z_2^a))^c}\partial B_{r_1}(z_1^a) \Big\} \cap \Big\{z_3^a\xleftrightarrow{B;(B_{r_1}(z_1^a)\cup B_{r_2}(z_2^a))^c}\partial B_{r_2}(z_2^a)\Big\}^c\\
    &\cap \Big\{z_4^a\xleftrightarrow{B;(B_{r_1}(z_1^a)\cup B_{r_2}(z_2^a))^c}\partial B_{r_2}(z_2^a) \Big\} \cap \Big\{ z_4^a\xleftrightarrow{B;(B_{r_1}(z_1^a)\cup B_{r_2}(z_2^a))^c}\partial B_{r_1}(z_1^a) \Big\}^c \\
    &\cap \Big\{\partial B_{r_1}(z_1^a)\xleftrightarrow{B}\partial B_{r_2}(z_2^a)\Big\} \cap \Big\{z_3^a\xleftrightarrow{B;(B_{r_1}(z_1^a)\cup B_{r_2}(z_2^a))^c} z_4^a\Big\}^c
\end{align*}
and we further define
\begin{align} \label{eqn::def_event_G}
\mathcal{G}_{r_1,r_2}^a(z_1^a,z_2^a,z_3^a,z_4^a):=\hat{\mathcal{G}}_{r_1,r_2}^{a}(z_1^a,z_2^a,z_3^a,z_4^a)\cup \hat{\mathcal{G}}_{r_2,r_1}^{a}(z_2^a,z_1^a,z_3^a,z_4^a).
\end{align}
See Figure~\ref{fig::outer_event} for illustrations of the events $\hat{\mathcal{G}}_{2^{m+1}a,2^{j+1}a}^{a}(z_1^a,z_2^a,z_3^a,z_4^a)$ and  $\hat{\mathcal{G}}_{2^{j+1}a,2^{m+1}a}^{a}(z_2^a,z_1^a,z_3^a,z_4^a)$.
\begin{figure}
\begin{subfigure}{1\textwidth}
		\begin{center}
\includegraphics[width=0.67\textwidth]{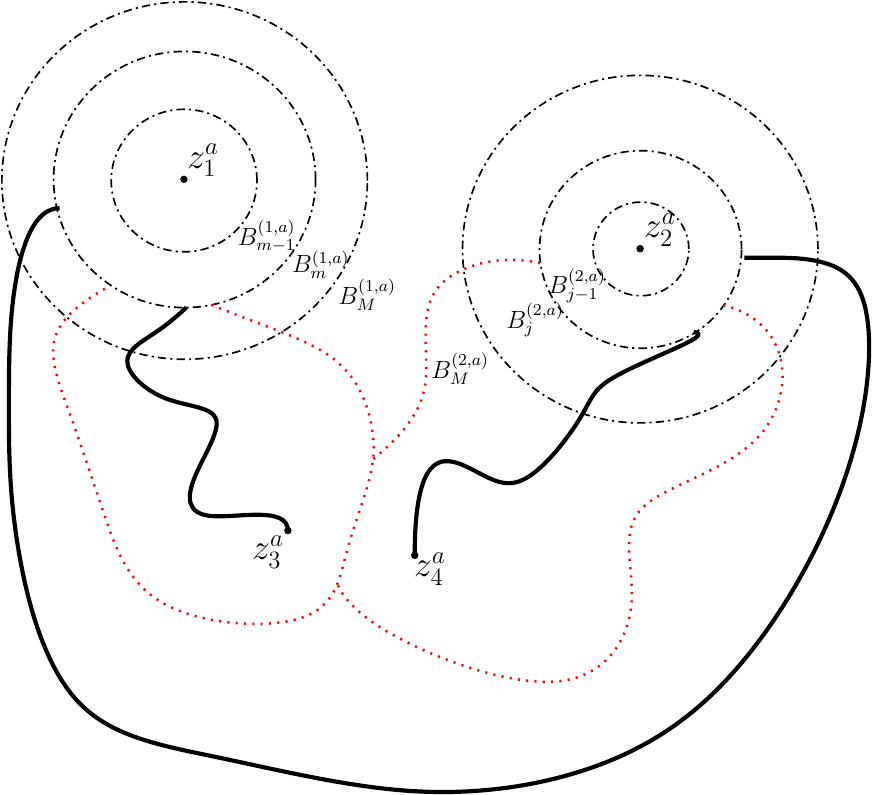}
		\end{center}
		\caption{The event $\hat{\mathcal{G}}_{2^{m+1}a,2^{j+1}a}^{a}(z_1^a,z_2^a,z_3^a,z_4^a)$; }
	\end{subfigure}
 \begin{subfigure}{1\textwidth}
		\begin{center}
\includegraphics[width=0.67\textwidth]{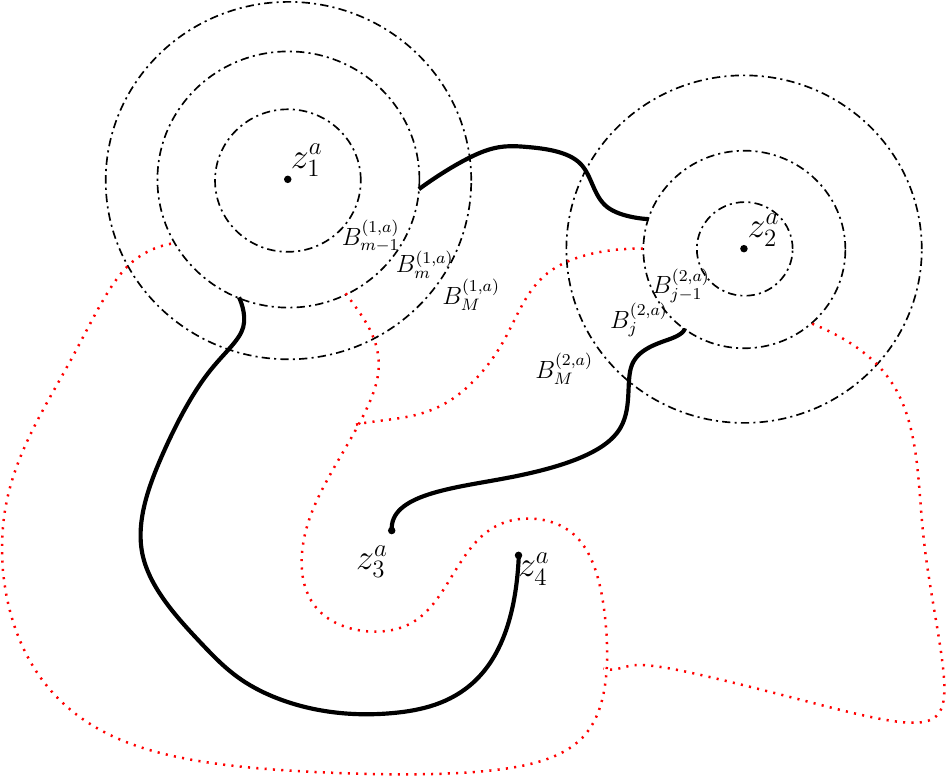}
		\end{center}
		\caption{The event $\hat{\mathcal{G}}_{2^{j+1}a,2^{m+1}a}^{a}(z_2^a,z_1^a,z_3^a,z_4^a)$; }
	\end{subfigure}
    \caption{The events $\hat{\mathcal{G}}_{2^{m+1}a,2^{j+1}a}^{a}(z_1^a,z_2^a,z_3^a,z_4^a)$ and  $\hat{\mathcal{G}}_{2^{j+1}a,2^{m+1}a}^{a}(z_2^a,z_1^a,z_3^a,z_4^a)$. The black solid lines denote black paths, while the red, dotted lines denote white paths.}
    \label{fig::outer_event}
\end{figure}
We have the following lemma.
\begin{lemma} \label{lem::four_point_energy_spin_aux}
Let $z_1,\ldots,z_4\in \mathbb{C}$ be distinct. There exists a positive-valued function $U(z_1,z_2,z_3,z_4)$ such that, for any $x,y\in (0,1)$ and $m\in \{\lfloor Mx\rfloor,\, \lceil Mx \rceil\}$, $j\in \{\lfloor My\rfloor,\, \lceil My\rceil\}$, the terms $T_{m,j}^{(i,a)}$ defined in ~\eqref{eqn::three_aux2} satisfy
\begin{align} \label{eqn::three_aux3}
\lim_{a\to 0}a^{-\frac{5}{2}}\pi_a^{-2}\times\Big(\sum_{i=1}^2 \hat{T}_{m,j}^{(i,a)}-\sum_{i=3}^4 \hat{T}_{m,j}^{(i,a)}\Big)= U(z_1,z_2,z_3,z_4)>0.
\end{align}
Moreover, we have $U(z_1,z_2,z_3,z_4)=\big(V_3-\frac{1}{2}\big)^2 \times \hat{U}(z_1,z_2,z_3,z_4)$, where $V_3>\frac{1}{2}$ is the constant in~\cite[Lemmas~3.5 and~3.6]{CF24} and
      \begin{equation} \label{eqn::def_hat_U}
\begin{split}
                \hat{U}(z_1,z_2,z_3,z_4):=&\lim_{a\to 0} a^{-\frac{5}{2}}\pi_a^{-2}\times \mathbb{P}^a\Big[z_1^{(a,-)}\xleftrightarrow[B_{m-1}^{(1,a)}]{B;B_m^{(1,a)}}z_1^{(a,+)}\Big]\times  \mathbb{P}^a\Big[z_2^{(a,-)}\xleftrightarrow[B_{j-1}^{(2,a)}]{B;B_j^{(2,a)}}z_2^{(a,+)}\Big] \\
            &\qquad\qquad\qquad\times \mathbb{P}^a\Big[\mathcal{G}^a_{2^{m+1}a,2^{j+1}a}(z_1^a,z_2^a,z_3^a,z_4^a)\Big]\\
            = & \; C \, \epsilon^{-\frac{5}{2}-\frac{5}{24}}\lim_{\eta\to 0} P\Big[\mathcal{G}_{\eta,\eta}(z_1,z_2,z_3,z_4)\big | \mathcal{F}_{\eta,\epsilon}(z_1),\, \mathcal{F}_{\eta,\epsilon}(z_2),\, z_3\xleftrightarrow{B}\partial B_{\epsilon}(z_3),\, z_4\xleftrightarrow{B} \partial B_{\epsilon}(z_4)\Big],
\end{split}
        \end{equation}
        where $C\in (0,\infty)$ is a constant and \begin{align*}
        &P\Big[\mathcal{G}_{\eta,\eta}(z_1,z_2,z_3,z_4)\big | \mathcal{F}_{\eta,\epsilon}(z_1),\, \mathcal{F}_{\eta,\epsilon}(z_2),\, z_3\xleftrightarrow{B}\partial B_{\epsilon}(z_3),\, z_4\xleftrightarrow{B} \partial B_{\epsilon}(z_4)\Big]\\
        &\quad:=\lim_{a\to 0} \mathbb{P}^a\Big[\mathcal{G}^a_{\eta,\eta}(z^a_1,z^a_2,z^a_3,z^a_4)\big | \mathcal{F}_{\eta,\epsilon}(z_1^a),\, \mathcal{F}_{\eta,\epsilon}(z_2^a),\, z_3^a\xleftrightarrow{B}\partial B_{\epsilon}(z_3^a),\, z_4^a\xleftrightarrow{B}\partial B_{\epsilon}(z_4^a)\Big] \in (0,\infty).
    \end{align*}
\end{lemma}


We postpone the proof of Lemma~\ref{lem::four_point_energy_spin_aux} and use Steps 1-4 to complete the proof of~\eqref{eqn::mixed_correlation_spin_energy}.
It follows from a standard application of RSW estimates (see, e.g., the proofs of Lemmas 2.1 and 2.2 of~\cite{CamiaNewman2009ising}) that
\begin{align} \label{eqn::three_aux4}
	\Big|a^{-\frac{5}{2}}\pi_a^{-2}\times \Big(\sum_{i=1}^2 \hat{T}_{m,j}^{(i,a)}-\sum_{i=3}^4 \hat{T}_{m,j}^{(i,a)} \Big)\Big|\lesssim 1,\quad \text{for all } 1\leq m,j\leq M+1.
\end{align}
Recall that $M=M(a,\epsilon)$ in~\eqref{eqn::three_aux2} is the largest integer such that $2^{(M+1)} (2a)\leq \epsilon$. Hence, for any sufficiently small $\epsilon>0$,
combining~\eqref{eqn::three_aux1}--\eqref{eqn::three_aux4} and using the dominated convergence theorem, we obtain
\begin{align} \label{eqn::three_aux5}
\langle \phi(z_1)\phi(z_2)\psi(z_3)\psi(z_4)\rangle:=	\lim_{a\to 0} a^{-\frac{5}{2}} |\log a|^{-2}\pi_a^{-2}\times \langle \mathcal{E}_{z_1^a}\mathcal{E}_{z_2^a}S_{z_3^a}S_{z_4^a}\rangle^a= \frac{1}{(\log 2)^2} U(z_1,z_2,z_3,z_4)\in (0,\infty),
\end{align}
which gives~\eqref{eqn::mixed_correlation_spin_energy} with $\ell =2$ and $k=1$. \end{proof}





\begin{proof}[Proof of Lemma~\ref{lem::four_point_energy_spin_aux}]
    We adopt the same notation as in the proof of~\eqref{eqn::mixed_correlation_spin_energy} in Theorem~\ref{thm::energy_spin} with $\ell=2$ and $k=1$.	 Fix $x,y\in (0,1)$ and $m\in \{\lfloor My\rfloor,\; \lceil My\rceil\}$, $j\in \{\lfloor Mx\rfloor,\; \lceil Mx\rceil\}$. We claim that 
        \begin{align}
            \lim_{a\to 0} a^{-\frac{5}{2}}\pi_a^{-2}\times \hat{T}_{m,j}^{(1,a)}=& \; V_3^2 \times\hat{U}(z_1,z_2,z_3,z_4), \label{eqn::limit_Tmj_1}\\
            \lim_{a\to 0} a^{-\frac{5}{2}}\pi_a^{-2}\times \hat{T}_{m,j}^{(2,a)}=& \; \frac{1}{4}\times\hat{U}(z_1,z_2,z_3,z_4),\label{eqn::limit_Tmj_2}\\
            \lim_{a\to 0} a^{-\frac{5}{2}}\pi_a^{-2}\times \hat{T}_{m,j}^{(3,a)}=& \; \frac{1}{2}V_3\times\hat{U}(z_1,z_2,z_3,z_4),\label{eqn::limit_Tmj_3}\\
             \lim_{a\to 0} a^{-\frac{5}{2}}\pi_a^{-2}\times \hat{T}_{m,j}^{(4,a)}=& \; \frac{1}{2}V_3\times\hat{U}(z_1,z_2,z_3,z_4),\label{eqn::limit_Tmj_4}
        \end{align}
        where $V_3>\frac{1}{2}$ is the constant in~\cite[Lemmas~3.5 and~3.6]{CF24}, the function $\hat{U}(z_1,z_2,z_3,z_4)$ is defined in~\eqref{eqn::def_hat_U}, the factor $\frac{1}{2}$ in~\eqref{eqn::limit_Tmj_3} and~\eqref{eqn::limit_Tmj_4} is the constant $V_2$ in~\cite[Lemmas~3.5 and~3.6]{CF24}, and the factor $\frac{1}{4}$ in~\eqref{eqn::limit_Tmj_2} comes from the square of $V_2=\frac{1}{2}$.
   
        Assuming this claim, we have
        \begin{align*}
            \lim_{a\to 0} a^{-\frac{5}{2}}\pi_a^{-2}\times \Big(\sum_{i=1}^2\hat{T}_{m,j}^{(i,a)}-\sum_{i=3}^4 \hat{T}_{m,j}^{(i,a)}\Big)= \Big(V_3-\frac{1}{2}\Big)^2 \times \hat{U}(z_1,z_2,z_3,z_4),
        \end{align*}
        which gives~\eqref{eqn::three_aux3} with $U(z_1,z_2,z_3,z_4):=\big(V_3-\frac{1}{2}\big)^2\hat{U}(z_1,z_2,z_3,z_4)$.

        We now prove~\eqref{eqn::limit_Tmj_1}.  Note that the event $\mathcal{G}^a_{2^{m+1}a,2^{j+1}a}(z_1^a,z_2^a,z_3^a,z_4^a)$ defined in~\eqref{eqn::def_event_G} implies the event $\mathcal{F}_{2^{m+1}a,\epsilon}(z_1^a)\cap \mathcal{F}_{2^{j+1}a,\epsilon}(z_2^a)$. We define
       \begin{align*}
        \tilde{\mathcal{G}}^a_{2^{m+1}a,2^{j+1}a}(z_1^a,z_2^a,z_3^a,z_4^a):=\mathcal{G}^a_{2^{m+1}a,2^{j+1}a}(z_1^a,z_2^a,z_3^a,z_4^a)\cap \mathring{\mathcal{F}}_{2^{m+1}a,\epsilon}(z_1^a)\cap \mathring{\mathcal{F}}_{2^{j+1}a,\epsilon}(z_2^a).
       \end{align*}
        Note also that the three events 
        \begin{align*}
            \Big\{z_1^{(a,-)}\xleftrightarrow[B_{m-1}^{(1,a)}]{B;B_{m}^{(1,a)}}z_1^{(a,+)}\Big\}, \quad  \Big\{z_2^{(a,-)}\xleftrightarrow[B_{j-1}^{(2,a)}]{B;B_{j}^{(2,a)}}z_2^{(a,+)}\Big\}, \quad \text{and} \quad \mathcal{G}^a_{2^{m+1}a,2^{j+1}a}(z_1^a,z_2^a,z_3^a,z_4^a)
        \end{align*}
        are independent. Combining this with~\eqref{eqn::five_arm_four_arm}, we have 
   \begin{equation} \label{eqn::hatT1}
       \begin{split}
                    & \hat{T}_{m,j}^{(1,a)}=\underbrace{\mathbb{P}^a\Big[z_1^{(a,-)}\xleftrightarrow[B_{m-1}^{(1,a)}]{B;B_{m}^{(1,a)}}z_1^{(a,+)}\Big]\mathbb{P}^a\Big[z_2^{(a,-)}\xleftrightarrow[B_{j-1}^{(2,a)}]{B;B_{j}^{(2,a)}}z_2^{(a,+)}\Big]\mathbb{P}^a\Big[\mathcal{G}^a_{2^{m+1}a,2^{j+1}a}(z_1^a,z_2^a,z_3^a,z_4^a)\Big]}_{U_1^a}\\
            &\qquad\times \Big(\underbrace{\mathbb{P}^a\Big[E_{m,j}^{(4,a)} \big| z_1^{(a,-)}\xleftrightarrow[B_{m-1}^{(1,a)}]{B;B_{m}^{(1,a)}}z_1^{(a,+)},\, z_2^{(a,-)}\xleftrightarrow[B_{j-1}^{(2,a)}]{B;B_{j}^{(2,a)}}z_2^{(a,+)},\, \tilde{\mathcal{G}}^a_{2^{m+1}a,2^{j+1}a}(z_1^a,z_2^a,z_3^a,z_4^a)\Big]}_{U_2^a}+o(1)\Big).
       \end{split}
   \end{equation}

        First, for the term $U_1^a$, we write
        \begin{align*}
            U_1^a =&\;\mathbb{P}^a\Big[\mathcal{F}_{2a,2^ma}(z_1^a)\Big]\mathbb{P}^a\Big[\mathcal{F}_{2^{m+1}a,\epsilon}(z_1^a)\Big] \mathbb{P}^a\Big[\mathcal{F}_{2a,2^ja}(z_2^a)\Big]\mathbb{P}^a\Big[\mathcal{F}_{2^{j+1}a,\epsilon}(z_2^a)\Big]\\
           &\times\mathbb{P}^a\Big[z_3^a\xleftrightarrow{B} \partial B_{\epsilon}(z_3^a)\Big]\mathbb{P}^a\Big[z_4^a\xleftrightarrow{B} \partial B_{\epsilon}(z_4^a)\Big]\\
           &\times \mathbb{P}^a\Big[z_1^{(a,-)}\xleftrightarrow[B_{m-1}^{(1,a)}]{B;B_{m}^{(1,a)}}z_1^{(a,+)}\big| \mathcal{F}_{2a,2^ma}(z_1^a)\Big] \mathbb{P}^a\Big[z_2^{(a,-)}\xleftrightarrow[B_{j-1}^{(2,a)}]{B;B_{j}^{(2,a)}}z_2^{(a,+)}\big| \mathcal{F}_{2a,2^ja}(z_2^a)\Big]\\
           &\times \mathbb{P}^a\Big[\mathcal{G}_{2^{m+1}a,2^{j+1}a}^a(z_1^a,z_2^a,z_3^a,z_4^a)\big| \mathcal{F}_{2^{m+1}a,\epsilon}(z_1^a),\, \mathcal{F}_{2^{j+1}a,\epsilon}(z_2^a),\,z_3^a\xleftrightarrow{B} \partial B_{\epsilon}(z_3^a),\, z_4^a\xleftrightarrow{B} \partial B_{\epsilon}(z_4^a)\Big].
        \end{align*}
        According to~\eqref{eqn::four_arm_quasi}, we have 
        \begin{align*}
            \lim_{a\to 0}a^{-\frac{5}{4}} \times \mathbb{P}^a\Big[\mathcal{F}_{2a,2^ma}(z_1^a)\Big]\mathbb{P}^a\Big[\mathcal{F}_{2^{m+1}a,\epsilon}(z_1^a)\Big] = \lim_{a\to 0}a^{-\frac{5}{4}}\times \mathbb{P}^a\Big[\mathcal{F}_{2a,2^ja}(z_2^a)\Big]\mathbb{P}^a\Big[\mathcal{F}_{2^{j+1}a,\epsilon}(z_2^a)\Big]=C_1^*C_2^*\epsilon^{-\frac{5}{4}},
        \end{align*}
        while~\eqref{eqn::lim=C^*_3} implies
        \begin{align*}
            \lim_{a\to 0}\mathbb{P}^a\Big[z_1^{(a,-)}\xleftrightarrow[B_{m-1}^{(1,a)}]{B;B_{m}^{(1,a)}}z_1^{(a,+)}\big| \mathcal{F}_{2a,2^ma}(z_1^a)\Big] =\lim_{a\to 0}\mathbb{P}^a\Big[z_2^{(a,-)}\xleftrightarrow[B_{j-1}^{(2,a)}]{B;B_{j}^{(2,a)}}z_2^{(a,+)}\big| \mathcal{F}_{2a,2^ja}(z_2^a)\Big]=\tilde{C}_3^*,
        \end{align*}
        for some constant $\tilde{C}_3^*\in(0,\infty)$.
        It follows from~\cite{GarbanPeteSchrammPivotalClusterInterfacePercolation} (see the first limit in the third displayed equation on page 999) that
        \begin{align*}
            \lim_{a\to 0} \pi_a^{-1}\times \mathbb{P}^a\Big[z_3^a\xleftrightarrow{B}\partial B_{\epsilon}(z_3^a)\Big]=\lim_{a\to 0} \pi_a^{-1}\times \mathbb{P}^a\Big[z_4^a\xleftrightarrow{B}\partial B_{\epsilon}(z_4^a)\Big]=\epsilon^{-\frac{5}{48}}.
        \end{align*}
        We furthermore claim that
        \begin{equation} \label{eqn::claim_G_coupling}
\begin{split}
    &\lim_{a\to 0}\mathbb{P}^a\Big[\mathcal{G}_{2^{m+1}a,2^{j+1}a}^a(z_1^a,z_2^a,z_3^a,z_4^a)\big| \mathcal{F}_{2^{m+1}a,\epsilon}(z_1^a),\, \mathcal{F}_{2^{j+1}a,\epsilon}(z_2^a),\,z_3^a\xleftrightarrow{B} \partial B_{\epsilon}(z_3^a),\, z_4^a\xleftrightarrow{B} \partial B_{\epsilon}(z_4^a)\Big]\\
&\quad=\lim_{\eta\to 0} P\Big[\mathcal{G}_{\eta,\eta}(z_1,z_2,z_3,z_4)\big | \mathcal{F}_{\eta,\epsilon}(z_1),\, \mathcal{F}_{\eta,\epsilon}(z_2),\, z_3\xleftrightarrow{B}\partial B_{\epsilon}(z_3),\, z_4\xleftrightarrow{B} \partial B_{\epsilon}(z_4)\Big]\in (0,\infty).
\end{split}
        \end{equation}
    
    Assuming~\eqref{eqn::claim_G_coupling}, we have
 {\small      \begin{align} \label{eqn::limit_Tmj_1_aux1}
            \lim_{a\to 0} a^{-\frac{5}{2}}\pi_{a}^{-2} U_1^a
            =(C_1^*C_2^*\tilde{C}_3^*)^2\epsilon^{-\frac{5}{2}-\frac{5}{24}}\lim_{\eta\to 0} P\Big[\mathcal{G}_{\eta,\eta}(z_1,z_2,z_3,z_4)\big | \mathcal{F}_{\eta,\epsilon}(z_1),\, \mathcal{F}_{\eta,\epsilon}(z_2),\, z_3\xleftrightarrow{B}\partial B_{\epsilon}(z_3),\, z_4\xleftrightarrow{B} \partial B_{\epsilon}(z_4)\Big].
        \end{align}}

We now prove the claim~\eqref{eqn::claim_G_coupling}. The existence of the first limit follows from the same arguments as in the proof of~\cite[Proof of Theorem~1.5]{Cam23}. To prove the existence of the second limit and the equality, let $0<\eta\ll 1$. As in the proof of Lemma~\ref{lem::four_arm_coupling_argument}, one can find a coupling between the two measures 
\begin{align*}
\mathbb{P}^a\Big[\cdot\big| \mathcal{F}_{2^{m+1}a,\epsilon}(z_1^a),\, \mathcal{F}_{2^{j+1}a,\epsilon}(z_2^a),\,z_3^a\xleftrightarrow{B} \partial B_{\epsilon}(z_3^a),\, z_4^a\xleftrightarrow{B} \partial B_{\epsilon}(z_4^a)\Big]
\end{align*}
and
\begin{align*}
    \mathbb{P}^a\Big[\cdot\big| \mathcal{F}_{\eta,\epsilon}(z_1^a),\, \mathcal{F}_{\eta,\epsilon}(z_2^a),\,z_3^a\xleftrightarrow{B} \partial B_{\epsilon}(z_3^a),\, z_4^a\xleftrightarrow{B} \partial B_{\epsilon}(z_4^a)\Big]
\end{align*}
showing that
\begin{align*}
   & \Big| \mathbb{P}^a\Big[\mathcal{G}_{2^{m+1}a,2^{j+1}a}^a(z_1^a,z_2^a,z_3^a,z_4^a)\big| \mathcal{F}_{2^{m+1}a,\epsilon}(z_1^a),\, \mathcal{F}_{2^{j+1}a,\epsilon}(z_2^a),\,z_3^a\xleftrightarrow{B} \partial B_{\epsilon}(z_3^a),\, z_4^a\xleftrightarrow{B} \partial B_{\epsilon}(z_4^a)\Big]\\
    &\quad-\mathbb{P}^a\Big[\mathcal{G}^a_{\eta,\eta}(z^a_1,z^a_2,z^a_3,z_4^a)\big | \mathcal{F}_{\eta,\epsilon}(z_1^a),\, \mathcal{F}_{\eta,\epsilon}(z_2^a),\, z_3^a\xleftrightarrow{B}\partial B_{\epsilon}(z_3^a),\, z_4^a\xleftrightarrow{B}\partial B_{\epsilon}(z_4^a)\Big]\Big|\leq \Big(\min\Big\{\frac{1}{\eta}, 2^{M-m},2^{M-j}\Big\}\Big)^{-c_*}
\end{align*}
for some constant $c_*\in (0,\infty)$. Then, as in the proof of~\eqref{eqn::one_arm_cap_four_arm}, letting $a\to 0$ first and then $\eta\to 0$, we get~\eqref{eqn::claim_G_coupling}, as desired.

Second, for the term $U_2^a$, note that
{\small\begin{equation} \label{eqn::hatT2}
\begin{split}
    U_2^a&=\mathbb{P}^a\Big[\mathcal{C}_{2^{m+1}a,\epsilon}^a(z_1^a),\, \mathcal{C}_{2^{j+1}a,\epsilon}^a(z_2^a) \big| z_1^{(a,-)}\xleftrightarrow[B_{m-1}^{(1,a)}]{B;B_{m}^{(1,a)}}z_1^{(a,+)},\, z_2^{(a,-)}\xleftrightarrow[B_{j-1}^{(2,a)}]{B;B_{j}^{(2,a)}}z_2^{(a,+)},\, \tilde{\mathcal{G}}^a_{2^{m+1}a,2^{j+1}a}(z_1^a,z_2^a,z_3^a,z_4^a)\Big]\\
&=\underbrace{\mathbb{P}^a\Big[\mathcal{C}_{2^{m+1}a,\epsilon}^a(z_1^a) \big| z_1^{(a,-)}\xleftrightarrow[B_{m-1}^{(1,a)}]{B;B_{m}^{(1,a)}}z_1^{(a,+)},\, \tilde{\mathcal{G}}^a_{2^{m+1}a,2^{j+1}a}(z_1^a,z_2^a,z_3^a,z_4^a)\Big]}_{U_{2,1}^a}\\
&\enspace\times \underbrace{\mathbb{P}^a\Big[ \mathcal{C}_{2^{j+1}a,\epsilon}^a(z_2^a)\big|\mathcal{C}_{2^{m+1}a,\epsilon}^a(z_1^a),\, z_1^{(a,-)}\xleftrightarrow[B_{m-1}^{(1,a)}]{B;B_{m}^{(1,a)}}z_1^{(a,+)},\, z_2^{(a,-)}\xleftrightarrow[B_{j-1}^{(2,a)}]{B;B_{j}^{(2,a)}}z_2^{(a,+)},\, \tilde{\mathcal{G}}^a_{2^{m+1}a,2^{j+1}a}(z_1^a,z_2^a,z_3^a,z_4^a)\Big]}_{U_{2,2}^a},
\end{split}
\end{equation}}where we have used the independence of the events $\Big\{z_2^{(a,-)}\xleftrightarrow[B_{j-1}^{(2,a)}]{B;B_{j}^{(2,a)}}z_2^{(a,+)}\Big\}$ and
\[\mathcal{C}_{2^{m+1}a,\epsilon}^a(z_1^a)\cap\Big\{z_1^{(a,-)}\xleftrightarrow[B_{m-1}^{(1,a)}]{B;B_{m}^{(1,a)}}z_1^{(a,+)}\Big\}\cap\tilde{\mathcal{G}}^a_{2^{m+1}a,2^{j+1}a}(z_1^a,z_2^a,z_3^a,z_4^a) \]
to get the last equality. 
For the term $U_{2,1}^a$, when $a$ is small enough, as in the proof of Lemma~\ref{lem::four_arm_coupling_argument}, one can find a coupling between the two measures
\begin{align*}
    \mathbb{P}^a\Big[\cdot\big| z_1^{(a,-)}\xleftrightarrow[B_{m-1}^{(1,a)}]{B;B_{m}^{(1,a)}}z_1^{(a,+)},\, \tilde{\mathcal{G}}^a_{2^{m+1}a,2^{j+1}a}(z_1^a,z_2^a,z_3^a,z_4^a)\Big]
\end{align*}
 and
 \begin{align*}
      \mathbb{P}^a\Big[\cdot\big| z_1^{(a,-)}\xleftrightarrow[B_{m-1}^{(1,a)}]{B;B_{m}^{(1,a)}}z_1^{(a,+)},\, \tilde{\mathcal{F}}^a_{2^{m+1}a,\epsilon}(z_1^a,z_3^a,z_4^a)\Big]
 \end{align*}
 showing that
 \begin{align} \label{eqn::four_point_energy_spin_aux0}
     \Big|U_{2,1}^a- \mathbb{P}^a\Big[\mathcal{C}_{2^{m+1}a,\epsilon}(z_1^a)\big| z_1^{(a,-)}\xleftrightarrow[B_{m-1}^{(1,a)}]{B;B_{m}^{(1,a)}}z_1^{(a,+)},\, \tilde{\mathcal{F}}^a_{2^{m+1}a,\epsilon}(z_1^a,z_3^a,z_4^a)\Big]\Big|\lesssim 2^{-c_*(M-m)}
 \end{align}
 for some constant $c_*\in (0,\infty)$. On the one hand, it follows from~\eqref{eqn::change_of_limits_aux6} and the translation invariance of percolation that
 \begin{align} \label{eqn::four_point_energy_spin_aux1}
     \lim_{a\to 0}\mathbb{P}^a\Big[\mathcal{C}_{2^{m+1}a,\epsilon}(z_1^a)\big| z_1^{(a,-)}\xleftrightarrow[B_{m-1}^{(1,a)}]{B;B_{m}^{(1,a)}}z_1^{(a,+)},\, \tilde{\mathcal{F}}^a_{2^{m+1}a,\epsilon}(z_1^a,z_3^a,z_4^a)\Big]=\lim_{\eta\to 0}P\Big[\mathcal{C}_{2,1/\eta}\big| -\eta\xleftrightarrow[B_1(0)]{B;B_2(0)}\eta,\; \mathring{\mathcal{F}}_{2,1/\eta}(0) \Big].
 \end{align}
 On the other hand, in the continuum, let $z,z'\notin \{z_3,z_4\}$ be distinct points such that $B_{2\epsilon}(z')\cap \{z_3,z_4\}=\emptyset$ and $B_{2\epsilon}(z')\cap \{z_3,z_4\}=\emptyset$, 
 and define $\tilde{B}_{n}:=B_{2^n|z-z'|}(\tfrac{z+z'}{2})$ for $1\leq n\leq \tilde{M}$, where $\tilde{M}=\tilde{M}(z,z')$ is the largest integer such that $2^{\tilde{M}+1}|z-z'|\leq \epsilon$. Let $z^a,(z')^a$ be vertices in $a\mathcal{T}$ such that $z^a\to z$ and $(z')^a\to z'$ as $a\to 0$. For $2\leq n\leq \tilde{M} $, define
 \begin{align*}
     &P\Big[z_3\xleftrightarrow{B}z_4 \big| z \xleftrightarrow[\hat{B}_{n-1}]{B;\tilde{B}_n} z',\, \tilde{\mathcal{F}}_{2^n|z-z'|,\epsilon}\big(\tfrac{z+z'}{2},z_3,z_4\big)\Big]\\
     &\qquad:=\lim_{a\to 0} \pi_a^{-2} \times \mathbb{P}^a\Big[z_3^a\xleftrightarrow{B}z_4^a \big| z ^a\xleftrightarrow[\tilde{B}_{n-1}]{B;\tilde{B}_n} (z')^a,\, \tilde{\mathcal{F}}^a_{2^n|z^a-(z')^a|,\epsilon}\Big(\tfrac{z^a+(z')^a}{2},z_3^a,z_4^a\Big)\Big],
 \end{align*}
 where the existence of the limit can be shown using the strategy in~\cite[Proof of Theorem~1.5]{Cam23}. Fix $n\in \{\lfloor \tilde{M}x \rfloor, \, \lceil \tilde{M}x\rceil\}$. As in the proof of~\eqref{eqn::change_of_limits_aux6}, one can show that
 \begin{align}\label{eqn::four_point_energy_spin_aux2}
     \lim_{z'\to z} P\Big[z_3\xleftrightarrow{B}z_4 \big| z \xleftrightarrow[\tilde{B}_{n-1}]{B;\tilde{B}_n} z',\, \tilde{\mathcal{F}}_{2^n|z-z'|,\epsilon}\big(\tfrac{z+z'}{2},z_3,z_4\big)\Big] =\lim_{\eta\to 0}P\Big[\mathcal{C}_{2,1/\eta}\big| -\eta\xleftrightarrow[B_1(0)]{B;B_2(0)}\eta,\; \mathring{\mathcal{F}}_{2,1/\eta}(0) \Big].
 \end{align}Besides, it follows from~\cite[Lemmas~3.5 and~3.6]{CF24} that
 \begin{align} \label{eqn::four_point_energy_spin_aux3}
     \lim_{z'\to z} P\Big[z_3\xleftrightarrow{B}z_4 \big| z \xleftrightarrow[\tilde{B}_{n-1}]{B;\tilde{B}_n} z',\, \tilde{\mathcal{F}}_{2^n|z-z'|,\epsilon}\big(\tfrac{z+z'}{2},z_3,z_4\big)\Big] =V_3,
 \end{align}
 where $V_3>\frac{1}{2}$ is the constant in~\cite[Lemmas~3.5 and~3.6]{CF24}.
 Combining~\eqref{eqn::four_point_energy_spin_aux0}--\eqref{eqn::four_point_energy_spin_aux3}, we have 
 \begin{align} \label{eqn::four_point_energy_spin_aux4}
     \lim_{a\to 0} U_{2,1}^a=V_3.
 \end{align}

 For the term $U_{2,2}^a$, when $a$ is small enough, as in the proof of Lemma~\ref{lem::four_arm_coupling_argument}, one can find a coupling between the two measures 
 \begin{align*}
     \mathbb{P}^a\Big[ \cdot\big|\mathcal{C}_{2^{m+1}a,\epsilon}^a(z_1^a),\, z_1^{(a,-)}\xleftrightarrow[B_{m-1}^{(1,a)}]{B;B_{m}^{(1,a)}}z_1^{(a,+)},\, z_2^{(a,-)}\xleftrightarrow[B_{j-1}^{(2,a)}]{B;B_{j}^{(2,a)}}z_2^{(a,+)},\, \tilde{\mathcal{G}}^a_{2^{m+1}a,2^{j+1}a}(z_1^a,z_2^a,z_3^a,z_4^a)\Big]
 \end{align*}
 and
 \begin{align*}
     \mathbb{P}^a\Big[ \cdot\big| z_2^{(a,-)}\xleftrightarrow[B_{j-1}^{(2,a)}]{B;B_{j}^{(2,a)}}z_2^{(a,+)},\,\tilde{\mathcal{F}}^a_{2^{j+1}a,\epsilon}(z_2^a,z_3^a,z_4^a)\Big]
 \end{align*}
 showing that
 \begin{align*}
     \Big|U_{2,2}^a- \mathbb{P}^a\Big[z_3^a\xleftrightarrow{B}z_4^a\big| z_2^{(a,-)}\xleftrightarrow[B_{j-1}^{(2,a)}]{B;B_{j}^{(2,a)}}z_2^{(a,+)},\,\tilde{\mathcal{F}}^a_{2^{j+1}a,\epsilon}(z_2^a,z_3^a,z_4^a)\Big]\Big| \lesssim 2^{-c_*(M-j)},
 \end{align*}
 for some constant $c_*\in (0,\infty)$. As in the proof of~\eqref{eqn::four_point_energy_spin_aux4}, we have
 \begin{align*}
      \lim_{a\to 0}\mathbb{P}^a\Big[ z_3^a\xleftrightarrow{B}z_4^a\big| z_2^{(a,-)}\xleftrightarrow[B_{j-1}^{(2,a)}]{B;B_{j}^{(2,a)}}z_2^{(a,+)},\,\tilde{\mathcal{F}}^a_{2^{j+1}a,\epsilon}(z_2^a,z_3^a,z_4^a)\Big]=V_3,
 \end{align*}
 which implies that 
 \begin{align*}
     \lim_{a\to 0} U_{2,2}^a=V_3.
 \end{align*}
 Therefore, we have
 \begin{align} \label{eqn::limit_Tmj_1_aux2}
     \lim_{a\to o}U_2^a= V_3^2. 
 \end{align}

Combining~\eqref{eqn::limit_Tmj_1_aux1} with~\eqref{eqn::limit_Tmj_1_aux2} gives~\eqref{eqn::limit_Tmj_1}.
The proof of~\eqref{eqn::limit_Tmj_2}--\eqref{eqn::limit_Tmj_4} is essentially the same. Indeed, for $\hat{T}_{m,j}^{(i,a)}$, $i=2,3,4$, one has expansions analogous to~\eqref{eqn::hatT1} and~\eqref{eqn::hatT2}. {The only difference is that the quantities corresponding to $U_{2,1}^a$ and $U_{2,2}^a$ now converge to either $\frac{1}{2}$ or $V_3$, leading to the expressions in~\eqref{eqn::limit_Tmj_2}--\eqref{eqn::limit_Tmj_4}.}
\end{proof}

\begin{proof}[Proof of~\eqref{eqn::cov_mixed_correlation_spin_energy} in Theorem~\ref{thm::energy_spin} with $\ell=2$ and $k=1$. ]

For any $\delta>0$, we define
\begin{align*}
    \langle \phi^{\delta}(z_1)\phi^{\delta}(z_2)\psi(z_3)\psi(z_4)\rangle:=\lim_{a\to 0}\pi_a^{-6}\times\langle \mathcal{E}^\delta_{z_1^a}\mathcal{E}_{z_2^a}^\delta S_{z_3^a}S_{z_4^a}\rangle,
\end{align*}
where $\mathcal{E}_{\cdot}^\delta$ is defined in~\eqref{def::Edelta} and the existence of the limit follows from~\cite[Theorem~1.5]{Cam23}. Let $\bar{M}=\bar{M}(\delta,\epsilon)$ be the largest integer such that $2^{\bar{M}+1}(2\delta)\leq \epsilon$. We define $\bar{B}_{m}^{(i)}:=\{w: |w-z_i|\leq 2^m (2\delta)\}$ for $1\leq m\leq \bar{M}$ and $1\leq i\leq 2$. Recall that $z_1^{(a,\delta,\pm )}$, $z_2^{(a,\delta,\pm)}$, $z_3^{a}$ and $z_4^a$ are vertices in $a\mathcal{T}$ satisfying
\begin{align*}
    z_1^{(a,\delta,\pm)} \to z_1\pm \delta,\quad z_2^{(a,\delta,\pm)} \to z_2 \pm \delta,\quad z_{3}^a\to z_3,\quad \text{and}\quad z_4^a\to z_4,\quad\text{as }a\to 0.
\end{align*}
For $i=1,2$, we define
\begin{align*}
   & P\Big[z_i-\delta\xleftrightarrow{B}z_3\centernot{\xleftrightarrow{B}}z_i+\delta\xleftrightarrow{B}z_4,\; z_{3-i}-\delta\xleftrightarrow{B}z_{3-i}+\delta\Big]\\
   & \qquad:=\lim_{a\to 0} \pi_a^{-6}\times  \mathbb{P}^a\Big[z_i^{(a,\delta,-)}\xleftrightarrow{B}z_3^a\centernot{\xleftrightarrow{B}}z_i^{(a,\delta,+)}\xleftrightarrow{B}z_4^a,\; z_{3-i}^{(a,\delta,-)}\xleftrightarrow{B} z_{3-i}^{(a,\delta,+)}\Big],\\
   & P\Big[z_i-\delta\xleftrightarrow{B}z_4\centernot{\xleftrightarrow{B}}z_i+\delta\xleftrightarrow{B}z_3,\; z_{3-i}-\delta\xleftrightarrow{B}z_{3-i}+\delta\Big]\\
   & \qquad:=\lim_{a\to 0} \pi_a^{-6}\times  \mathbb{P}^a\Big[z_i^{(a,\delta,-)}\xleftrightarrow{B}z_4^a\centernot{\xleftrightarrow{B}}z_i^{(a,\delta,+)}\xleftrightarrow{B}z_3^a,\; z_{3-i}^{(a,\delta,-)}\xleftrightarrow{B} z_{3-i}^{(a,\delta,+)}\Big],\\
   &P\Big[z_1-\delta\xleftrightarrow{B}z_1+\delta,\; z_2-\delta\xleftrightarrow{B}z_2+\delta,\; z_3\xleftrightarrow{B}z_4\Big]\\
   &\qquad:= \lim_{a\to 0}\pi_a^{-6}\times \mathbb{P}^a\Big[z_1^{(a,\delta,-)}\xleftrightarrow{B}z_1^{(a,\delta,+)},\; z_2^{(a,\delta,-)}\xleftrightarrow{B}z_2^{(a,\delta,+)},\; z_3^a\xleftrightarrow{B}z_4^a\Big].
\end{align*}
The existence of these limits can be proved using the strategy in the proof of~\cite[Theorem~1.5]{CF24}.

 In order to analyze the function $\langle \phi^{\delta}(z_1)\phi^{\delta}(z_2)\psi(z_3)\psi(z_4)\rangle$, we repeat the arguments in Step 1 - Step 4 of the proof of~\eqref{eqn::mixed_correlation_spin_energy} with $\ell=2$ and $k=1$, presented above.
Using the definitions of the continuum quantities as limits of lattice quantities, the proof of each step is essentially the same as for the corresponding step in the discrete case, so we only list the conclusions, omitting the details of the derivations.
For a heuristic understanding of the exponents appearing in the equations below, the reader is referred to the discussion in the last paragraph of Section~\ref{subsec::def_result}.

\noindent\textbf{Step $1^*$.} For $i=1,2$, we have 
\begin{align*}
		&\Big| P\Big[z_i-\delta\xleftrightarrow{B}z_3\centernot{\xleftrightarrow{B}}z_i+\delta\xleftrightarrow{B}z_4\Big] P\Big[z_{3-i}-\delta\xleftrightarrow{B}z_{3-i}+\delta\Big]\\
	&\qquad\qquad\qquad\qquad	-P\Big[z_i-\delta\xleftrightarrow{B}z_3\centernot{\xleftrightarrow{B}}z_i+\delta\xleftrightarrow{B}z_4,\; z_{3-i}-\delta\xleftrightarrow{B}z_{3-i}+\delta\Big]\Big|\lesssim \delta^{\frac{5}{2}}|\log \delta|\delta^{-\frac{5}{12}}, \\
			&\Big| P\Big[z_i-\delta\xleftrightarrow{B}z_4\centernot{\xleftrightarrow{B}}z_i+\delta\xleftrightarrow{B}z_3\Big] P\Big[z_{3-i}-\delta\xleftrightarrow{B}z_{3-i}+\delta\Big]\\
	&\qquad\qquad\qquad\qquad	-P\Big[z_i-\delta\xleftrightarrow{B}z_4\centernot{\xleftrightarrow{B}}z_i+\delta\xleftrightarrow{B}z_3,\; z_{3-i}-\delta\xleftrightarrow{B}z_{3-i}+\delta\Big]\Big|\lesssim \delta^{\frac{5}{2}}|\log \delta|\delta^{-\frac{5}{12}}.
\end{align*}

\noindent\textbf{Step $2^*$.} We can write
\begin{align*}
	\langle \phi^{\delta}(z_1)\phi^{\delta}(z_2)\psi(z_3)\psi(z_4)\rangle =&\,\underbrace{P\Big[z_1-\delta\xleftrightarrow{B}z_1+\delta,\; z_2-\delta\xleftrightarrow{B}z_2+\delta,\; z_3\xleftrightarrow{B}z_4\Big]}_{R_1}\\
	&+\underbrace{P\Big[z_1-\delta\xleftrightarrow{B}z_1+\delta\Big] P\Big[z_2-\delta\xleftrightarrow{B}z_2+\delta\Big] P\Big[ z_3\xleftrightarrow{B}z_4\Big]}_{R_2}\\
	&-\underbrace{P\Big[z_2-\delta\xleftrightarrow{B}z_2+\delta\Big]P\Big[ z_1-\delta\xleftrightarrow{B}z_1+\delta,\; z_3\xleftrightarrow{B}z_4\Big]}_{R_3}\\
	&-\underbrace{P\Big[z_1-\delta\xleftrightarrow{B}z_1+\delta\Big]P\Big[ z_2-\delta\xleftrightarrow{B}z_2+\delta,\; z_3\xleftrightarrow{B}z_4\Big]}_{R_4}+R_5+R_6,
\end{align*}
where $R_5$ is a sum of differences studied in Step $1^*$, and 
\begin{align*}
		& R_6=P\left[z_1-\delta\xleftrightarrow{B} z_2-\delta\centernot{\xleftrightarrow{B}}z_1+\delta\xleftrightarrow{B}z_2+\delta,\; z_3\xleftrightarrow{B}z_4\right]\\
		&\qquad+P\left[z_1-\delta\xleftrightarrow{B} z_2+\delta\centernot{\xleftrightarrow{B}}z_1+\delta\xleftrightarrow{B}z_2-\delta,\; z_3\xleftrightarrow{B}z_4\right]\\
          &\qquad+P\left[z_1-\delta\xleftrightarrow{B}z_2-\delta\centernot{\xleftrightarrow{B}}z_1+\delta\xleftrightarrow{B}z_3\centernot{\xleftrightarrow{B}}z_2+\delta\xleftrightarrow{B}z_4\right]\\
        &\qquad+P\left[z_1-\delta\xleftrightarrow{B}z_2-\delta\centernot{\xleftrightarrow{B}}z_1+\delta\xleftrightarrow{B}z_4\centernot{\xleftrightarrow{B}}z_2+\delta\xleftrightarrow{B}z_3\right]\\
        &\qquad+P\left[z_1-\delta\xleftrightarrow{B}z_2+\delta\centernot{\xleftrightarrow{B}}z_1+\delta\xleftrightarrow{B}z_3\centernot{\xleftrightarrow{B}}z_2-\delta\xleftrightarrow{B}z_4\right]\\
        &\qquad+P\left[z_1-\delta\xleftrightarrow{B}z_2+\delta\centernot{\xleftrightarrow{B}}z_1+\delta\xleftrightarrow{B}z_4\centernot{\xleftrightarrow{B}}z_2-\delta\xleftrightarrow{B}z_3\right]\\
        &\qquad+P\left[z_1-\delta\xleftrightarrow{B}z_3\centernot{\xleftrightarrow{B}}z_1+\delta\xleftrightarrow{B}z_2+\delta\centernot{\xleftrightarrow{B}}z_2-\delta\xleftrightarrow{B}z_4\right]\\
        &\qquad+P\left[z_1-\delta\xleftrightarrow{B}z_3\centernot{\xleftrightarrow{B}}z_1+\delta\xleftrightarrow{B}z_2-\delta\centernot{\xleftrightarrow{B}}z_2+\delta\xleftrightarrow{B}z_4\right]\\
        &\qquad+P\left[z_1-\delta\xleftrightarrow{B}z_4\centernot{\xleftrightarrow{B}}z_1+\delta\xleftrightarrow{B}z_2+\delta\centernot{\xleftrightarrow{B}}z_2-\delta\xleftrightarrow{B}z_3\right]\\ &\qquad+P\left[z_1-\delta\xleftrightarrow{B}z_4\centernot{\xleftrightarrow{B}}z_1+\delta\xleftrightarrow{B}z_2-\delta\centernot{\xleftrightarrow{B}}z_2+\delta\xleftrightarrow{B}z_3\right]\lesssim \delta^{\frac{5}{2}-\frac{5}{12}}.
	\end{align*}
It follows from Step $1^*$ that 
\begin{align*}
	|R_5|\lesssim \delta^{\frac{5}{2}}|\log \delta|\delta^{-\frac{5}{12}}. 
\end{align*} 
Therefore, we have
\begin{align} \label{eqn::three_aux1_conti} 
\langle \phi^{\delta}(z_1)\phi^{\delta}(z_2)\psi(z_3)\psi(z_4)\rangle =\sum_{i=1}^4 R_i+O\big(\delta^{\frac{5}{2}}|\log\delta|\delta^{-\frac{5}{12}}\big).
\end{align}

\noindent\textbf{Step $3^*$.} With the convention that $\bar{B}_{\bar{M}+1}^{(1)}=\bar{B}_{\bar{M}+1}^{(2)}=\mathbb{C}$ and $\bar{B}_0^{(1)}=\bar{B}_{0}^{(2)}=\emptyset$, we can write
\begin{align} \label{eqn::three_aux2_conti}
\begin{split}
\sum_{i=1}^4 R_i&=\sum_{m=1}^{\bar{M}+1}\sum_{j=1}^{\bar{M}+1} \underbrace{P\Big[z_1-\delta\xleftrightarrow[\bar{B}_{m-1}^{(1)}]{B;\bar{B}_m^{(1)}} z_1+\delta,\; z_2-\delta\xleftrightarrow[\bar{B}_{j-1}^{(2)}]{B;\bar{B}_j^{(2)}}z_2+\delta,\; \bar{E}_{m,j}^{(4)}\Big]}_{\bar{T}_{m,j}^{(1)}}\\
	&\quad+\sum_{m=1}^{\bar{M}+1}\sum_{j=1}^{\bar{M}+1} \underbrace{P\Big[z_1-\delta\xleftrightarrow[\bar{B}_{m-1}^{(1)}]{B;\bar{B}_m^{(1)}} z_1+\delta\Big]P\Big[ z_2-\delta\xleftrightarrow[\bar{B}_{j-1}^{(2)}]{B;\bar{B}_j^{(2)}}z_2+\delta\Big]P\Big[ \bar{E}_{m,j}^{(4)}\Big]}_{\bar{T}_{m,j}^{(2)}}\\
	&\quad- \sum_{m=1}^{\bar{M}+1}\sum_{j=1}^{\bar{M}+1} \underbrace{P\Big[z_2-\delta\xleftrightarrow[\bar{B}_{j-1}^{(2)}]{B;\bar{B}_j^{(2)}} z_2+\delta\Big]P\Big[ z_1-\delta\xleftrightarrow[\bar{B}_{m-1}^{(1)}]{B;\bar{B}_m^{(1)}}z_1+\delta,\; \bar{E}_{m,j}^{(4)}\Big]}_{\bar{T}_{m,j}^{(3)}}\\
	&\quad- \sum_{m=1}^{\bar{M}+1}\sum_{j=1}^{\bar{M}+1} \underbrace{P\Big[z_1-\delta\xleftrightarrow[\bar{B}_{m-1}^{(1)}]{B;\bar{B}_m^{(1)}} z_1+\delta\Big]P\Big[ z_2-\delta\xleftrightarrow[\bar{B}_{j-1}^{(2)}]{B;\bar{B}_j^{(2)}}z_2+\delta,\; \bar{E}_{m,j}^{(4)}\Big]}_{\bar{T}_{m,j}^{(4)}},
\end{split}
\end{align}
where
\begin{align*}
    &P\Big[z_1-\delta\xleftrightarrow[\bar{B}_{m-1}^{(1)}]{B;\bar{B}_m^{(1)}} z_1+\delta,\; z_2-\delta\xleftrightarrow[\bar{B}_{j-1}^{(2)}]{B;\bar{B}_j^{(2)}}z_2+\delta,\; \bar{E}_{m,j}^{(4)}\Big]\\
   & \quad :=\lim_{a\to 0}\pi_a^{-6}\times \mathbb{P}^a\Big[z_1^{(a,\delta,-)}\xleftrightarrow[\bar{B}_{m-1}^{(1)}]{B;\bar{B}_m^{(1)}} z_1^{(a,\delta,+)},\, z_2^{(a,\delta,-)}  \xleftrightarrow[\bar{B}_{j-1}^{(2)}]{B;\bar{B}_j^{(2)}} z_2^{(a,\delta,+)},\, z_3^a  \xleftrightarrow[\big(\bar{B}^{(1)}_{2^{m+1}\delta}(z_1^a)\big)^c\oplus \big(\bar{B}^{(2)}_{2^{j+1}\delta}(z_2^a)\big)^c]{B}z_4^a\Big],\\
   &P\Big[\bar{E}_{m,j}^{(4)}\Big]:=\lim_{a\to 0}\pi_a^{-2}\times \mathbb{P}^a\Big[z_3^a  \xleftrightarrow[\big(\bar{B}^{(1)}_{2^{m+1}\delta}(z_1^a)\big)^c\oplus \big(\bar{B}^{(2)}_{2^{j+1}\delta}(z_2^a)\big)^c]{B}z_4^a\Big],\\
   &P\Big[ z_1-\delta\xleftrightarrow[\bar{B}_{m-1}^{(1)}]{B;\bar{B}_m^{(1)}}z_1+\delta,\; \bar{E}_{m,j}^{(4)}\Big] \\
   & \quad :=\lim_{a\to 0}\pi_a^{-4}\times\mathbb{P}^a\Big[z_1^{(a,\delta,-)}\xleftrightarrow[\bar{B}_{m-1}^{(1)}]{B;\bar{B}_m^{(1)}} z_1^{(a,\delta,+)},\, z_3^a  \xleftrightarrow[\big(\bar{B}^{(1)}_{2^{m+1}\delta}(z_1^a)\big)^c\oplus \big(\bar{B}^{(2)}_{2^{j+1}\delta}(z_2^a)\big)^c]{B}z_4^a\Big],\\
   &P\Big[ z_2-\delta\xleftrightarrow[\bar{B}_{j-1}^{(2)}]{B;\bar{B}_j^{(2)}}z_2+\delta,\; \bar{E}_{m,j}^{(4)}\Big] \\
   & \quad :=\lim_{a\to 0}\pi_a^{-4}\times\mathbb{P}^a\Big[z_2^{(a,\delta,-)}\xleftrightarrow[\bar{B}_{j-1}^{(2)}]{B;\bar{B}_j^{(2)}} z_2^{(a,\delta,+)},\, z_3^a  \xleftrightarrow[\big(\bar{B}^{(1)}_{2^{m+1}\delta}(z_1^a)\big)^c\oplus \big(\bar{B}^{(2)}_{2^{j+1}\delta}(z_2^a)\big)^c]{B}z_4^a\Big]
\end{align*}
and the existence of these limits can be shown using the strategy in the proof of~\cite[Theorem~1.5]{CF24}.

\noindent\textbf{Step $4^*$.}
For $0<r_1<r_2<\epsilon$, recall the event $\mathcal{G}_{r_1,r_2}^a(z_1^a,z_2^a,z_3^a,z_4^a)$ defined in~\eqref{eqn::def_event_G}. Fix any $x,y\in (0,1)$ and $m \in \{\lfloor \bar{M}x\rfloor,\, \lceil \bar{M}x \rceil\}$, $j\in \{\lfloor \bar{M}y\rfloor,\, \lceil \bar{M}y\rceil\}$. Then the term $\bar{T}_{m,j}^{(i)}$ defined in~\eqref{eqn::three_aux2_conti} satisfies
\begin{align}
    \begin{split}
        &\lim_{\delta\to 0}\delta^{-\frac{5}{2}+\frac{5}{12}} \times \Big(\sum_{i=1}^2 \bar{T}^{(i)}_{m,j} -\sum_{i=3}^4 \bar{T}_{m,j}^{(i)}\Big)\\
        &\qquad=\Big(V_3-\frac{1}{2}\Big)^2\times \lim_{\delta\to 0}\delta^{-\frac{5}{2}+\frac{5}{12}} \times P\Big[z_1-\delta \xleftrightarrow[\bar{B}_{\bar{m}-1}^{(1)}]{B;\bar{B}_{\bar{m}}^{(1)}}z_1+\delta\Big] P\Big[z_2-\delta \xleftrightarrow[\bar{B}_{j-1}^{(2)}]{B;\bar{B}_{j}^{(2)}}z_2+\delta\Big] \\
        &\qquad\qquad\qquad\quad\times P\Big[\mathcal{G}_{2^{m+1}\delta,2^{j+1}\delta}(z_1,z_2,z_3,z_4)\Big]\\
        &\qquad=\Big(V_3-\frac{1}{2}\Big)^2 \hat{C}\epsilon^{-\frac{5}{2}}\epsilon^{-\frac{5}{24}}\lim_{\eta\to 0} P\Big[\mathcal{G}_{\eta,\eta}(z_1,z_2,z_3,z_4)\big | \mathcal{F}_{\eta,\epsilon}(z_1),\, \mathcal{F}_{\eta,\epsilon}(z_2),\, z_3\xleftrightarrow{B}\partial B_{\epsilon}(z_3),\, z_4\xleftrightarrow{B} \partial B_{\epsilon}(z_4)\Big]\\
        &\qquad=\frac{\hat{C}}{C} U(z_1,z_2,z_3,z_4),
    \end{split}
\end{align}
where $V_3>\frac{1}{2}$ is the constant in~\cite[Lemmas~3.5 and~3.6]{CF24}, $\hat{C}\in (0,\infty)$ is a constant, $C$ is the constant in Lemma~\ref{lem::four_point_energy_spin_aux}, $U(z_1,z_2,z_3,z_4)$ is the function in Lemma~\ref{lem::four_point_energy_spin_aux}, and
\begin{align*}
P\Big[\mathcal{G}_{2^{m+1}\delta,2^{j+1}\delta}(z_1,z_2,z_3,z_4)\Big]:=\lim_{a\to 0}\pi_a^{-2} \times \mathbb{P}^a\Big[\mathcal{G}^a_{2^{m+1}\delta,2^{m+1}\delta}(z_1^a,z_2^a,z_3^a,z_4^a)\Big],
\end{align*}
and where the existence of the last limit can be shown using the strategy in the proof of~\cite[Theorem~1.5]{Cam23}.

It follows from a standard application of RSW estimates (see, e.g., the proofs of Lemmas 2.1 and 2.2 of~\cite{CamiaNewman2009ising}) that
\begin{align} \label{eqn::three_aux4_conti}
	\Big|\delta^{-\frac{5}{2}+\frac{5}{12}}\times \Big(\sum_{i=1}^2 \bar{T}_{m,j}^{(i)}-\sum_{i=3}^4 \bar{T}_{m,j}^{(i)}\Big)\Big|\lesssim 1,\quad \text{for all } 1\leq m,j\leq \bar{M}+1.
\end{align}
Combining~\eqref{eqn::three_aux1_conti}--\eqref{eqn::three_aux4_conti} with~\eqref{eqn::three_aux5} and using the dominated convergence theorem, we obtain
\begin{align}\label{eqn::three_aux7}
    \lim_{\delta\to 0} \delta^{-\frac{25}{12}} |\log\delta|^{-2}\times \langle \phi^{\delta}(z_1)\phi^{\delta}(z_2)\psi(z_3)\psi(z_4)\rangle=\frac{\hat{C}}{C} \langle \phi(z_1)\phi(z_2)\psi(z_3)\psi(z_4)\rangle. 
\end{align}
    Then the desired M\"obius covariance rule for $\langle \phi(z_1)\phi(z_2)\psi(z_3)\psi(z_4)\rangle$ follows readily from~\eqref{eqn::cov_psi} and~\eqref{eqn::three_aux7}. This proves~\eqref{eqn::cov_mixed_correlation_spin_energy} with $\ell=2$ and $k=1$.
\end{proof}

\begin{proof}[Proof of~\eqref{eqn::mixed_correlation_spin_energy} and~\eqref{eqn::cov_mixed_correlation_spin_energy} in Theorem~\ref{thm::energy_spin} with $\ell\geq 2$ and $k\geq 1$.]
The proof of the general case does not require new ideas and follows the same strategy as the proof of~\eqref{eqn::mixed_correlation_spin_energy} and~\eqref{eqn::cov_mixed_correlation_spin_energy} for the special case $\ell =2$ and $k=1$, presented above, but spelling out the details would require the introduction of very complicated notation. We omit the details for the sake of simplicity.
\end{proof}

\subsection{Correlations of $\phi$}
We focus for simplicity on the case $\ell=2$, that is, Eq.~\eqref{eqn::energy_two_point} for the two-point function $\langle \phi(z_1)\phi(z_2)\rangle$, as the proof of the general case~\eqref{eqn::energy_n_point} is essentially the same.

\begin{proof}[Proof of~\eqref{eqn::energy_two_point} in Theorem~\ref{thm::energy_spin}]
A simple calculation gives
	 \begin{align} \label{eqn::energy_two_point_aux1}
\begin{split}
			\langle \mathcal{E}_{z_1^a}\mathcal{E}_{z_2^a}\rangle^a=&\langle S_{z_1^{(a,-)}}S_{z_1^{(a,+)}}S_{z_2^{(a,-)}}S_{z_2^{(a,+)}}\rangle^a-\langle S_{z_1^{(a,-)}}S_{z_1^{(a,+)}}\rangle^a\langle S_{z_2^{(a,-)}}S_{z_2^{(a,+)}}\rangle^a\\
	=&\underbrace{\mathbb{P}^a\Big[z_1^{(a,-)}\xleftrightarrow{B}z_2^{(a,-)}\centernot{\xleftrightarrow{B}}z_1^{(a,+)}\xleftrightarrow{B}z_2^{(a,+)}\Big]}_{R_1^a}+\underbrace{\mathbb{P}^a\Big[z_1^{(a,-)}\xleftrightarrow{B}z_2^{(a,+)}\centernot{\xleftrightarrow{B}}z_1^{(a,+)}\xleftrightarrow{B}z_2^{(a,-)}\Big]}_{R_2^a}\\
	&+\underbrace{\mathbb{P}^a\Big[z_1^{(a,-)}\xleftrightarrow{B}z_1^{(a,+)},\; z_2^{(a,-)}\xleftrightarrow{B}z_2^{(a,+)}\Big]- \mathbb{P}^a\Big[z_1^{(a,-)}\xleftrightarrow{B}z_1^{(a,+)}\Big]\mathbb{P}^a\Big[z_2^{(a,-)}\xleftrightarrow{B}z_2^{(a,+)}\Big]}_{R_3^a}.
\end{split}
	\end{align}
We choose $\epsilon>0$ to be  the largest number such that$B_{2\epsilon}(z_1)\cap B_{2\epsilon}(z_2)=\emptyset$.

For the term $R_1^a$ (resp., $R_2^a$), note that the event $\{z_1^{(a,-)}\xleftrightarrow{B}z_2^{(a,-)}\centernot{\xleftrightarrow{B}}z_1^{(a,+)}\xleftrightarrow{B}z_2^{(a,+)}\}$ (resp., $\{z_1^{(a,-)}\xleftrightarrow{B}z_2^{(a,+)}\centernot{\xleftrightarrow{B}}z_1^{(a,+)}\xleftrightarrow{B}z_2^{(a,-)}\}$) implies the event $\mathcal{F}_{\epsilon}(z_1^a)\cap \mathcal{F}_{\epsilon}(z_2^a)$ for small enough $a>0$. 
We then conclude from~\eqref{eqn::four_arm_proba} and the independence between $\mathcal{F}_{\epsilon}(z_1^a)$ and $\mathcal{F}_{\epsilon}(z_2^a)$ that 
\begin{align} \label{eqn::energy_two_point_aux2}
	R_1^a \lesssim a^{\frac{5}{2}} \quad\text{and}\quad R_2^a\lesssim a^{\frac{5}{2}}. 
\end{align}

It remains to treat the term $R_3^a$. Let $M$ be the largest integer such that $2^M (2a)\leq \epsilon$; note that $M$ is of the order $-\log a$. For $1\leq i\leq 2$ and $1\leq m\leq M$, define $B_{m}^{(i,a)}:=\{w: |w-z_i^a|\leq 2^m(2a)\}$. 

On the one hand, for $\mathbb{P}^a\Big[z_1^{(a,-)}\xleftrightarrow{B}z_1^{(a,+)},\; z_2^{(a,-)}\xleftrightarrow{B}z_2^{(a,+)}\Big]$, as in~\eqref{eqn::four_spin_decom1}, we can write{\footnotesize
\begin{align*}
& \mathbb{P}^a\Big[z_1^{(a,-)}\xleftrightarrow{B}z_1^{(a,+)},\; z_2^{(a,-)}\xleftrightarrow{B}z_2^{(a,+)}\Big]\\
& = \mathbb{P}^a\Big[z_1^{(a,-)}\xleftrightarrow{B;B_1^{(1,a)}} z_1^{(a,+)},\; z_2^{(a,-)} \xleftrightarrow[\big(B_1^{(1,a)}\big)^c]{B}z_2^{(a,+)}\Big]+ \mathbb{P}^a \Big[z_1^{(a,-)}\xleftrightarrow{B;B_1^{(1,a)}}z_1^{(a,+)}\Big] \mathbb{P}^a\Big[z_2^{(a,-)}\xleftrightarrow{B;\big(B_1^{(1,a)}\big)^c}z_2^{(a,+)}\Big]\\
& \quad +\sum_{m=2}^M\Bigg( \mathbb{P}^a\Big[z_1^{(a,-)}\xleftrightarrow[B_{m-1}^{(1,a)}]{B;B_m^{(1,a)}} z_1^{(a,+)},\; z_2^{(a,-)} \xleftrightarrow[\big(B_m^{(1,a)}\big)^c]{B}z_2^{(a,+)}\Big]+ \mathbb{P}^a \Big[z_1^{(a,-)}\xleftrightarrow[B_{m-1}^{(1,a)}]{B;B_m^{(1,a)}}z_1^{(a,+)}\Big] \mathbb{P}^a\Big[z_2^{(a,-)}\xleftrightarrow{B;\big(B_m^{(1,a)}\big)^c}z_2^{(a,+)}\big]\Bigg)\\
& \quad + \mathbb{P}^a\Big[z_1^{(a,-)}\xleftrightarrow[B_{M}^{(1,a)}]{B;\big(B_M^{(2,a)}\big)^c} z_1^{(a,+)},\; z_2^{(a,-)} \xleftrightarrow[B_M^{(2,a)}]{B}z_2^{(a,+)}\Big]+ \mathbb{P}^a \Big[z_1^{(a,-)}\xleftrightarrow[B_{M}^{(1,a)}]{B;\big(B_M^{(2,a)}\big)^c} z_1^{(a,+)}\Big] \mathbb{P}^a\Big[z_2^{(a,-)}\xleftrightarrow{B;B_M^{(2,a)}}z_2^{(a,+)}\big]\\
& \quad +\sum_{m=1}^{M-1}\Bigg( \mathbb{P}^a\Big[ z_1^{(a,-)} \xleftrightarrow[\big(B_{m+1}^{(2,a)}\big)^c]{B;\big(B_m^{(2,a)}\big)^c}z_1^{(a,+)}, \; z_2^{(a,-)}\xleftrightarrow[B_m^{(2,a)}]{B} z_2^{(a,+)}\Big]+ \mathbb{P}^a\Big[ z_1^{(a,-)} \xleftrightarrow[\big(B_{m+1}^{(2,a)}\big)^c]{B;\big(B_m^{(2,a)}\big)^c}z_1^{(a,+)}\Big]\mathbb{P}^a\Big[ z_2^{(a,-)}\xleftrightarrow{B;B_m^{(2,a)}} z_2^{(a,+)}\Big]\Bigg)\\
& \quad + \mathbb{P}^a\Big[ z_1^{(a,-)} \xleftrightarrow[\big(B_{1}^{(2,a)}\big)^c]{B}z_1^{(a,+)}, \; z_2^{(a,-)}\xleftrightarrow{B} z_2^{(a,+)}\Big].
\end{align*}}
On the other hand,  we can decompose $\mathbb{P}^a\Big[z_1^{(a,-)}\xleftrightarrow{B}z_1^{(a,+)}\Big]\mathbb{P}^a\Big[ z_2^{(a,-)}\xleftrightarrow{B}z_2^{(a,+)}\Big]$ in a similar way. Consequently, we have
{\footnotesize\begin{align} \label{eqn::energy_two_point_aux3}
\begin{split}
			&	\mathbb{P}^a\Big[z_1^{(a,-)}\xleftrightarrow{B}z_1^{(a,+)},\; z_2^{(a,-)}\xleftrightarrow{B}z_2^{(a,+)}\Big]- \mathbb{P}^a\Big[z_1^{(a,-)}\xleftrightarrow{B}z_1^{(a,+)}\Big]\mathbb{P}^a\Big[ z_2^{(a,-)}\xleftrightarrow{B}z_2^{(a,+)}\Big]\\
	=&	 \underbrace{\mathbb{P}^a\Big[z_1^{(a,-)}\xleftrightarrow{B;B_1^{(1,a)}} z_1^{(a,+)},\; z_2^{(a,-)} \xleftrightarrow[\big(B_1^{(1,a)}\big)^c]{B}z_2^{(a,+)}\Big]- \mathbb{P}^a\Big[z_1^{(a,-)}\xleftrightarrow{B;B_1^{(1,a)}} z_1^{(a,+)}\Big] \mathbb{P}^a\Big[z_2^{(a,-)} \xleftrightarrow[\big(B_1^{(1,a)}\big)^c]{B}z_2^{(a,+)}\Big]}_{T_1^{(1,a)}}\\
	&+\sum_{m=2}^M\Bigg( \underbrace{\mathbb{P}^a\Big[z_1^{(a,-)}\xleftrightarrow[B_{m-1}^{(1,a)}]{B;B_m^{(1,a)}} z_1^{(a,+)},\; z_2^{(a,-)} \xleftrightarrow[\big(B_m^{(1,a)}\big)^c]{B}z_2^{(a,+)}\Big]-\mathbb{P}^a\Big[z_1^{(a,-)}\xleftrightarrow[B_{m-1}^{(1,a)}]{B;B_m^{(1,a)}} z_1^{(a,+)}\Big]\mathbb{P}^a\Big[ z_2^{(a,-)} \xleftrightarrow[\big(B_m^{(1,a)}\big)^c]{B}z_2^{(a,+)}\Big]}_{T_m^{(1,a)}}\Bigg)\\
	&+ \underbrace{\mathbb{P}^a\Big[z_1^{(a,-)}\xleftrightarrow[B_{M}^{(1,a)}]{B;\big(B_M^{(2,a)}\big)^c} z_1^{(a,+)},\; z_2^{(a,-)} \xleftrightarrow[B_M^{(2,a)}]{B}z_2^{(a,+)}\Big]- \mathbb{P}^a\Big[z_1^{(a,-)}\xleftrightarrow[B_{M}^{(1,a)}]{B;\big(B_M^{(2,a)}\big)^c} z_1^{(a,+)}\Big]\mathbb{P}^a\Big[ z_2^{(a,-)} \xleftrightarrow[B_M^{(2,a)}]{B}z_2^{(a,+)}\Big]}_{T_{M+1}^a}\\
	&+\sum_{m=1}^{M-1}\Bigg( \underbrace{\mathbb{P}^a\Big[ z_1^{(a,-)} \xleftrightarrow[\big(B_{m+1}^{(2,a)}\big)^c]{B;\big(B_m^{(2,a)}\big)^c}z_1^{(a,+)}, \; z_2^{(a,-)}\xleftrightarrow[B_m^{(2,a)}]{B} z_2^{(a,+)}\Big]-\mathbb{P}^a\Big[ z_1^{(a,-)} \xleftrightarrow[\big(B_{m+1}^{(2,a)}\big)^c]{B;\big(B_m^{(2,a)}\big)^c}z_1^{(a,+)}\Big]\mathbb{P}^a\Big[ z_2^{(a,-)}\xleftrightarrow[B_m^{(2,a)}]{B} z_2^{(a,+)}\Big]}_{T_{m+1}^{(2,a)}}\Bigg)\\
	&+\underbrace{\mathbb{P}^a\Big[ z_1^{(a,-)} \xleftrightarrow[\big(B_{1}^{(2,a)}\big)^c]{B}z_1^{(a,+)}, \; z_2^{(a,-)}\xleftrightarrow{B} z_2^{(a,+)}\Big]-\mathbb{P}^a\Big[ z_1^{(a,-)} \xleftrightarrow[\big(B_{1}^{(2,a)}\big)^c]{B}z_1^{(a,+)}\Big]\mathbb{P}^a\Big[z_2^{(a,-)}\xleftrightarrow{B} z_2^{(a,+)}\Big]}_{T_1^{(2,a)}}.
\end{split}
\end{align}}See Figure~\ref{fig::log_geometry_energy} for an illustration of the event $\{ z_1^{(a,-)} \xleftrightarrow[B_{m-1}^{(1,a)}]{B;B_m^{(1,a)}} z_1^{(a,+)},\; z_2^{(a,-)}\xleftrightarrow[\big(B_m^{(1,a)}\big)^c]{B} z_2^{(a,+)} \}$.

  \begin{figure}
	\includegraphics[width= 0.7\textwidth]{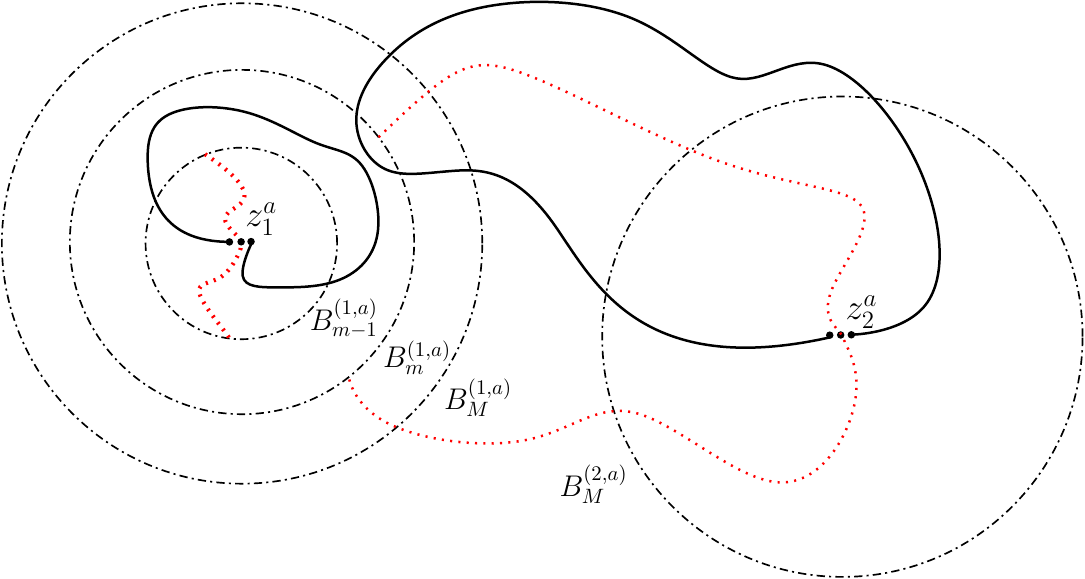}
	\caption{The event $\{ z_1^{(a,-)} \xleftrightarrow[B_{m-1}^{(1,a)}]{B;B_m^{(1,a)}} z_1^{(a,+)},\; z_2^{(a,-)}\xleftrightarrow[\big(B_m^{(1,a)}\big)^c]{B} z_2^{(a,+)} \}$. The black, solid lines denote black paths, while the red, dotted lines denote white paths. $z_1^{(a,-)}$ and $z_1^{(a,+)}$ are contained in $B_{m-1}^{(1,a)}$. They are not connected by a black path within the disk $B_{m-1}^{(1,a)}$, but are connected within the larger disk $B_{m}^{(1,a)}$, with radius twice that of $B_{m-1}^{(1,a)}$. The points $z_2^{(a,-)}$ and $z_2^{(a,+)}$ are connected by a black path, but not outside $B_m^{(1,a)}$. The number, $2M$, of disks one can insert between the two groups of points $\{z_1^{(a,-)},z_1^{(a,+)}\}$ and $\{z_2^{(a,-)},z_2^{(a,+)}\}$ is of order $-\log{a}$.}
	\label{fig::log_geometry_energy}
\end{figure}

Using~\eqref{eqn::four_arm_proba} and the quasi-multiplicativity of four-arm probabilities (see e.g.,~\cite[Proposition~12]{NolinNearCriticalPercolation}), we see that 
\begin{align} \label{eqn::energy_two_point_aux4}
	T_1^{(1,a)} \lesssim a^{\frac{5}{2}},\quad T_1^{(2,a)}\lesssim a^{\frac{5}{2}},\quad T_{M+1}^{a}\lesssim a^{\frac{5}{2}}. 
\end{align}
One can proceed as in the analysis of $\sum_{m=2}^M T_m^a$ in the proof of~\eqref{eqn::three_point_energy_density} to show that 
\begin{align} \label{eqn::energy_two_point_axu5}
	 \lim_{a\to 0} a^{-\frac{5}{2}}|\log a|^{-1}\times \sum_{m=2}^M \Big(T_m^{(1,a)}+T_m^{(2,a)}\Big)  \in (0,\infty). 
\end{align}

Combining~\eqref{eqn::energy_two_point_aux1}-\eqref{eqn::energy_two_point_axu5} gives 
\begin{align*}
	\lim_{a\to 0} \big(a^{\frac{5}{2}}|\log a|\big)^{-2} \, \langle \mathcal{E}_{z_1^a}\mathcal{E}_{z_2^a}\rangle^a=0,
\end{align*}
which completes the proof. 
\end{proof}

\begin{proof}[Proof of~\eqref{eqn::energy_n_point} in Theorem~\ref{thm::energy_spin} with $\ell \geq 2$]
One can proceed essentially as in the proof of~\eqref{eqn::energy_two_point} presented above, but the notation would become much more complicated, so we omit the details for the sake of simplicity. One could also adopt the strategy used in the proof of~\eqref{eqn::mixed_correlation_spin_energy} with $\ell=2$ and $k=1$, presented above, to deal with multiple summations over annuli centered at $z_1^a,\ldots, z_{\ell-1}^a$, respectively.
\end{proof}


\paragraph*{Acknowledgments.}
F.~C.'s research is supported by NYU Abu Dhabi via a personal research grant. Y.~F. thanks NYUAD for its hospitality during two visits in the fall of 2023 and of 2024. The first visit was partially supported by the Short-Term Visiting Fund for Doctoral Students of Tsinghua University.
Both authors profited from their participation in the Trimester Program ``Probabilistic methods in quantum field theory" held at the Hausdorff Research Institute for Mathematics (HIM), Bonn, Germany, where this paper was completed. They thank the organizers of the program for their invitation and the Hausdorff Institute for its hospitality.
The program was funded by the Deutsche Forschungsgemeinschaft (DFG, German Research Foundation) under Germany's Excellence Strategy - EXC-2047/1 - 390685813.

\paragraph*{Data availability.}
No datasets were generated or analyzed during the current study.

\paragraph*{Conflict of interest.}
The authors have no competing interests to declare that are relevant to the content of this article.

\end{document}